%% file: Main_report.tex
\documentclass[smallextended,final,numbook]{svjour3}     

\newif\iftechreport 
\techreporttrue     

\usepackage[colorlinks=true, linkcolor=blue, citecolor=blue, urlcolor=blue, pdfborder={0 0 0}]{hyperref}
\usepackage[letterpaper,top=2in, bottom=1.5in, left=1in, right=1in]{geometry}
\usepackage{graphicx}

\graphicspath{ {./figures/} }
\usepackage{amssymb}
\usepackage{tikz}
\input{tikz_setting}

\usepackage[%
    font={small,sf},
    labelfont=bf,
    format=hang,
    format=plain,
    margin=0pt,
    width=0.8\textwidth,
]{caption}
\usepackage[list=true]{subcaption}
\usepackage{xcolor,colortbl}
\input{macros} 
\input{davismacros}

\def\cut#1{{}}
\smartqed

\title{A Three-Operator Splitting Scheme and its Optimization Applications\thanks{This work is supported in part by NSF grant DMS-1317602.
}}

\author{Damek Davis   \and
        Wotao Yin}

\institute{D. Davis \and W. Yin\at
              Department of Mathematics, University of California, Los Angeles\\
              Los Angeles, CA 90025, USA\\
              \email{damek / wotaoyin@ucla.edu}}

\date{\today}
\journalname{Report} 

\begin{document}

\maketitle
\abstract{
\input{abstract}
}






\input{section_intro}
\input{section_motivation}

\input{section_convergencetheory}
\input{section_rates}

\input{section_numerical}
\input{conclusion}

\bibliographystyle{spmpsci}
\bibliography{bibliography,admm,applications}

\appendix
\input{section_appendix}
\end{document}

%% file: tikz_setting.tex
\usetikzlibrary{fit, arrows, shapes, positioning, shadows, matrix, calc}

\tikzstyle fwd=[line width=1pt, ->]   
\tikzstyle fwddash=[line width=1pt, dashed, ->]   
\tikzstyle bwd=[line width=1pt, ->]  
\tikzstyle refl=[line width=1pt, ->]    

\tikzset{
  >=stealth', 
  invisible/.style={opacity=0}, 
  alt/.code args={<#1>#2#3}{\alt<#1>{\pgfkeysalso{#2}}{\pgfkeysalso{#3}}}, 
  visible on/.style={alt=#1{}{invisible}}, 
  smallnode/.style={circle, fill=black, thick, inner sep=1pt, minimum size=1.5pt}, 
  punkt/.style={
           rectangle,
           rounded corners,
           draw=black, very thick,
           text width=5em,
           minimum height=2em,
           text centered},
}

%% file: macros.tex
\usepackage{mathtools}
\usepackage[normalem]{ulem} 

\newtheorem{algo}{Algorithm}

\newcommand{\remove}[1]{{}}

\newcommand{\vx}{{\mathbf{x}}}
\newcommand{\vy}{{\mathbf{y}}}

\newcommand{\vN}{{\mathbf{N}}}

\newcommand{\vR}{{\mathbf{R}}}

\newcommand{\cA}{{\mathcal{A}}}

\newcommand{\cC}{{\mathcal{C}}}

\newcommand{\cH}{{\mathcal{H}}}

\newcommand{\cK}{{\mathcal{K}}}

\newcommand{\cS}{{\mathcal{S}}}

\newcommand{\St}{{\mathrm{subject~to}}} 
\newcommand{\dom}{{\mathrm{dom}}} 

\newcommand{\prox}{\mathbf{prox}}

\newcommand{\tnabla}{\widetilde{\nabla}}
\newcommand{\TFDRS}{T_{\mathrm{FDRS}}}
\newcommand{\TDRS}{T_{\mathrm{DRS}}}

\DeclareMathOperator*{\argmin}{arg\,min}

\DeclareMathOperator*{\Min}{minimize}

\DeclareMathOperator*{\Fix}{Fix}
\DeclareMathOperator*{\zer}{zer}

\DeclareMathOperator*{\gra}{gra}

\newcommand{\bc}{\begin{center}}
\newcommand{\ec}{\end{center}}

\newcommand{\bdm}{\begin{displaymath}}
\newcommand{\edm}{\end{displaymath}}

\newcommand{\beq}{\begin{equation}}
\newcommand{\eeq}{\end{equation}}

\newcommand{\bfl}{\begin{flushleft}}
\newcommand{\efl}{\end{flushleft}}

\newcommand{\bt}{\begin{tabbing}}
\newcommand{\et}{\end{tabbing}}

\newcommand{\beqn}{\begin{align}}
\newcommand{\eeqn}{\end{align}}

\newcommand{\beqs}{\begin{align*}} 
\newcommand{\eeqs}{\end{align*}}  



%% file: davismacros.tex
\usepackage[lined,boxed,commentsnumbered, ruled,vlined]{algorithm2e}

\newcommand\numberthis{\addtocounter{equation}{1}\tag{\theequation}}

\DeclarePairedDelimiter{\dotp}{\langle}{\rangle}

\newcommand{\refl}{\mathbf{refl}}
\newcommand{\TFBS}{T_{\mathrm{FBS}}}
\usepackage{booktabs}

%% file: abstract.tex
Operator splitting schemes have been successfully used in computational sciences to reduce complex problems into a series of simpler subproblems. Since 1950s, these schemes have been widely used to solve problems in PDE and control. Recently, large-scale optimization problems in machine learning, signal processing, and imaging have created a resurgence of interest in operator-splitting based algorithms because they often have simple descriptions, are easy to code, and have (nearly) state-of-the-art performance for large-scale optimization problems. Although operator splitting techniques were introduced over 60 years ago, their importance has significantly increased in the past decade.

This paper introduces a new operator-splitting scheme for solving a variety of problems that are reduced to a monotone inclusion of three operators, one of which is cocoercive. Our scheme is very simple, and it does not reduce to any existing splitting schemes. Our scheme recovers the existing forward-backward, Douglas-Rachford, and forward-Douglas-Rachford splitting schemes as special cases.

Our new splitting scheme leads to a set of new and simple algorithms for a variety of other problems, including the 3-set split feasibility problems, 3-objective minimization problems, and doubly and multiple regularization problems, as well as the simplest extension of the classic ADMM from 2 to 3 blocks of variables. In addition to the basic scheme, we introduce several modifications and enhancements that can improve the convergence rate in practice, including an acceleration that achieves the optimal rate of convergence for strongly monotone inclusions. Finally, we evaluate the algorithm on several applications. 

%% file: section_intro.tex
\section{Introduction}
Operator splitting schemes reduce complex problems built from simple pieces into a series smaller subproblems which can be solved sequentially or in parallel. Since the 1950s they have been successfully applied to problems in PDE and control, but recent large-scale applications in machine learning, signal processing, and imaging have created a resurgence of interest in operator-splitting based algorithms. These algorithms often have very simple descriptions, are straightforward to implement on computers, and have (nearly) state-of-the-art performance for large-scale optimization problems. Although operator splitting techniques were introduced over 60 years ago, their importance has significantly increased in the past decade.

This paper introduces a new operator-splitting scheme, which solves nonsmooth optimization problems of many different forms, as well as monotone inclusions. In an abstract form, this  new splitting scheme will
\begin{align}
\text{find}\quad x\in \cH \quad\text{such that} \quad 0 \in Ax + Bx + Cx \label{eq:mainprob}
\end{align}
for three maximal monotone operators $A, B, C$ defined on a Hilbert space $\cH$, where \textbf{the operator $C$ is cocoercive.}\footnote{An operator $C$ is $\beta$-cocoercive (or $\beta$-inverse-strongly monotone), $\beta >0$, if $\dotp{Cx-Cy,x-y}\ge\beta \|Cx-Cy\|^2,~\forall x,y\in\cH$. This property generalizes many others. In particular, $\nabla h$ of an $L$-Lipschitz differentiable convex function $h$ is $1/L$-cocoercive. }

The most straightforward example of \eqref{eq:mainprob} arises from the optimization problem
\begin{align}
\Min\, f(x)+g(x)+h(x), \label{eq:mainprob1}
\end{align}
where $f$, $g$, and $h$ are proper, closed, and convex functions and $h$ is \textbf{Lipschitz differentiable}. The first-order optimality condition of \eqref{eq:mainprob1} reduces to \eqref{eq:mainprob} with $Ax=\partial f(x)$, $Bx = \partial g(x)$, and $Cx = \nabla h(x)$, where $\partial f,\partial g$ are subdifferentials of $f$ and $g$, respectively. Note that $C$ is cocoercive because $h$ is Lipschitz differentiable.

A number of other examples of \eqref{eq:mainprob} can be found in Section \ref{sc:motive} including split feasibility, doubly regularized, and monotropic programming problems, which have surprisingly many applications.

To introduce our splitting scheme, let $I_\cH$ denote the identify map in $\cH$ and $J_S:=(I+S)^{-1}$ denote the resolvent of a monotone operator $S$. (When  $S=\partial f$, $J_S(x)$ reduces to the proximal map: $\argmin_y f(y)+\frac{1}{2}\|x-y\|^2$.) Let $\gamma\in (0,2\beta)$ be a scalar. Our splitting scheme for solving \eqref{eq:mainprob} is summarized by the operator
\begin{align}\label{eq:newoperator}
 \boxed{T:=I_{\cH} - J_{\gamma B} + J_{\gamma A}\circ(2J_{\gamma B} - I_{\cH} - \gamma C\circ J_{\gamma B}).}
 \end{align}
Calculating $Tx$ requires evaluating $J_{\gamma A}$, $J_{\gamma B}$, and $C$ only once each, though  $J_{\gamma B}$ appears three times in $T$. In addition, we will show that a fixed-point of $T$ encodes a solution to \eqref{eq:mainprob} and $T$ is an averaged operator.

The problem \eqref{eq:mainprob} can be solved by iterating
\begin{align}\label{eq:zitr}
z^{k+1} := (1-\lambda_k) z^k + \lambda_k Tz^k,
\end{align}
where $z^0$ is an arbitrary point and  $\lambda_k\in (0, (4\beta - \gamma)/2\beta)$ is a  relaxation parameter. (For simplicity, one can fix $\gamma < 2\beta$ and  $\lambda_k\equiv 1$.)
This iteration can be implemented as follows:
\begin{algo}\label{alg:basic} Set an arbitrary point $z^0\in \cH$, stepsize $\gamma \in (0, 2\beta)$, and relaxation sequence $(\lambda_j)_{j \geq 0} \in (0, (4\beta - \gamma)/2\beta)$. For $k=0,1,\ldots,$ iterate:
\begin{enumerate}
\item  get $x_B^k = J_{\gamma B}(z^k)$;
\item  get $x_A^k = J_{\gamma A}(2x_B^k - z^k- \gamma Cx_B^k)$;\qquad //comment: $ x_A^k = J_{\gamma A}\circ(2J_{\gamma B} - I_{\cH} - \gamma C\circ J_{\gamma B})z^k $
\item  get $z^{k+1}  = z^k+\lambda_k(x_A^k - x_B^k)$; \hspace{30pt} //comment: $z^{k+1} = (1-\lambda_k) z^k + \lambda_k Tz^k$
\end{enumerate}
\end{algo}
Algorithm \ref{alg:basic} leads to new algorithms for a large number of applications, which are given in Section \ref{sc:motive} below. Although some of those applications can be solved by other splitting methods, for example, by the alternating directions method of multipliers (ADMM), our new algorithms are typically simpler, use fewer or no additional variables, and take advantage of the differentiability of smooth terms in the objective function. The dual form of our algorithm is the simplest extension of ADMM from the classic two-block form to the three-block form that has a general convergence result.  The details of these are given in Section~\ref{sc:motive}.

The full convergence result for Algorithm~\ref{alg:basic} is stated in Theorem~\ref{thm:convergence}. For brevity we include the following simpler version here:
\begin{theorem}[Convergence of Algorithm~\ref{alg:basic}]
Suppose that $\Fix T\not=\emptyset$. Let $\alpha = 2\beta/(4\beta - \gamma)$ and suppose that $(\lambda_j)_{j \geq 0}$ satisfies $\sum_{j=0}^\infty (1-\lambda_j/\alpha)\lambda_j/\alpha = \infty$ (which is true if the sequence is strictly bounded away from $0$ and $1/\alpha$). Then the sequences $(z^j)_{j \geq 0}$, $(x_B^j)_{j \geq 0}$, and $(x_A^j)_{j \geq 0}$ generated by Algorithm~\ref{alg:basic} satisfy the following:
\begin{enumerate}
\item $(z^j)_{j \geq 0}$ converges weakly to a fixed point of $T$; and
\item $(x_B^j)_{j \geq 0}$ and $(x_A^j)_{j \geq 0}$ converge weakly to an element of $\zer(A + B+ C)$.
\end{enumerate}
\end{theorem}

\subsection{Existing two--operator splitting schemes}
A large variety of recent algorithms  \cite{chambolle2011first,esser2010general,pock2009algorithm} and their  generalizations and enhancements  \cite{bo2014convergence,bot2013algorithm,bo?2013douglas,briceno2011monotone+,combettes2013systems,combettes2014forward,combettes2012primal,condat2013primal,komodakis2014playing,vu2013splitting} are (skillful) applications of
one of the following three operator-splitting schemes: (i) forward-backward-forward splitting (FBFS)~\cite{tseng2000modified},  (ii) forward-backward splitting (FBS)~\cite{passty1979ergodic}, and (iii) Douglas-Rachford splitting (DRS)~\cite{lions1979splitting}, which all split the sum of \emph{two} operators. (The recently introduced forward-Douglas-Rachford splitting (FDRS) turns out to be  a special case of FBS applied to a suitable monotone inclusion~\cite[Section 7]{davis2014convergenceFDRS}.) Until now, these algorithms are the only basic operator-splitting schemes for monotone inclusions, if we ignore variants involving inertial dynamics, special metrics, Bregman divergences, or different stepsize choices\footnote{For example, Peaceman-Rachford splitting (PRS) \cite{lions1979splitting} doubles the step size in DRS.}. To our knowledge, no new splitting schemes have been proposed since the introduction of FBFS in 2000.

The proposed splitting scheme $T$  in Equation~\eqref{eq:newoperator} is the first algorithm to split the sum of \emph{three} operators that does not appear to reduce to any of the existing schemes.  In fact,  FBS, DRS, and FDRS  are special cases of Algorithm~\ref{alg:basic}.

\begin{proposition}[Existing operator splitting schemes as special cases]
\begin{enumerate}
\item {Consider the forward-backward splitting (FBS) operator \cite{passty1979ergodic}}, $\TFBS := J_{\gamma A} \circ(I_{\cH} - \gamma C)$, for solving $0\in Ax+Cx$ where $A$ is maximal monotone and $C$ is cocoercive.
If we set $B = 0$ in  \eqref{eq:newoperator}, then $T = \TFBS$.\\[-6pt]
\item {Consider the Douglas-Rachford splitting (DRS) operator \cite{lions1979splitting}}, $\TDRS:=I_{\cH}  - J_{\gamma B}+ J_{\gamma A} \circ (2 J_{\gamma B} - I_{\cH})$, for solving $0\in Ax+Bx$ where $A,B$ are maximal monotone.  If wet set $C = 0$ in \eqref{eq:newoperator}, then $T = \TDRS$.\\[-6pt]
\item {Consider the forward-Douglas-Rachford splitting (FDRS) operator \cite{briceno2012forward}}, $\TFDRS:= I_{\cH} - P_V + J_{\gamma A}\circ(2P_V- I_{\cH} - \gamma P_V\circ C' \circ P_V)$, for solving $0\in Ax+C'x+N_Vx$ where  $A$ is maximal monotone, $C'$ is cocoercive, $V$ is a closed vector space,  $N_V$ is the normal cone operator of $V$, and $P_V$ denote the projection to $V$. If  we set $B = N_V$  and $C = P_V \circ C' \circ P_V$ in \eqref{eq:newoperator}, then $T = \TFDRS$.
\end{enumerate}
\end{proposition}


The operator $T$ is also related to the Peaceman-Rachford splitting (PRS) operator \cite{lions1979splitting}. Let us introduce the ``reflection'' operator $\refl_{A} := 2J_A-I_{\cH}$ where $A:\cH\to\cH$ is a maximal monotone operator, and set
\begin{align}\label{eq:reflectionversionofmap}
S:=2T-I=\refl_{\gamma A}\circ\left( \refl_{\gamma B} - \gamma C\circ J_{\gamma B}\right) - \gamma C \circ J_{\gamma B}.
\end{align}
If we set $C=0$,  then $S$ reduces to the PRS operator.

\subsection{Convergence rate guarantees}\label{sec:theoreticalresults}
We show in Lemma~\ref{lem:fixedpoints} that from any fixed point $z^\ast $ of the operator $T$, we obtain $x^\ast := J_{\gamma B}(z^\ast) $ as a zero of the monotone inclusion~\eqref{eq:mainprob}, i.e., $x^*\in \zer(A + B+ C)$. 
In addition, under various scenarios,  the following {convergence rates} can be deduced:
\begin{enumerate}
\item {\bf Fixed-point residual (FPR) rate}: The FPR $\|Tz^k - z^k\|^2$ has the sharp rate $o\left(1/\sqrt{k+1}\right)$. (Part~\ref{thm:convergence:part:convergencerate} of Theorem~\ref{thm:convergence} and Remark \ref{rem:FPRslow}.)
\item {\bf Function value rate}: Under mild conditions on Problem~\eqref{eq:mainprob1}, although $(f + g+ h)(x^k) - (f + g+ h)(x^\ast)$ is not monotonic, it is bounded by $o\left(1/\sqrt{k+1}\right)$. Two averaging procedures improve this rate to $O\left(1/(k+1)\right)$. The running best sequence, $\min_{i = 0, \cdots, k}(f + g+ h)(x^i)-(f + g+ h)(x^\ast)$, further improves to $o\left(1/(k+1)\right)$  whenever $f$ is differentiable and $\nabla f$ is Lipschitz continuous.
These rates are also sharp. 

\item {\bf Strong convergence:} When $A$ (respectively $B$ or $C$) is strongly monotone, the sequence $\|x_A^k - x^\ast\|^2$  (respectively $\|x_B^k - x^\ast\|^2$) converges with rate $o(1/\sqrt{k+1})$. The running best and averaged sequences improve this rate to $o(1/(k+1))$ and $O(1/(k+1))$, respectively. 

\item {\bf Linear convergence:} We reserve $\mu\in[0,\infty)$  for strong monotonicity constants and $L\in (0,\infty]$ for Lipschitz constants. If strong monotonicity does not hold, then $\mu=0$. If Lipschitz continuity does not hold, then  $L=\infty$.  Algorithm~\ref{alg:basic} converges linearly whenever $(\mu_A + \mu_B + \mu_C)(1/L_A + 1/L_B) > 0$, i.e., whenever  at least one of $A,$ $B$, or $C$ is strongly monotone and  at least one of  $A$ or $B$ is Lipschitz continuous. 
We present a counterexample where $A$ and $B$ are not Lipschitz continuous and Algorithm~\ref{alg:basic} fails to converge linearly.  
\item {\bf Variational inequality convergence rate}: We can apply Algorithm \ref{alg:basic} to  primal-dual optimality conditions and other structured monotone inclusions with $A = \overline{A} + \partial f$, $B = \overline{B} + \partial g$ and $C = \overline{C} + \nabla h$ for some monotone operators $\overline{A}$, $\overline{B}$, and $\overline{C}$. A typical example is when $\overline{A}$ and $\overline{B}$ are bounded skew linear maps and $\overline{C} = 0$. Then, the corresponding variational inequality converges with rate $o\left({1}/{\sqrt{k+1}}\right)$ under mild conditions on $\overline{A}$ and $f$. Again, averaging can improve the rate to $O(1/(k+1))$. 

\end{enumerate}

%

\input{section_modification}

\subsection{Practical implementation issues: Line search}

Recall that $\beta$, the cocoercivity constant of $C$, determines the stepsize condition $\gamma\in(0,2\beta)$ for Algorithm~\ref{alg:basic}. When $\beta$ is unknown, one can find $\gamma$ by trial and error. Whenever the FPR is observed to increase (which does not happen if $\gamma\in(0,2\beta)$ by Part~\ref{thm:convergence:part:FPRmonotone} of Theorem~\ref{thm:convergence}), reduce $\gamma$ and restart the algorithm from the initial or last iterate.

For the case of $C = \nabla h$ for some convex function $h$ with Lipschitz $\nabla h$, we propose a line search procedure that uses a fixed stepsize $\gamma$ but involves an  auxiliary 
factor $\rho \in (0, 1]$. It works better than the above approach of  changing $\gamma$ since the latter changes fixed point. 
Let
\begin{align*}
\refl_{\gamma B}^\rho := (1+\rho) J_{\gamma B} - \rho I_{\cH}.
\end{align*}
Note that $\refl_{\gamma B}^1 = \refl_{\gamma B}$ and $\refl_{\gamma B}^0 = J_{\gamma B}$.  Define
\begin{align*}
T_{\gamma}^\rho &= I_{\cH} - J_{\gamma B} + J_{\rho\gamma A}\circ(\refl_{\gamma B}^\rho -  \rho \gamma \nabla h\circ J_{\gamma B}).
\end{align*}

Our line search procedure iterates $z^{k+1}=T_{\gamma}^\rho (z^k)$ with a special choice of $\rho$:
\begin{algo}[Algorithm \ref{alg:basic} with line search]\label{alg:linesearch} Choose $z^0\in \cH$ and $\gamma \in (0, \infty)$. For $k=0,1,\ldots$, iterate
\begin{enumerate}
\item get $x_B^k = J_{\gamma B}(z^k)$;
\item get $\rho \in (0, 1]$ such that
$$h(x_A^k) \leq h(x_B^k) + \dotp{x_A^k - x_B^k, \nabla h(x_B^k)} + \frac{1}{2\gamma \rho} \|x_A^k - x_B^k\|^2$$
where
$$x_A^k = J_{\gamma\rho A}(x_B^k + \rho(x_B^k - z^k) - \gamma \rho \nabla h(x_B^k));$$
\item get $z^{k+1} = z^{k} + x_A^k - x_B^k$.
\end{enumerate}
\end{algo}

A straightforward calculation shows the following lemma:
\begin{lemma}
For all $\rho \in (0, 1)$ and all $\gamma > 0$, we have
\begin{align*}
\zer(A + B + \nabla h) = J_{\gamma B}(\Fix T_{\gamma}^\rho)\quad\text{and}\quad \Fix(T_{\gamma}^\rho) = \Fix(T_{\gamma}^1).
\end{align*}
\end{lemma}

\begin{remark}
In practice, Algorithm~\ref{alg:linesearch}, which can start with a larger $\gamma$, can be an order of magnitude faster than Algorithm~\ref{alg:basic}. Unfortunately, we have no proof of convergence for this method.
\end{remark}

\subsection{Definitions, notation and some facts}\label{sec:notation}
In what follows, $\cH$ denotes a (possibly infinite dimensional) Hilbert space. We  use $\dotp{~,~}$ to denote the inner product associated to a Hilbert space. 
In all of the algorithms we consider, we utilize two stepsize sequences: the implicit sequence $(\gamma_j)_{j \geq 0} \subseteq \vR_{++}$ and the explicit sequence $(\lambda_j)_{j \geq 0} \subseteq \vR_{++}$.

The following definitions and facts are mostly standard and can be found in \cite{bauschke2011convex}.


Let $L \geq 0$, and let $D$ be a nonempty subset of $\cH$.  A map $T : D \rightarrow \cH$ is called $L$-Lipschitz if for all $x, y \in \cH$, we have $\|Tx - Ty\| \leq L\|x-y\|$. In particular, $N$ is called \emph{nonexpansive} if it is $1$-Lipschitz. A map $N : D \rightarrow \cH$ is called $\lambda$-averaged \cite[Section 4.4]{bauschke2011convex} if it can be written as
\begin{align}\label{eq:averagednotation}
N = T_{\lambda}:= (1-\lambda) I_{\cH} + \lambda T
\end{align}
for a nonexpansive map $T : D \rightarrow \cH$ and a real number $\lambda \in (0, 1)$. A $(1/2)$-averaged map is called \emph{firmly nonexpansive}.  We  use a $\ast$ superscript to denote a fixed point of a nonexpansive map, e.g., $z^\ast$.

Let $2^\cH$ denote the power set of $\cH$.  A set-valued operator $A : \cH \rightarrow 2^\cH$ is called \emph{monotone} if for all $x, y \in \cH$, $u \in Ax$, and $v \in Ay$, we have $\dotp{ x- y, u - v} \geq 0$. We denote the set of zeros of a monotone operator by $\zer(A) := \{x \in \cH \mid 0 \in Ax\}.$ The \emph{graph} of $A$ is denoted by $\gra(A) := \{(x, y) \mid x\in \cH, y \in Ax\}$. Evidently, $A$ is uniquely determined by its graph. A monotone operator $A$ is called \emph{maximal monotone} provided that $\gra(A)$ is not properly contained in the graph of any other monotone set-valued operator.  The \emph{inverse} of $A$, denoted by $A^{-1}$, is defined uniquely by its graph $\gra(A^{-1}) := \{(y, x) \mid  x\in \cH, y \in Ax\}$. Let $\beta \in \vR$ be a positive real number. The operator $A$ is called \emph{$\beta$-strongly monotone} provided that for all $x, y \in \cH$,  $u \in Ax$, and $v \in Ay$, we have $\dotp{x - y, u - v} \geq \beta \|x-y\|^2$. A \emph{single-valued} operator $B : \cH \rightarrow 2^\cH$  maps each point in $\cH$ to a singleton and will be identified with the natural $\cH$-valued map it defines.  
The \emph{resolvent} of a monotone operator $A$ is defined by the inversion $J_A := (I + A)^{-1}$.  Minty's theorem shows that  $J_A$ is single-valued and has full domain $\cH$ if, and only if,  $A$ is maximally monotone.  Note that $A$ is monotone if, and only if, $J_A$ is firmly nonexpansive.  Thus, the \emph{reflection operator}
\begin{align}\label{eq:refl}
\refl_{A} := 2J_A - I_{\cH}
\end{align}
 is nonexpansive on $\cH$ whenever $A$ is maximally monotone. 

Let $f : \cH \rightarrow (-\infty, \infty]$ denote a closed (i.e., lower semi-continuous), proper, and convex function.  Let $\dom(f) := \{x \in \cH \mid f(x) < \infty \}$. We  let $\partial f(x) : \cH \rightarrow 2^\cH$ denote the subdifferential of $f$: $\partial f(x) := \{ u\in \cH \mid \forall y \in \cH, f(y) \geq f(x) +  \dotp{y-x, u}\}$. We  always let $$\tnabla f(x) \in \partial f(x)$$ denote a subgradient of $f$ drawn at the point $x$.  The subdifferential operator of $f$ is maximally monotone. The inverse of $\partial f$ is given by $\partial f^\ast$ where $f^\ast(y) := \sup_{x \in \cH} \dotp{y, x} - f(x)$ is the \emph{Fenchel conjugate} of $f$. If the function $f$ is $\beta$-strongly convex, then $\partial f$ is $\beta$-strongly monotone and $\partial f^\ast$ is single-valued and $\beta$-cocoercive.

 If a convex function $f : \cH \rightarrow (-\infty, \infty]$ is Fr{\'e}chet differentiable at $x \in \cH$, then $\partial f(x) = \{\nabla f(x)\}$. Suppose $f$ is convex and Fr{\'e}chet differentiable on $\cH$, and let $\beta \in \vR$ be a positive real number. Then the Baillon-Haddad theorem states that $\nabla f$ is $(1/\beta)$-Lipschitz if, and only if, $\nabla f$ is $\beta$-cocoercive.

 The resolvent operator associated to $\partial f$ is called the \emph{proximal operator} and is uniquely defined by the following (strongly convex) minimization problem: $\prox_{f}(x) :=J_{\partial f}(x) = \argmin_{y \in \cH} f(y) + (1/2) \|y - x\|^2$. 
The indicator function of a closed, convex set $C \subseteq \cH$ is denoted by $\iota_C : \cH \rightarrow
\{0, \infty\}$; the indicator function is $0$ on $C$ and is $\infty$ on $\cH \backslash C$. The normal cone operator of $C$ is the  monotone operator $N_C := \partial \iota_C$.

Finally, we call the following identity the \emph{cosine rule}:
\begin{align*}
\|y-z\|^2+2\dotp{y-x,z-x}=\|y-x\|^2+\|z-x\|^2,\quad\forall x,y,z\in\cH \numberthis\label{eq:cosinerule}.
\end{align*}

%% file: section_modification.tex
\subsection{Modifications and enhancements of the algorithm}

\subsubsection{Averaging}\label{averaging}
The  averaging strategies in this subsection maintain additional running averages of  its sequences $(x_A^j)_{j \geq 0}$ and $(x_B^j)_{j \geq 0}$ in  Algorithm~\ref{alg:basic}. Compared to the worst-case rate $o(1/\sqrt{k+1})$ of the original iterates, the running averages have the improved rate of $O(1/(k+1))$, which is   referred to as the \emph{ergodic rate}. This better rate, however, is often contradicted by worse practical performance, for the following reasons:
(i) In many finite dimensional applications, when the iterates reach a  solution neighborhood,  convergence  improves from sublinear to linear, but the ergodic rate typically stays sublinear at $O(1/(k+1))$;  (ii) structures such as sparsity and low-rankness in current iterates often get lost when they are averaged with all their past iterates.
This effect is dramatic in sparse optimization because the average of many sparse vectors can be dense.


The following averaging scheme is typically used in the literature for splitting schemes~\cite{davis2014convergence,davis2014convergenceprimaldual,bo2014convergence}: \begin{align*}
\overline{x}_B^{k} = \frac{1}{\sum_{i=0}^k\lambda_i} \sum_{i=0}^k \lambda_ix_B^i && \text{and} && \overline{x}_A^{k} = \frac{1}{\sum_{i=0}^k\lambda_i} \sum_{i=0}^k \lambda_ix_A^i,\numberthis\label{avg1}
\end{align*}
where all $\lambda_i$, $x_A^i$, and $x_B^i$ are given by Algorithm~\ref{alg:basic}. By maintaining the running averages in Algorithm~\ref{alg:basic}, $\overline{x}_B^{k}$ and $\overline{x}_A^{k}$ are essentially costless to compute.



The following averaging scheme, inspired by~\cite{nedichweighted}, uses a \emph{constant} sequence of relaxation parameters $\lambda_i$ but it gives more weight to the later iterates:
\begin{align*}
\overline{x}_B^{k} = \frac{2}{(k+1)(k+2)} \sum_{i=0}^k (i+1) x_B^i && \text{and} && \overline{x}_A^{k} = \frac{2}{(k+1)(k+2)} \sum_{i=0}^k (i+1)x_A^i.\numberthis\label{avg2}
\end{align*}
This seems intuitively better: the older iterates should matter less than the current iterates.  The above ergodic iterates are closer to the current iterate, but they maintain the improved convergence rate of $O(1/(k+1))$. Like before, $\overline{x}_B^{k}$ and $\overline{x}_A^{k}$ can be computed by updating $\overline{x}_B^{k-1}$ and $\overline{x}_A^{k-1}$ at little cost.
%
%

\subsubsection{Some accelerations}\label{sec:accl}

In this section we introduce an acceleration of Algorithm~\ref{alg:basic} that applies whenever $B$ or $C$ is strongly monotone. If $f$ is strongly convex, then $S=\partial f$ is strongly monotone.   Instead of fixing the step size $\gamma$, a \emph{varying} sequence of stepsizes $(\gamma_j)_{j \geq 0 }$  are used for acceleration. The acceleration is significant on problems where Algorithm~\ref{alg:basic} works nearly at its performance lower bound and the strong convexity constants are easy to obtain. The new algorithm is presented in variables different from those in Algorithm \ref{alg:basic}  since the change from $\gamma_k$ to $\gamma_{k+1}$ occurs in the middle of each iteration of Algorithm \ref{alg:basic}, right  after $J_{\gamma B}$ is applied. In case that $\gamma_k\equiv \gamma$ is fixed, the new algorithm  reduces to Algorithm \ref{alg:basic} with a constant relaxation parameter $\lambda_k\equiv 1$ via the change of variable: $z^k=x_A^{k-1} + \gamma_{k-1} u_B^{k-1}$.   The new algorithm is as follows:

\begin{algo}[Algorithm \ref{alg:basic} with acceleration]\label{alg:accl} Choose $z^0\in \cH$ and stepsizes $(\gamma_j)_{j \geq 0} \in (0, \infty)$. Let $x_A^0 \in \cH$ and set $x_B^0 = J_{\gamma_0 B}(x_A^0), u_B^0 = (1/\gamma_0)(I - J_{\gamma B})(x_A^0)$. For $k=1,2,\ldots$, iterate
\begin{enumerate}
\item get $x_B^{k} = J_{\gamma B}(x_A^{k-1} + \gamma_{k-1} u_B^{k-1});$
\item get $u_B^{k} = (1/\gamma_{k-1})(x_A^{k-1} + \gamma_{k-1}u_B^{k-1} - x_B^{k});$
\item get $x_A^{k} = J_{\gamma_{k}A}(x_B^{k} - \gamma_{k}u_B^{k} - \gamma_{k} Cx_B^{k});$
\end{enumerate}
\end{algo}
The sequence of stepsizes $(\gamma_j)_{j \geq 0 }$, which are related to \cite[Algorithm 2]{chambolle2011first} and~\cite[Algorithm~5]{boct2013convergence}, are introduced in Theorem \ref{thm:accl}. These stepsizes improve the convergence rate of $\|x_B^k - x^\ast\|^2$ to $O(1/(k+1)^2)$.

\begin{theorem}[Accelerated variants of Algorithm~\ref{alg:basic}]\label{thm:accl}
Let $B$ be $\mu_B$-strongly monotone, where we allow the case $\mu_B = 0$.
\begin{enumerate}
\item \label{thm:accl:part:coco} Suppose that $C$ is $\beta$-cocoercive and $\mu_C$-strongly monotone. Let $\eta \in (0, 1)$ and choose $\gamma_0 \in (0, 2\beta(1-\eta))$.  In algorithm \ref{alg:accl}, for all $k \geq 0$, let
\begin{align}\label{eq:faststepsizecoco}
\gamma_{k+1} := \frac{-2\gamma_k^2\mu_C\eta + \sqrt{(2\gamma_k^2\mu_C\eta)^2  + 4(1 + 2\gamma_k\mu_B)\gamma_k^2}}{2(1+2\gamma_k\mu_B)}.
\end{align}
Then 
we have $\|x_B^{k} - x^\ast\|^2 = O(1/(k+1)^2)$.
\item \label{thm:accl:part:Lipschitz}  Suppose that $C$ is $L_C$-Lipschitz, but not necessarily strongly monotone or cocoercive. Suppose that $\mu_B > 0$. Let $\gamma_0 \in (0, 2\mu_B/L_C^2)$. In algorithm \ref{alg:accl}, for all $k \geq 0$, let
\begin{align*}
\gamma_{k+1} &:= \frac{\gamma_k}{\sqrt{1+2\gamma_k(\mu_B - \gamma_kL_C^2/2)}}\numberthis\label{eq:faststepsizeLipschitz}
\end{align*}
Then we have  $\|x_B^{k} - x^\ast\|^2 = O(1/(k+1)^2)$.
\end{enumerate}
\end{theorem}
The proof can be found in 
\iftechreport
  Appendix \ref{appdx:accl}.
\else
  our technical report \cite[Appendix A]{OurTechReport}.
\fi

%% file: section_motivation.tex
\section{Motivation and Applications}\label{sc:motive}
\label{sec:applications}
Our splitting scheme provides simple numerical solutions to a large number of problems that appear in signal processing, machine learning, and statistics. In this section, we provide some concrete problems that reduce to the monotone inclusion problem \eqref{eq:mainprob}. These are a small fraction of the problems to which our algorithm will apply. 
For example, when a problem has four or more blocks, we can reduce it to three or fewer blocks by grouping similar components or lifting the problem to a higher-dimensional space.

For every method, we list the three monotone operators $A$, $B$, and $C$ from problem~\eqref{eq:mainprob}, and a minimal list of conditions needed to guarantee convergence.

We do not include any examples with only one or two blocks  they can be solved by existing splitting algorithms that are special cases of our algorithm.

\subsection{The 3-set (split) feasibility problem} This problem is to find  \beq
x \in \cC_1\cap\cC_2\cap \cC_3,
\eeq where  $\cC_1,\cC_2,\cC_3$ are three nonempty convex sets and the projection to each set can be computed numerically.
The  more general \emph{3-set split feasibility problem} is to find
\beq\label{eq:splitfeas}
x \in \cC_1\cap\cC_2\quad\text{such that}\quad Lx\in \cC_3,
\eeq
where $L$ is a linear mapping. We can reformulate the problem as
\beq\label{eq:distc1c2}
\Min_x~\frac{1}{2}d^2(Lx,\cC_3)~\St~x \in \cC_1\cap\cC_2,
\eeq
where $d(Lx,\cC_3):=\|Lx-P_{\cC_3}(Lx)\|$ and  $P_{\cC_3}$ denotes the projection to $\cC_3$. Problem \eqref{eq:splitfeas} has a solution if and only if problem \eqref{eq:distc1c2} has a solution that gives 0 objective value.

The following algorithm is an instance of Algorithm~\ref{alg:basic} applied with the monotone operators:
\begin{align*}
A x:=  N_{C_1}(x); && Bx :=  N_{C_2}(x); && Cx:=\nabla_x \frac{1}{2}d^2(Lx,\cC_3) = L^*(Lx-P_{\cC_3}(Lx)).
\end{align*}
%
\begin{algo}[3-set split feasibility algorithm]\label{alg:feasibilityproblem}
Set an arbitrary $z^0 \in \cH$, stepsize $\gamma \in (0, 2/\|L\|^2)$, and sequence of relaxation parameters $(\lambda_j)_{j \geq 0} \in (0, 2- \gamma\|L\|^2/2)$. For $k=0,1,\ldots$, iterate
\begin{enumerate}
\item  get $x^k = P_{\cC_2}(z^k)$;
\item  get $y^k = Lx^k$;
\item  get $z^{k+\frac{1}{2}} = 2x^k - z^k- \gamma L^*(y^k-P_{\cC_3}(y^k))$;\qquad //comment: $ z^{k+\frac{1}{2}} = (2J_{\gamma B} - I_{\cH} - \gamma C\circ J_{\gamma B})z^k $
\item  get $z^{k+1}  = z^k + \lambda_k(P_{\cC_1}(z^{k+\frac{1}{2}}) - x^k )$.
\end{enumerate}
\end{algo}

{Note that the algorithm only explicitly applies $L$ and $L^*$, the adjoint of $L$, and does not need to invert a map involving $L$ or $L^\ast$.} The stepsize rule $\gamma\in (0, 2/\|L\|^2)$ follows because $\nabla_x \frac{1}{2}d^2(x,\cC_3)$ is $1$-Lipschitz~\cite[Corollary 12.30]{bauschke2011convex}.

\subsection{The {3-objective minimization problem}}\label{eq:3obj} The problem is to find a solution to
\beq\label{eq:fghL}
\Min_x~f(x) + g(x) + h(Lx),
\eeq
where $f,g,h$ are proper closed convex functions, \textbf{$h$ is $(1/\beta)$-Lipschitz-differentiable},  and $L$ is a linear mapping. Note that any constraint $x\in \cC$ can be written as the indicator function $\iota_{\cC}(x)$ and  incorporated in  $f$ or $g$. Therefore, the problem \eqref{eq:distc1c2} is a special case of \eqref{eq:fghL}.

The following algorithm is an instance of Algorithm~\ref{alg:basic} applied with the monotone operators:
\begin{align*}
A = \partial f; && B =  \partial g; && C=\nabla (h \circ L) = L^\ast\circ  \nabla h \circ L.
\end{align*}

\begin{algo}[for problem \eqref{eq:fghL}]\label{alg:3obj} Set an arbitrary $z^0$, stepsize $\gamma \in (0, 2/(\beta\|L\|^2))$, and sequence of relaxation parameters $(\lambda_j)_{j \geq 0} \in (0, 2- \gamma\beta\|L\|^2/2)$. For $k=0,1,\ldots$, iterate
\begin{enumerate}
\item  get $x^k = \prox_{\gamma g}(z^k)$;
\item  get $y^k = Lx^k$;
\item  get $z^{k+\frac{1}{2}} = 2x^k - z^k- \gamma L^*\nabla h(y^k)$;\qquad //comment: $ z^{k+\frac{1}{2}} = (2J_{\gamma B} - I_{\cH} - \gamma C\circ J_{\gamma B})z^k $
\item  get $z^{k+1}  = z^k +\lambda_k(\prox_{\gamma f}(z^{k+\frac{1}{2}}) - x^k)$.
\end{enumerate}
\end{algo}

\subsubsection{Application: double-regularization and multi-regularization}
Regularization helps  recover a signal with the structures that are either known \emph{a priori} or sought after. In practice, regularization is often enforced through nonsmooth objective functions, such as $\ell_1$ and nuclear norms, or constraints, such as nonnegativity, bound, linear, and norm constraints.   Many problems  involve more than one regularization term (counting both objective functions and constraints), in order to reduce the  ``search space'' and more accurately shape their solutions. Such problems have the general form
\beq\label{mregprob}\Min_{x\in\cH}\, \sum_{i=1}^m r_i(x) + h_0(Lx),\eeq
where $r_i$ are possibly-nonsmooth regularization functions and $h_0$ is a Lipschitz differentiable function. When $m=1,2$, our algorithms can be directly applied to \eqref{mregprob} by setting $f=r_1$ and $g=r_2$ in Algorithm \ref{alg:3obj}.

{When $m\ge 3$, a simple approach is to introduce variables $x_{(i)}$, $i=1,\ldots,m$, and apply Algorithm \ref{alg:3obj} to either of the following problems, both of which are equivalent to   \eqref{mregprob}:
\beq\label{mregprob_prod}\Min_{x,x_{(1)},\ldots,x_{(m)}\in\cH}~ \underbrace{\sum_{i=1}^m r_i(x_{(i)})}_{f}+ \underbrace{\iota_{\{x=x_{(1)}=\cdots=x_{(m)}\}}(x,x_{(1)},\ldots,x_{(m)})}_{g} + h_0(Lx),\eeq
\beq\label{mregprob_prod0} \Min_{x_{(1)},\ldots,x_{(m)}\in\cH}~ \underbrace{\sum_{i=1}^m \left(r_i(x_{(i)}) + \frac{1}{m}h_0(Lx_{(i)})\right)}_{f + h\circ L}+ \underbrace{\iota_{\{x_{(1)}=\cdots=x_{(m)}\}}(x_{(1)},\ldots,x_{(m)})}_{g}\eeq
where $g$ returns 0 if all the inputs are identical and $\infty$ otherwise.\cut{ and for \eqref{mregprob_prod0}  we assume that $\sum_{i=1}^mh_i(L_ix)=h(Lx)$. A simple way to ensure this identity is to let $h_i(L_ix):=\frac{1}{m}h(Lx)$. We allow the more general case because some applications the structures that allow $L_i$ much simpler than $L$. } Problem \eqref{mregprob_prod} has a simpler form, but problem \eqref{mregprob_prod0} requires fewer variables and will be strongly convex in the product space whenever $h_0(Lx)$ is strongly convex in $x$. }

It is easy to adapt Algorithm \ref{alg:3obj} for problems \eqref{mregprob_prod0} and \eqref{mregprob_prod}. We give the one for problem \eqref{mregprob_prod0}:
\begin{algo}[for problem \eqref{mregprob_prod0}]\label{alg:mreg} Set arbitrary $z_{(1)}^0,\ldots,z_{(m)}^0$, {stepsize $\gamma \in (0, 2m/(\beta\|L\|^2))$, and sequence of relaxation parameters $(\lambda_j)_{j \geq 0} \in (0, 2- \gamma\beta\|L\|^2/(2m))$}. For $k=0,1,\ldots$, iterate
\begin{enumerate}
\item  get $x_{(1)}^k,\ldots,x_{(m)}^k = \frac{1}{m}(z_{(1)}^k+\cdots+z_{(m)}^k)$;
\item  get $z_{(i)}^{k+{1/2}} = 2x_{(i)}^k - z_{(i)}^k- \frac{\gamma}{m} L^\ast\nabla h(Lx^k_{(i)}))$  and   $z_{(i)}^{k+1}  = z_{(i)}^k +\lambda_k\left(\prox_{\gamma r_i}(z_{(i)}^{k+{1/2}}) - x_{(i)}^k\right)$, for $i=1,\ldots,m$, in parallel.
\end{enumerate}
\end{algo}
Because Step 1 yields identical $x_{(1)}^k,\ldots,x_{(m)}^k$, they can be consolidated to a single $x^k$ in both steps. {For the same reason, splitting $h_0(L\cdot)$ into multiple copies does not incur more computation.}
\remove{\begin{algo}[for problem \eqref{mregprob_prod}]\label{alg:mreg} Set an arbitrary $z^0,z_{(1)}^0,\ldots,z_{(m)}^0$, stepsize $\gamma \in (0, 1/(\beta\|L\|^2))$, and sequence of relaxation parameters $(\lambda_j)_{j \geq 0} \in (0, 2- \gamma\beta\|L\|^2/2)$. For $k=0,1,\ldots$, iterate
\begin{enumerate}
\item  get $x^k,x_{(1)}^k,\ldots,x_{(m)}^k = \frac{1}{m+1}(z^k+z_{(1)}^k+\cdots+z_{(m)}^k)$;
\item  get $z^{k+1} = z^k+\lambda_k\left(x^k - z^k- \gamma L^*\nabla h(y^k)\right)$, where $y^k = Lx^k$, and   $z_{(i)}^{k+1}  = z_{(i)}^k +\lambda_k\left(\prox_{\gamma r_i}(2x_{(i)}^k - z_{(i)}^k) - x_{(i)}^k\right)$, for $i=1,\ldots,m$, in parallel.
\end{enumerate}
\end{algo}
Step 2 can be computed in parallel since the functions $ h(Lx)$, $r_{1}(x_{(1)}),\ldots,r_{m}(x_{(m)})$ depend on different variables.}


\subsubsection{Application: texture  inpainting}
Let $\vy$ be a color texture image represented as a 3-way tensor where $\vy(:,:,1),\vy(:,:,2),\vy(:,:,3)$ are the red, green, and blue channels of the image, respectively. Let $P_\Omega$ be the linear operator that selects the set of known entries of $\vy$, that is, $P_\Omega\vy$ is given.  The inpainting problem is to recover a set of unknown entries of $\vy$. Because the matrix unfoldings of the texture image $\vy$ are  (nearly) low-rank (as in \cite[Equation (4)]{LiuMusialskiWonkaYe2013}), we formulate the inpainting problem as
\beq\label{eq:inpainting}
\Min_{\vx}~ \omega\|\vx_{(1)}\|_* + \omega\|\vx_{(2)}\|_*  + \frac{1}{2}\|P_\Omega \vx - P_\Omega \vy\|^2
\eeq
where $\vx$ is the 3-way tensor variable, $\vx_{(1)}$ is the matrix  $[\vx(:,:,1)~\vx(:,:,2)~\vx(:,:,3)]$, $\vx_{(2)}$ is the matrix  $[\vx(:,:,1)^T~\vx(:,:,2)^T~\vx(:,:,3)^T]^T$, $\|\cdot\|_*$ denotes matrix nuclear norm, and $\omega$ is a penalty parameter. Problem \eqref{eq:inpainting} can be solved by Algorithm \ref{alg:3obj}. The proximal mapping of the term $\|\cdot\|_*$ can be computed by singular value soft-thresholding.  Our numerical results are given in Section~\ref{sec:numerical:sub:tensor}.

\subsubsection{Matrix completion}

Let $X_0 \in \vR^{m \times n}$ be a matrix with entries that lie in  the interval $[l, u]$, where $l < u$ are positive real numbers. Let $\cA$ be a linear map that ``selects" a subset of the entries of an $m\times n$ matrix by setting each unknown entry in the matrix to $0$. We are interested in recovering matrices $X_0$ from the matrix of ``known" entries $\cA(X_0)$.  Mathematically, one approach to solve this problem is as follows~\cite{5454406}:
\begin{align*}
\Min_{X \in \vR^{m \times n}} & \; \frac{1}{2} \|\cA(X - X_0)\|^2 + \mu \|X\|_\ast\\
\text{subject to:} &\; l \leq X \leq u \numberthis\label{eq:matrixcompletion}
\end{align*}
where $\mu > 0$ is a parameter, $\|\cdot \|$ is the Frobenius norm, and $\|\cdot\|_\ast$ is the nuclear norm. Problem \eqref{eq:matrixcompletion} can be solved by Algorithm \ref{alg:3obj}. The proximal operator of  $\|\cdot\|_\ast$ ball can be computed by soft thresholding the singular values of $X$.  Our numerical results are given in Section~\ref{sec:numerical:sub:matcompletion}.

\subsubsection{Application: support vector machine classification and portfolio optimization}
Consider the constrained quadratic program in $\vR^d$:
\begin{align*}
\Min_{x\in \vR^d}~ & \; \frac{1}{2}\dotp{ Qx, x} + \dotp{ c, x} \numberthis \label{eq:cqprogramming}\\
\text{subject to}~ &\;  x \in \cC_1 \cap \cC_2
\end{align*}
where $Q \in \vR^{d \times d}$ is a symmetric positive semi-definite matrix, $c \in \vR^d$ is a vector, and $\cC_1, \cC_2 \subseteq \vR^d$ are constraint sets. Problem~\eqref{eq:cqprogramming} arises in the dual form soft-margin kernelized support vector machine classifier~\cite{cortes1995support} in which $\cC_1$ is a box constraint and $\cC_2$ is a linear constraint. It also arises in portfolio optimization problems in which $\cC_1$ is a single linear inequality constraint and $\cC_2$ is the standard simplex. See Sections~\ref{sec:numerical:sub:SVM} and~\ref{sec:numerical:sub:PO} for more details.

\remove{{\color{red} I think we should remove the next paragraph and the resulting algorithm. If you agree please remove them.}

Define the smooth function $h(x) = (1/2)\dotp{ Qx, x} + \dotp{ c, x}$ and the nonsmooth indicator functions $g(x) = \iota_{\cC_1}(x)$ (which is $0$ on $\cC_1$ and $\infty$ elsewhere) and  $f:= \iota_{\cC_2}$. This splitting is particularly nice because $\nabla h(x) = Qx + c$ is simple whereas the proximal operator of $h$ requires a matrix inversion, 
which is  expensive for large scale problems. The proximal operators of $f$ and $g$ are just the projections onto $\cC_1$ and $\cC_2$, respectively.

\begin{algo}[for problem \eqref{eq:cqprogramming}]\label{alg:SDP+} Set an arbitrary $z^0$. For $k=0,1,\ldots$, iterate
\begin{enumerate}
\item  get $w^k = P_{\cC_1}(z^k)$;
\item  get $x^k = P_{\cC_2}(2w^k - z^k - \gamma (Qx + c))$;
\item get $z^{k+1} = z^k + x^k - w^k$.
\end{enumerate}
\end{algo}
The algorithm also works in the infinite dimension.}

\subsection{Simplest {3-block extension of ADMM}} The  {3-block monotropic program has the form}
\begin{subequations}\label{eq:admmprob}
\begin{align}
\Min_{x_1,~x_2,~x_3}&~ f_1(x_1) + f_2(x_2) + f_3(x_3)\\
\St&~L_1x_1 + L_2 x_3 + L_3 x_3 = b,
\end{align}
\end{subequations}
where $\cH_1, \ldots, \cH_4$ are Hilbert spaces, the vector $b\in \cH_4$ is given and for $i = 1, 2,3$, the functions $f_i : \cH_i\rightarrow (-\infty, \infty]$ are proper closed convex functions, and $L_i : \cH_i \to \cH_4$ are linear mappings. As usual, any constraint $x_i\in \cC_i$ can be enforced through an indicator function $\iota_{\cC_i}(x)$ and incorporated in $f_i$. We assume that \textbf{$f_1$ is $\mu$-strongly convex} where $\mu>0$.

A new 3-block ADMM algorithm is obtained by  applying Algorithm~\ref{alg:basic} to the dual formulation of \eqref{eq:admmprob} and  rewriting the resulting algorithm using the original functions in \eqref{eq:admmprob}.
Let $f^*$ denote the convex conjugate of a function $f$, and let
\begin{align*}
d_1(w) :=f^*_1(L_1^*w),&& d_2(w) :=f^*_2(L_2^*w),&& d_3(w) :=f^*_3(L_3^*w)-\dotp{w,b}.
\end{align*}
The dual problem of \eqref{eq:admmprob} is
\beq\label{eq:dualprob}
\Min_w~d_1(w)+d_2(w)+d_3(w).
\eeq
Since $f_1$ is $\mu$-strongly convex, $d_1$ is $(\|L_1\|^2/\mu)$-Lipschitz continuous and, hence, the problem \eqref{eq:dualprob} is a special case of \eqref{eq:fghL}. We can 
adapt Algorithm \ref{alg:3obj} to \eqref{eq:dualprob} to get:
\begin{algo}[for problem \eqref{eq:dualprob}]\label{alg:admmdual} Set an arbitrary $z^0$ and stepsize $\gamma\in (0,2\mu/\|L_1\|^2)$. For $k=0,1,\ldots$, iterate
\begin{enumerate}
\item  get $w^k = \prox_{\gamma d_3}(z^k)$;
\item  get $z^{k+\frac{1}{2}} = 2w^k - z^k- \gamma \nabla d_1(w^k)$;
\item  get $z^{k+1}  = z^k+\prox_{\gamma d_2}(z^{k+\frac{1}{2}})-w^k $.
\end{enumerate}
\end{algo}
The following well-known proposition helps implement Algorithm \ref{alg:admmdual} using the original objective functions instead of the dual functions $d_i$.
\begin{proposition}\label{prop:pridual} Let $f$ be a closed proper convex function and let $d(w) := f^*(A^*w)-\dotp{w,c}.$ 
\begin{enumerate}
\item Any  $x'\in\argmin_x f(x) + \dotp{w,Ax-c}$ obeys $Ax'-c\in \partial d(w)$. If $f$ is strictly convex, then $Ax'-c=\nabla d(w)$.
\item Any $x''\in \argmin_x f(x) +\frac{\gamma}{2}\|Ax-c+(1/\gamma)y\|^2$ obeys $Ax''-c\in \partial d(\prox_{\gamma d}(y))$ and $\prox_{\gamma d}(y) = y-\gamma(Ax''-c).$
\end{enumerate}
(We use ``$\in$" with ``$\argmin$" since the minimizers are  not unique in general.)
\end{proposition}

For notational simplicity, let
$$s_{\gamma}(x_1,x_2,x_3,w) := L_1 x_1+L_2 x_2+L_3x_3 - b -\frac{1}{\gamma}w.$$
By Proposition \ref{prop:pridual} and algebraic manipulation, we derive the following algorithm from Algorithm \ref{alg:admmdual}.
\begin{algo}[3-block ADMM]\label{alg:admm} Set an arbitrary $w^0$ and $x_3^0$, as well as  stepsize $\gamma\in (0,2\mu/\|L_1\|^2)$. For $k=0,1,\ldots,$ iterate
\begin{enumerate}
\item get $x_1^{k+1}= \argmin_{x_1} f_1(x_1) + \dotp{w^k,L_1x_1}$;
\item get $x_2^{k+1}\in \argmin_{x_2} f_2(x_2) +\frac{\gamma}{2}\|s(x_1^{k+1}, x_2, x_3^k)\|^2$;
\item get $x_3^{k+1}\in \argmin_{x_3} f_3(x_3) +\frac{\gamma}{2}\|s(x_1^{k+1}, x_2^{k+1}, x_3)\|^2$;
\item\label{alg:admm:w} get $w^{k+1} = w^k - \gamma (L_1x_1^{k+1} + L_2 x_2^{k+1} + L_3 x_3^{k+1}-b)$.
\end{enumerate}
\end{algo}
Note that Step 1 does not involve a quadratic penalty term, and it returns a unique solution since $f_1$ is strongly convex. In contrast, Steps 2 and 3 involve quadratic penalty terms and may have multiple solutions (though the products $L_2x_2^{k+1}$ and $L_3x_3^{k+1}$ are still unique.)

\begin{proposition} If the initial points of Algorithms \ref{alg:admmdual} and \ref{alg:admm} satisfy $z^0 = w^0+\gamma (L_3x_3^0-b)$, then the two algorithms give the same sequence $\{w^k\}_{k\ge 0}$.
\end{proposition}
The proposition is a well-known result based on Proposition \ref{prop:pridual} and algebraic manipulations; the interested reader is referred to \cite[Proposition 11]{davis2014convergence}.
The convergence of Algorithm \ref{alg:admm} is given in the following theorem. 
\begin{theorem}\label{thm:admmconv} Let $\cH_1,\ldots,\cH_4$ be Hilbert spaces, $f_i:\cH_i\to\cH_4$ be proper closed convex functions, $i=1,2,3$, and assume that $f_1$ is $\mu$-strongly convex. Suppose that the set $\cS^*$ of the saddle-point solutions $(x_1,x_2,x_3,w)\in \cH_1\times\cdots\times\cH_4$ to \eqref{eq:admmprob} is nonempty. Let $\rho=\|L_1\|^2/\mu>0$ and pick $\gamma$ \cut{and $\alpha$ }satisfying
\beq \label{admmassump}0<\gamma<\frac{2}{\rho}\cut{\quad\text{and}\quad \alpha \in (0,\bar{\alpha})}\eeq
Then the sequences $\{w^k\}_{k\ge 0}$,  $\{L_2x_2^k\}_{k\ge 0}$, and $\{L_3x_3^k\}_{k\ge 0}$ of Algorithm \ref{alg:admm} converge weakly to $w^*$, $L_2x_2^*$, and $L_3x_3^*$, and $\{x_1^k\}_{k\ge 0}$ converges strongly to $x_1^*$, for some $(w^*,x_1^*,x_2^*,x_3^*)\in \cS^*$.
\end{theorem}

\begin{remark}
Note that it is possible to replace Step~\ref{alg:admm:w} of Algorithm~\ref{alg:admm} with the update rule $w^{k+1} = w^k - \alpha\gamma (L_1x_1^{k+1} + L_2 x_2^{k+1} + L_3 x_3^{k+1}-b)$ where $\alpha \in (0, \bar{\alpha})$ and $$\bar{\alpha} = (2(1-\rho\gamma))^{-1}\left(1-2\rho\gamma + \sqrt{(1-2\rho\gamma)^2+4(1-\rho\gamma)}\right)>1,$$ for $\rho=\|L_1\|^2/\mu$. We do not pursue this generalization here due to lack of space.
\end{remark}

Algorithm \ref{alg:admm} generalizes  several other  algorithms of the alternating direction type.
\begin{proposition}
\begin{enumerate}
\item Tseng's alternating minimization algorithm is a special case of Algorithm \ref{alg:admm} if the $x_3$-block vanishes.
\item The (standard) ADMM is a special case of Algorithm \ref{alg:admm} if  the $x_1$-block vanishes.
\item The augmented Lagrangian method (i.e., the method of multipliers) is a special case of Algorithm \ref{alg:admm} if the $x_1$- and $x_2$-blocks vanish.
\item The Uzawa (dual gradient ascent) algorithm is a special case of Algorithm \ref{alg:admm} if the $x_2$- and $x_3$-blocks vanish.
\end{enumerate}
\end{proposition}

Recently, it was shown that the direct extension of ADMM to three blocks does not converge~\cite{admmdoesnotconverge}. Compared to the recent work \cite{CaiHanYuan2014,ChenShenYou2013,HanYuan2012,LiSunToh2014,LinMaZhang2014} on convergent 3-block extensions of ADMM,  Algorithm \ref{alg:admm} is the simplest and works under the weakest assumption. The first subproblem in Algorithm \ref{alg:admm} does not involve $L_2$ or $L_3$, so it is simpler than the typical ADMM subproblem.  While $f_1$ needs to be strongly convex, no additional assumptions on $f_2,f_3$ and $L_1,L_2,L_3$ are required for the extension. In comparison, \cite{HanYuan2012} assume  that $f_1,f_2,f_3$ are strongly convex functions.  The condition is relaxed to two strongly convex functions in \cite{ChenShenYou2013,LinMaZhang2014} while \cite{ChenShenYou2013} also needs $L_1$ to have full column rank. The papers \cite{LiSunToh2014,CaiHanYuan2014} further reduce the condition to one strongly convex function, and  \cite{LiSunToh2014} uses proximal terms in all the three subproblems and assumes  some positive definitiveness conditions, and  \cite{CaiHanYuan2014} assumes full column rankness on matrices $L_2$ and $L_3$. A variety of convergence rates are established in these papers. It is worth noting that the  conditions assumed by the other ADMM extensions, beyond the strong convexity of $f_1$, are not sufficient for linear convergence, so in theory they do not necessarily convergence faster. In fact, some of the papers use additional conditions in order to prove linear convergence.

\subsubsection{An $m$-block ADMM with $(m-2)$ strongly convex objective functions}
There is a great benefit for not having a quadratic penalty term in Step 1 of Algorithm \ref{alg:admm}. When $f_1(x_1)$ is separable,  Step 1 decomposes to independent sub-steps. Consider the extended monotropic program
 \begin{subequations}\label{eq:madmmprob}
\begin{align}
\Min_{\bar{x}_1,\ldots,\bar{x}_m}&~ \bar{f}_1(\bar{x}_1) + \bar{f}_2(\bar{x}_2) + \cdots + \bar{f}_{m}(\bar{x}_{m})\\
\St&~\bar{L}_1\bar{x}_1 ~+~ \bar{L}_2\bar{x}_2 ~+ \cdots + \bar{L}_m \bar{x}_m = b,
\end{align}
\end{subequations}
where $\bar{f}_1,\ldots,\bar{f}_{m-2}$ are \emph{strongly} convex and $\bar{f}_{m-1},\bar{f}_m$ are convex (but not necessarily strongly convex.) Problem \eqref{eq:madmmprob} is a special case of problem \eqref{eq:admmprob} if we group the first $m-2$ blocks. Specifically, we let $f_1(x_1) :=\bar{f}_1(\bar{x}_1)+\cdots +\bar{f}_{m-2}(\bar{x}_{m-2})$, $f_2(x_2) := \bar{f}_{m-1}(\bar{x}_{m-1})$,  $f_3(x_3):=\bar{f}_{m}(\bar{x}_{m})$, and define $x_1,x_2,x_3, L_1,L_2,L_3$ in obvious ways. Define $\bar{s}_{\gamma}(x_1,x_2,x_3,w) := \bar{L}_1 \bar{x}_1+\bar{L}_2 \bar{x}_2+\cdots+\bar{L}_m\bar{x}_m - b -\frac{1}{\gamma}w.$
 Then, it is straightforward to adapt Algorithm \ref{alg:admm} for problem \eqref{eq:madmmprob} as:
\begin{algo}[$m$-block ADMM]\label{alg:madmm} Set an arbitrary $w^0$ and $\bar{x}_m^0$, and stepsize $\gamma \in(0, \min\{2\|L_i\|/\mu_i \mid i = 1, \cdots, m-2\})$. For $k=0,1,\ldots,$ iterate
\begin{enumerate}
\item get $\bar{x}_i^{k+1}= \argmin_{\bar{x}_i} \bar{f}_i(\bar{x}_i) + \dotp{w^k,\bar{L}_i\bar{x}_i}$ for $i=1,2,\ldots,m-2$, in parallel;\\[-5pt]
\item get $\bar{x}_{m-1}^{k+1}\in \argmin_{\bar{x}_{m-1}} \bar{f}_{m-1}(\bar{x}_{m-1}) +\frac{\gamma}{2}\|\bar{s}(\bar{x}_1^{k+1},\ldots,\bar{x}_{m-2}^{k+1}, \bar{x}_{m-1}, \bar{x}_m^k)\|^2$;\\[-5pt]
\item get $\bar{x}_m^{k+1}\in \argmin_{\bar{x}_m} \bar{f}_m(\bar{x}_m) +\frac{\gamma}{2}\|\bar{s}(\bar{x}_1^{k+1},\ldots,\bar{x}_{m-1}^{k+1}, \bar{x}_{m})\|^2$;\\[-5pt]
\item get $w^{k+1} = w^k - \gamma (\bar{L}_1 \bar{x}_1^{k+1}+\bar{L}_2 \bar{x}_2^{k+1}+\cdots+\bar{L}_m\bar{x}_m^{k+1} -b)$.
\end{enumerate}
\end{algo}
All convergence properties of Algorithm \ref{alg:madmm} are identical to those of Algorithm \ref{alg:admm}.

%
%
%
%
\subsection{Reducing the number of operators before splitting}\label{eq:relations}

Problems involving multiple operators can be reduced to fewer operators by applying grouping and lifting techniques. They allow Algorithm~\ref{alg:basic} and existing splitting schemes to handle four or more operators.

In general, two or more Lipschitz-differentiable functions (or cocoercive operators)  can be grouped into one function (or one cocoercive operator, respectively). On the other hand, grouping nonsmooth functions with simple proximal maps (or monotone operators with simple resolvent maps) may lead to a much more difficult proximal map (or resolvent map, respectively). One resolution is lifting: to introduce dual and dummy variables and  create fewer but ``larger" operators. 
It comes with the cost that the introduced variables increase the problem size and may slow down convergence. 

For example, we can  reformulate Problem~\eqref{eq:mainprob}  in the form (which abuses the block matrix notation):
\beq\label{pdsplit}
0\in \begin{bmatrix}B & I \\ -I & ~A^{-1}\end{bmatrix}\begin{bmatrix}x\\ y \end{bmatrix} + \begin{bmatrix}Cx\\0\end{bmatrix}=: \bar{A} \begin{bmatrix}x\\y\end{bmatrix}+\bar{C}\begin{bmatrix}x\\y\end{bmatrix}.
\eeq
Here we have introduced $ y\in Ax$, which is equivalent to $x\in A^{-1}y$ or the second row of \eqref{pdsplit}. Both the operators $\bar{A}$ and $\bar{C}$ are monotone, and the  operator $\bar{C}$ is cocoercive since $C$ is so. Therefore, the problem~\eqref{eq:mainprob} has been reduced to a monotone inclusion involving two ``larger" operators. Under a special metric, applying the FBS iteration in \cite{condat2013primal} gives the following algorithm:
\begin{algo}[\cite{condat2013primal}]\label{alg:pdfbs} Set an arbitrary $x^0,y^0$. Set stepsize parameters $\tau,\sigma$. For $k=1,\ldots,$ iterate:
\begin{enumerate}
\item get $x^{k} = J_{\tau B}(x^{k-1} - \tau Cx^{k-1} -\tau y^{k-1})$;
\item get $y^{k} = J_{\sigma A^{-1}}(y^{k-1}+\sigma (2 x^{k}-x^{k-1}))$ \qquad //comment: $J_{\sigma A^{-1}}= I- \sigma J_{\sigma^{-1} A}\circ(\sigma^{-1} I)$.
\end{enumerate}
\end{algo}
The lifting technique can be applied to the monotone inclusion problems with four or more operators together with Algorithm \ref{alg:basic}.  Since Algorithm \ref{alg:basic} handles three operators, it generally  requires less lifting than previous algorithms. We re-iterate that  FBS is a special case of our splitting, so Algorithm~\ref{alg:pdfbs} is a special case of Algorithm \ref{alg:basic} applied to \eqref{pdsplit} with a vanished $\bar{B}$.

Because both Algorithms \ref{alg:basic} and \ref{alg:pdfbs} solve the problem~\eqref{eq:mainprob}, it is interesting to compare them. Note that one cannot obtain one algorithm from the other through algebraic manipulation. Both algorithms  apply $J_A$, $J_B$, and $C$ once every iteration. We managed to rewrite Algorithm~\ref{alg:basic} in the following equivalent form (see
\iftechreport
  Appendix \ref{app:pdours}
\else
  our technical report \cite[Appendix B]{OurTechReport}
\fi
for a derivation) that is most similar to Algorithm  \ref{alg:pdfbs} for the purpose of comparison:
\begin{algo}[Algorithm \ref{alg:basic} in an equivalent form]\label{alg:pdours} Set an arbitrary $x^0$ and $y^0$. For $k=1,\ldots,$ iterate:
\begin{enumerate}
\item get $x^k = J_{\gamma B}\left(x^{k-1} - \gamma Cx^{k-1} - \gamma y^{k-1}\right)$;
\item get $y^{k} = J_{\frac{1}{\gamma} A^{-1}}\left( y^{k-1} + \frac{1}{\gamma}(2x^{k} - x^{k-1}) + (Cx^{k} - Cx^{k-1})\right)$ \quad //comment: $J_{\sigma A^{-1}}= I- \sigma J_{\sigma^{-1} A}\circ(\sigma^{-1} I)$.
\end{enumerate}
\end{algo}
The  difference between Algorithms~\ref{alg:pdfbs} and~\ref{alg:pdours} is the extra correction factor $Cx^{k} - Cx^{k-1}$. Without the correction factor, we cannot eliminate $y^k$ and express  Algorithms~\ref{alg:pdfbs} in the form of \eqref{eq:zitr}.

%% file: section_convergencetheory.tex
\section{Convergence theory}\label{sec:convergence}
In this section, we show that Problem~\eqref{eq:mainprob} can be solved by iterating the operator $T$ defined in Equation~\eqref{eq:newoperator}:
$T = I_{\cH} - J_{\gamma B} + J_{\gamma A}\circ(2J_{\gamma B} - I_{\cH} - \gamma C\circ J_{\gamma B}).
$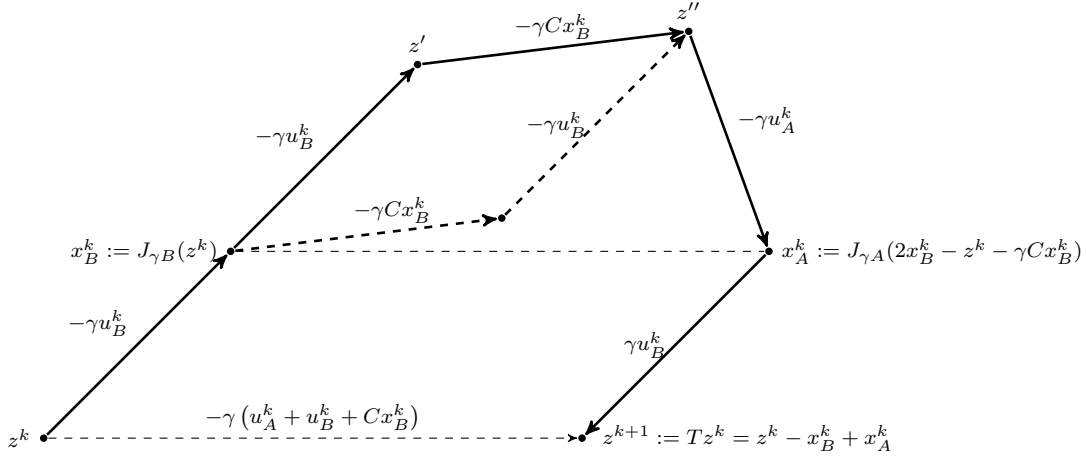
\begin{figure}[h!]

    \begin{center}
    \begin{tikzpicture}[node distance=100pt,auto]


    \node (z0) [label=left:${z^k}$,smallnode] at (0,0){};

    \node (xSU) [label=left:${x_B^k:=J_{\gamma B}(z^k)}$,above right of=z0,smallnode]{};
    \draw[style=bwd] (z0) -> node {$-\gamma u_B^k$} (xSU);

    \node (z') [label=above:$z'$,above right of=xSU,smallnode]{};
    \draw[style=refl] (xSU) -> node {$-\gamma u_B^k$} (z');

    \node (xSU') [above right=10pt and 100pt of xSU,smallnode]{};
    \draw[style=fwddash] (xSU) -> node[pos=0.60, above] {$-\gamma Cx_B^k$} (xSU');

    \node (z'') [label=above:$z''$,above right=10pt and 100pt of z',smallnode]{};
    \draw[style=fwd] (z') -> node[pos=0.50, above] {$-\gamma Cx_B^k$} (z'');
    \draw[style=fwddash] (xSU') -> node[pos=0.50, left] {$-\gamma u_B^k$} (z'');
    \node (xT) [label=right:${x_A^k := J_{\gamma A}(2x_B^k - z^k - \gamma Cx_B^k)}$,right=200pt of xSU,smallnode]{};
    \draw[style=bwd] (z'') -> node {$-\gamma u_A^k$} (xT);

    \node (z1) [label=right:${~z^{k+1}:=Tz^k=z^k-x_B^k+x_A^k}$,below left of=xT,smallnode]{};
    \draw[style=fwd] (xT) -> node[left] {$\gamma u_B^k$} (z1);

    \draw[dashed] (xSU) -> (xT);
    \draw[dashed,->] (z0) ->node[above] {\small $-\gamma \left(u_A^k+u_B^k+Cx_B^k\right) $} (z1);

    \end{tikzpicture}
\end{center}

    \caption{The mapping $T:z^k\mapsto z^{k+1}:=Tz^k$. The vectors $u_B^k\in Bx_B^k$ and $u_A^k \in Ax_A^k$ are defined in Lemma~\ref{lem:identities}.
}
    \label{fig:DYSTR}

\end{figure}

Figure~\ref{fig:DYSTR} depicts the process of applying $T$ to  a point $z \in \cH$. Lemma~\ref{lem:identities} defines the points in Figure~\ref{fig:DYSTR} .

%

\begin{lemma}\label{lem:identities}
Let $z\in \cH$ and define points:
\begin{align*}
x_B^k &:= J_{\gamma B}(z^k), && z':=2x_B^k-z^k,  && x_A^k := J_{\gamma A}(z''), \\
u_B^k &:=\gamma^{-1} (z ^k- x_B^k) \in Bx_B^k, && z'':=z'-\gamma Cx_B^k, && u_A^k :=\gamma^{-1} (z''-  x_A^k) \in Ax_A^k.
\end{align*}
Then the following identities hold:
\begin{align*}
Tz^k - z^k &= x_A ^k- x_B^k = -\gamma(u_B ^k+ u_A ^k+ Cx_B^k) && \text{and} && Tz^k = x_A ^k+ \gamma u_B^k.
\end{align*}
When $B = \partial g$, we let $\tnabla g(x_g^k) := u_B^k \in \partial g(x_g^k)$. Likewise when $A = \partial f$, we let $\tnabla f(x_f^k) := u_A^k \in \partial f(x_f^k)$.
\end{lemma}
\begin{proof}
Observe that $Tz^k = z^k + x_A^k - x_B^k$ by the definition of $T$ (Equation~\eqref{eq:newoperator}). In addition, $Tz^k = x_A ^k+ z ^k- x_B ^k= x_A ^k+ \gamma u_B^k$. Finally, we have $x_A^k - x_B^k = 2x_B^k- z ^k- \gamma u_A^k - \gamma Cx_B^k - x_B^k = -\gamma (u_A^k + u_B^k + Cx_B^k)$.\qed
\end{proof}

The following proposition computes a fixed point identity for the operator $T$. It shows that we can recover a zero of $A + B + C$ from any fixed point $z^*$ of $T$ by computing $J_{\gamma B}z^*$.
\begin{lemma}[Fixed-point encoding]\label{lem:fixedpoints}
The following set equality holds
\begin{align*}
\zer(A + B + C) &= J_{\gamma B}(\Fix T).
\end{align*}
In addition,
\begin{align*}
\Fix T &= \{ x + \gamma u \mid 0 \in (A + B + C)x, u  \in  (Bx) \cap (-Ax - Cx)\}.
\end{align*}
\end{lemma}
The proof can be found in 
\iftechreport
  Appendix \ref{app:convergencetheory}.
\else
  our technical report \cite[Appendix C]{OurTechReport}.
\fi
%
%
The next lemma will help us establish the averaged coefficient of the operator $T$ in the next proposition. Note that in the lemma,  if we let $W:=0$, $U:=I_{\cH}  - J_{\gamma B}$, and $T_1:=J_{\gamma A}$,  the operator $S$ reduces to the DRS operator $I_{\cH}  - J_{\gamma B}+ J_{\gamma A} \circ (2 J_{\gamma B} - I_{\cH})$, which is known to be 1/2-averaged. 
\begin{lemma}\label{lm:STUVW} Let $S := U+T_1\circ V$, where $U,\,T_1: \cH \rightarrow \cH$ are both firmly nonexpansive and $V: \cH \rightarrow \cH$. Let
$W=I-(2U+V).$
Then we have for all $z, w \in \cH$:
\beq\label{eq:STUVW}\|Sz-Sw\|^2 \le \|z-w\|^2-\|(I_{\cH} - S)z - (I_{\cH} - S)w\|^2-2\dotp{ T_1\circ Vz- T_1\circ Vw, Wz-Ww}.
\eeq
\end{lemma}
The proof can be found in 
\iftechreport
  Appendix \ref{app:convergencetheory}.
\else
  our technical report \cite[Appendix C]{OurTechReport}.
\fi

The following proposition will show that the operator $T$ is averaged. This proposition is crucial for proving the convergence of Algorithm~\ref{alg:basic}.

\begin{proposition}[Averageness of $T$]\label{prop:cocoerciveaveraged}
Suppose that $T_1,\, T_2 : \cH \rightarrow \cH$ are firmly nonexpansive and $C$ is $\beta$-cocoercive, $\beta>0$. Let $\gamma\in (0,2\beta)$. Then
\begin{align*}
T:=I - T_2 + T_1\circ(2T_2 - I_{\cH} - \gamma C \circ T_2)
\end{align*}
is $\alpha$-averaged with coefficient $\alpha := \frac{2\beta}{4\beta - \gamma}<1.$ In particular, the following inequality holds for all $z, w \in \cH$
\begin{align*}
\|Tz - Tw\|^2 &\leq \|z-w\|^2 -\frac{(1 - \alpha)}{\alpha}\|(I_{\cH} -T)z - (I_{\cH} - T)w\|^2. \numberthis\label{eq:operatormaininequality}
\end{align*}
\end{proposition}
\begin{proof}
To apply Lemma \ref{lm:STUVW}, we let $U:=I_\cH-T_2$, $V:=2T_2 - I_{\cH} - \gamma C \circ T_2$, and $W:=\gamma C\circ T_2$. Note that $U$ is firmly nonexpansive (because $T_2$ is), and we have $W=I_\cH-(2U+V)$. Let $S :=T= I_\cH-T_2+T_1\circ V$. We  evaluate the inner product in \eqref{eq:STUVW} as follows:
\begin{align*}
&-2\dotp{ T_1\circ Vz- T_1\circ Vw, Wz-Ww}\\
&=2\dotp{(I_\cH-T)z-(I_\cH-T)w, \gamma C\circ T_2z-\gamma C\circ T_2w}-2\dotp{T_2z-T_2w, \gamma C\circ T_2z-\gamma C\circ T_2 w}\\
&\le \varepsilon\|(I_{\cH} - T)z - (I_{\cH} - T)w\|^2 + \frac{\gamma^2}{\varepsilon}\|C\circ T_2z - C\circ T_2w\|^2-2\gamma\beta \|C\circ T_2z - C\circ T_2w\|^2\\
& = \varepsilon\|(I_{\cH} - T)z - (I_{\cH} - T)w\|^2  -{\gamma}(2\beta-\gamma/\varepsilon)\|C\circ T_2z - C\circ T_2w\|^2
\end{align*}
where the inequality follows from Young's inequality with any $\varepsilon>0$ and that $C$ is $\beta$-cocoercive. We set $$\varepsilon := \gamma/2\beta<1$$ so that the coefficient $\gamma (2\beta-\gamma/\varepsilon)=0$. Now applying Lemma \ref{lm:STUVW} and using $S=T$, we obtain
$$\|Tz-Tw\|^2 \le \|z-w\|^2-(1-\epsilon)\|(I_{\cH} - T)z - (I_{\cH} - T)w\|^2,$$
which is identical to \eqref{eq:operatormaininequality} under our definition of $\alpha$. \qed
\end{proof}

\begin{remark} It is easy to slightly strengthen the inequality  \eqref{eq:operatormaininequality} as follows: For any $\bar\varepsilon\in(0,1)$ and $\bar\gamma \in (0, 2\beta\bar\varepsilon)$,  let $\bar\alpha := 1/(2-\bar\varepsilon)<1$. Then the following holds for all $z, w \in \cH$:
\begin{align*}
\|Tz - Tw\|^2 &\leq \|z-w\|^2 -\frac{(1 - \bar\alpha)}{\bar\alpha}\|(I_{\cH} - T)z - (I_{\cH} - T)w\|^2 \\
&-\bar\gamma\left(2\beta - \frac{\bar\gamma}{\bar\varepsilon}\right) \|C\circ T_2 (z) - C\circ T_2(w)\|^2. \numberthis\label{eq:operatormaininequalitystronger}
\end{align*}
\end{remark}

\begin{remark}
When $C = 0$, the mapping in Equation~\eqref{eq:reflectionversionofmap} reduces to  $S = \refl_{\gamma A} \circ \refl_{\gamma B}$, which is nonexpansive because it is the composition of nonexpansive maps. Thus, $T = (1/2)I_{\cH} + (1/2) S$ is firmly nonexpansive by definition. However, when $C \neq 0$, the mapping $S$ in~\eqref{eq:reflectionversionofmap} is no longer nonexpansive.  The mapping $2T_2 - I_{\cH} - \gamma C$, which is a part of $S$, can be expansive. Indeed, consider the following example: Let $\cH = \vR^2$, let $B = \partial\iota_{\{(x_1, 0) \mid x_1 \in \vR\}}$ be the normal cone of the $x_1$ axis, and let $C = \nabla ((1/2) \|x_1 + x_2 \|^2) = (x_1 + x_2, x_1 + x_2)$. In particular, $T_2(x_1, x_2) = J_{\gamma B}(x_1, x_2) = (x_1, 0)$ for all $(x_1, x_2) \in \vR^2$ and $\gamma > 0$. Then the point $0$ is a fixed point of $R = 2T_2 - I_{\cH} - \gamma C\circ T_2$, and $$R(1, 1) = (1, -1) - \gamma C(1, 0) = (1- \gamma, -1 - \gamma).$$
Therefore, $\|R(1, 1) - R(0, 0) \| = \sqrt{2 + 2 \gamma^2} > \sqrt{2}= \|(1, 1) - (0, 0)\|$ for all $\gamma > 0$.
\end{remark}

\begin{remark}
When $B = 0$, the averaged parameter $\alpha = 2\beta/(4\beta - \gamma)$ in Proposition~\ref{prop:cocoerciveaveraged} reduces to the best (i.e., smallest) known averaged coefficient for the forward-backward splitting algorithm~\cite[Proposition 2.4]{Combettes2014}.
\end{remark}

We are now ready to prove convergence of Algorithm~\ref{alg:basic}.

\begin{theorem}[Main convergence theorem]\label{thm:convergence}
Suppose that $\Fix T\not=\emptyset$. Set  a stepsize $\gamma \in (0, 2\beta\varepsilon)$, where $\varepsilon \in (0, 1)$.  Set $(\lambda_j)_{j \geq 0} \subseteq (0, 1/\alpha)$ as a sequence of relaxation parameters, where  $\alpha = 1/(2-\varepsilon) < 2\beta/(4\beta - \gamma)$, such that  for  $\tau_k: = (1-\lambda_k/\alpha)\lambda_k/\alpha$ we have $\sum_{i=0}^\infty \tau_i = \infty$. Pick any start point $z^0 \in \cH$. Let $(z^j)_{j \geq 0}$ be generated by Algorithm \ref{alg:basic}, i.e., the following iteration: for all $k \geq 0$,
\begin{align*}
z^{k+1} = z^k + \lambda_k (Tz^k - z^k).
\end{align*}
Then the following hold
\begin{enumerate}
\item \label{thm:convergence:part:biginequality} Let $z^\ast \in \Fix T$. 
Then $(\|z^j - z^\ast\|)_{j \geq 0}$ is monotonically decreasing.
\item \label{thm:convergence:part:FPRmonotone} The sequence $(\|Tz^{j} - z^{j}\|)_{j \geq 0}$ is monotonically decreasing and converges to $0$.
\item \label{thm:convergence:part:weakfixedpoint} The sequence $(z^j)_{j\geq 0}$ weakly converges to a fixed point of $T$.
\item \label{thm:convergence:part:gradientsum} Let $x^\ast \in \zer(A + B+ C)$. Suppose that $\inf_{j \geq 0} \lambda_j > 0$. Then the following sum is finite:
\begin{align*}
\sum_{i=0}^\infty  \lambda_k\|Cx_B^k - C x^\ast\|^2 \leq \frac{1}{\gamma(2\beta - \gamma/\varepsilon)}\|z^0 - z^\ast\|^2
\end{align*}
In particular, $(Cx_B^j)_{j \geq 0}$ converges strongly to $Cx^\ast.$
\item \label{thm:convergence:part:weakminimizerB} Suppose that $\inf_{j \geq 0} \lambda_j > 0$ and let $z^*$ be the weak sequential limit of $(z^j)_{j \geq 0}$. Then the sequence $(J_{\gamma B}(z^j))_{j \geq 0}$ weakly converges to $J_{\gamma B}(z^\ast) \in \zer(A + B + C)$.
\item \label{thm:convergence:part:weakminimizerA} Suppose that $\inf_{j \geq 0} \lambda_j > 0$ and let $z^*$ be the weak sequential limit of $(z^j)_{j \geq 0}$. Then the sequence $(J_{\gamma A}\circ(2J_{\gamma B} - I_{\cH} - \gamma C\circ J_{\gamma B}) (z^j))_{j \geq 0}$ weakly converges to $J_{\gamma B}(z^\ast) \in \zer(A + B + C)$.
\item \label{thm:convergence:part:convergencerate} Suppose that $\underline{\tau} := \inf_{j \geq 0} \tau_j > 0$. For all $k \geq 0$, the following convergence rates hold:
\begin{align*}
\|Tz^{k} - z^k\|^2 \leq \frac{\|z^0 - z^\ast\|^2}{\underline{\tau}(k+1)} && \mathrm{and} && \|Tz^{k} - z^k\|^2 = o\left(\frac{1}{k+1}\right)
\end{align*}
for any point $z^\ast \in \Fix(T)$.
\item \label{thm:convergence:part:strongconvergence} Let $z^*$ be the weak sequential limit of $(z^j)_{j \geq 0}$. The sequences $(J_{\gamma B}(z^j))_{j \geq 0}$ and $(J_{\gamma A}\circ(2J_{\gamma B} - I_{\cH} - \gamma C\circ J_{\gamma B}) (z^j))_{j \geq 0}$ converge strongly to a point in $\zer(A + B+ C)$ whenever any of the following holds:
\begin{enumerate}
\item\label{thm:convergence:part:strongconvergence:part:A} $A$ is uniformly monotone\footnote{A mapping $A$ is uniformly monotone if there exists increasing function $\phi:\vR_+\to [0,+\infty]$ such that $\phi(0)=0$ and for any $u\in Ax$ and $v\in Ay$, $\dotp{x-y,u-v}\ge \phi(\|x-y\|).$ If $\phi\equiv \beta (\cdot)^2>0$, then the mapping $A$ is strongly monotone. If a proper function $f$ is uniformly (strongly) convex, then $\partial f$ is uniformly (strongly, resp.) monotone.} on every nonempty bounded subset of $\dom(A)$;
\item \label{thm:convergence:part:strongconvergence:part:B}$B$ is uniformly monotone on every nonempty bounded  subset of $\dom(B)$;
\item \label{thm:convergence:part:strongconvergence:part:C} $C$ is demiregular at every point $x \in \zer(A+ B + C)$.\footnote{A mapping $C$ is demiregular at $x \in \dom(C)$ if for all $u \in Cx$ and all sequences $(x^k, u^k) \in \gra(C)$ with $x^k \rightharpoonup x$ and $u^k \rightarrow u$, we have $x^k \rightarrow x$.}
\end{enumerate}
\end{enumerate}
\end{theorem}
\begin{proof}
Part~\ref{thm:convergence:part:biginequality}: Fix $k \geq 0$.  Observe that
\begin{align*}
\|z^{k+1} - z^\ast\|^2 &= \| (1-\lambda_k)(z^{k} - z^\ast) + \lambda_k(Tz^k - z^\ast)\|^2\\
&= (1-\lambda_k)\| z^{k} - z^\ast\|^2 + \lambda_k\|Tz^k - z^\ast\|^2 - \lambda_k(1-\lambda_k)\|Tz^k - z^{k}\|^2\numberthis\label{theorem:convergence:eq:cvxcomb}
\end{align*}
by Corollary~\cite[Corollary 2.14]{bauschke2011convex}. In addition, from Equation~\eqref{eq:operatormaininequalitystronger}, we have
\begin{align*}
\|Tz^k - z^\ast\|^2 &\leq \|z^k - z^\ast\|^2 - \frac{1-\alpha}{\alpha} \|Tz^{k} - z^k\|^2  - \gamma\left(2\beta - \frac{\gamma}{\varepsilon}\right) \|C\circ T_2 (z^k) - C\circ T_2(z^\ast)\|^2.
\end{align*}
Therefore, the monotonicity follows by combining the above two equations and using the simplification
\begin{align*}
 \tau_k &= \lambda_k(1-\lambda_k) + \frac{\lambda_k(1-\alpha)}{\alpha}
\end{align*}
to get
\begin{align*}
\|z^{k+1} - z^\ast\|^2 + \tau_k\|Tz^k - z^{k}\|^2 + \gamma\lambda_k\left(2\beta - \frac{\gamma}{\varepsilon}\right)\|Cx_B^k - CJ_{\gamma B}(z^\ast)\|^2 & \leq \|z^k - z^\ast\|^2.
\end{align*}

Part~\ref{thm:convergence:part:FPRmonotone}: This follows from~\cite[Proposition 5.15(ii)]{bauschke2011convex}.

Part~\ref{thm:convergence:part:weakfixedpoint}: This follows from~\cite[Proposition 5.15(iii)]{bauschke2011convex}.

Part~\ref{thm:convergence:part:gradientsum}: The inequality follows by summing the last inequality derived in Part~\ref{thm:convergence:part:biginequality}. The convergence of $(Cx_B^j)_{j \geq 0}$ follows because $ \inf_{j \geq 0} \lambda_j > 0$ and the sum is finite.

Part~\ref{thm:convergence:part:weakminimizerB}:  Recall the notation from Lemma~\ref{lem:identities}: set $x_B^k = J_{\gamma B}(z^k), ~x_A^k = J_{\gamma A}(2x_B^k - z^k - \gamma C x_B^k),~ u_B^k = (1/\gamma)(z^k - x_B^k) \in Bx_B^k$, and $u_A^k = (1/\gamma)(2x_B^k - z^k - \gamma C x_B^k - x_A^k) \in Ax_A^k$.

Since $\|x_B^k - J_{\gamma B}(z^\ast)\|= \|J_{\gamma B}(z^k) - J_{\gamma B}(z^\ast)\| \leq \|z^k - z^\ast\| \leq \|z^0 - z^\ast\|$, $\forall k \geq 0$, the sequence $(x_B^j)_{j \geq 0}$ is bounded and has a weak sequential cluster point $\overline{x}$. Let $x_B^{k_j} \rightharpoonup \overline{x}$ as $j \rightarrow \infty$ for  index subsequence  $(k_j)_{j \geq 0}$.

Let $x^\ast \in \zer(A + B + C)$. Because $C$ is maximal monotone, $Cx_B^k \rightarrow Cx^\ast$, and $x_B^{k_j} \rightharpoonup \overline{x}$, it follows by the weak-to-strong sequential closedness of $C$ that $C\overline{x} = Cx^\ast$ \cite[Proposition 20.33(ii)]{bauschke2011convex} and thus $Cx_B^{k_j} \rightarrow C\overline{x}$. Because $x_A^k - x_B^k = Tz^k - z^k \rightarrow 0$ as $k \rightarrow \infty$ by Part~\ref{thm:convergence:part:FPRmonotone} and Lemma~\ref{lem:identities}, it follows that
\begin{align*}
x_B^{k_j} \rightharpoonup \overline{x},\quad x_A^{k_j} \rightharpoonup \overline{x},\quad Cx_B^{k_j} \rightarrow C\overline{x},\quad  u_B^{k_j} \rightharpoonup \frac{1}{\gamma}(z^\ast - \overline{x}),\quad \text{and } u_A^{k_j} \rightharpoonup \frac{1}{\gamma}(\overline{x} - z^\ast - \gamma C\overline{x})
\end{align*}
as $j \rightarrow \infty$.

Thus,~\cite[Proposition 25.5]{bauschke2011convex} applied to  $(x_A^{k_j},u_A^{k_j})\in\gra  A,~ (x_B^{k_j},u_B^{k_j})\in B,$ and $(x_B^{k_j},Cx_B^{k_j})\in C$ shows that $\overline{x} \in \zer(A + B+ C)$, $z^\ast - \overline{x} \in  \gamma B \overline{x}$, and $\overline{x} - z^\ast - \gamma C\overline{x} \in \gamma A\overline{x}$. Hence, as $\overline{x} = J_{\gamma B}(z^\ast)$ is unique, $\overline{x}$ is the unique weak sequential cluster point of $(x_B^j)_{j \geq 0}$.  Therefore, $(x_B^j)_{j \geq 0}$ converges weakly to $J_{\gamma B}(z^\ast)$ by~\cite[Lemma 2.38]{bauschke2011convex}.

Part~\ref{thm:convergence:part:weakminimizerA}: Assume the notation of Part~\ref{thm:convergence:part:weakminimizerB}. We shall show $x_A^k\rightharpoonup J_{\gamma B}(z^*)$. This follows because $x_A^k - x_B^k = Tz^k - z^k \rightarrow 0$ as $k \rightarrow \infty$ and $x_B^k \rightharpoonup J_{\gamma B}(z^\ast)$.

Part~\ref{thm:convergence:part:convergencerate}: The result follows from~\cite[Theorem 1]{davis2014convergence}.

Part~\ref{thm:convergence:part:strongconvergence}: Assume the notation of Part~\ref{thm:convergence:part:weakminimizerB} and let $x^\ast = J_{\gamma B}(z^\ast), u_B^\ast = (1/\gamma)(z^\ast - x^\ast) \in Bx^\ast$, and $u_A^\ast = (1/\gamma)(x^\ast - z^\ast) - Cx^\ast$. Now we move to the subcases.

Part~\ref{thm:convergence:part:strongconvergence:part:A}: Because $B + C$ is monotone and $(x_B^{k},u_B^{k})\in B$, we have $\dotp{x_B^k - x^\ast, u_B^k + Cx_B^k - (u_B^\ast + Cx_B^k)} \geq 0$ for all $k \geq 0$. Consider the bounded set $S = \{x^\ast\} \cup \{x_A^j \mid j \geq 0\}$. Then there exists an increasing function $\phi_A : \vR_+ \rightarrow [0, \infty]$ that vanishes only at $0$ such that
\begin{align*}
\gamma \phi_A(\|x_A^k - x^\ast\|) &\leq \gamma \dotp{x_A^k - x^\ast, u_A^k - u_A^\ast} + \gamma \dotp{x_B^k - x^\ast, u_B^k + Cx_B^k - (u_B^\ast + Cx_B^\ast)} \\
&= \gamma \dotp{x_A^k - x_B^k, u_A^k - u_A^\ast} + \gamma \dotp{x_B^k - x^\ast, u_A^k - u_A^\ast} +\gamma\dotp{x_B^k - x^\ast, u_B^k + Cx_B^k - (u_B^\ast + Cx_B^\ast)}\\
&= \gamma \dotp{x_A^k - x_B^k, u_A^k - u_A^\ast} + \gamma \dotp{x_B^k - x^\ast, u_A^k + u_B^k + Cx_B^k} \\
&=  \dotp{ x_B^k - x_A^k,  x_B^k - \gamma u_A^k - (x^\ast - \gamma u_A^\ast)} \\
&= \dotp{ x_B^k - x_A^k, z^k - z^\ast} + \gamma\dotp{x_B^k - x_A^k, Cx_B^k - Cx^\ast} \rightarrow 0 \quad \text{as $k \rightarrow \infty$}
\end{align*}
where the convergence to $0$ follows because $x_B^k - x_A^k = z^k - Tz^k \rightarrow 0,$ $z^k \rightharpoonup z^\ast$, and $Cx_B^k \rightarrow Cx^\ast$ as $k \rightarrow \infty$. Furthermore, $x_B^k \rightarrow x^\ast$ because $x_A^k - x_B^k \rightarrow 0$ as $k \rightarrow \infty$.

Part~\ref{thm:convergence:part:strongconvergence:part:B}:
Because $A$ is monotone, we have $\dotp{x_A^k - x^\ast, u_A^k - u_A} \geq 0$ for all $k \geq 0$. In addition, note that $B+C$ is also uniformly monotone on all bounded sets. Consider the bounded set $S = \{x^\ast\} \cup \{x_B^j \mid j \geq 0\}$. Then there exists an increasing function $\phi_B : \vR_+ \rightarrow [0, \infty]$ that vanishes only at $0$ such that
\begin{align*}
\phi_B(\|x_B^k - x^\ast\|) &\leq \gamma \dotp{x_A^k - x^\ast, u_A^k - u_A^\ast} + \gamma \dotp{x_B^k - x^\ast, u_B^k + Cx_B^k - (u_B^\ast + Cx_B^\ast)} \rightarrow 0 \quad \text{as $k \rightarrow \infty$}
\end{align*}
by the argument in Part~\ref{thm:convergence:part:strongconvergence:part:A}. Therefore, $x_B^k \rightarrow x^\ast$ strongly.

Part~\ref{thm:convergence:part:strongconvergence:part:C}: Note that $Cx_B^k \rightarrow Cx^\ast$ and $x_B^k \rightharpoonup x^\ast$. Therefore, $ x_B^k \rightarrow x^\ast$ by the demiregularity of $C$.
\qed\end{proof}

\begin{remark}
Theorem~\ref{thm:convergence} can easily be extended to the summable error scenario, where for all $k \geq 0$, we have
\begin{align*}
z^{k+1} &= z^k + \lambda_k (Tz^k - z^k + e_k)
\end{align*}
for a sequence  $(e_j)_{j \geq 0}\subseteq \cH$ of errors that satisfy $\sum_{i = 0}^\infty \lambda_k \|e_j\| < \infty$ (e.g., using~\cite[Proposition 3.4]{Combettes2014}). The result is straightforward and will only serve to complicate notation, so we omit this extension.
\end{remark}

\begin{remark}\label{rem:FPRslow}
Note that the convergence rates for the fixed-point residual in Part~\ref{thm:convergence:part:convergencerate} of Theorem~\ref{thm:convergence} are sharp---even in the case of the variational Problem~\eqref{eq:mainprob1} with $h = 0$~\cite[Theorem 8]{davis2014convergence}.
\end{remark}

%% file: section_rates.tex
\section{Convergence rates}

In this section, we discuss the convergence rates Algorithm~\ref{alg:basic} under several different assumptions on the regularity of the problem. Section~\ref{sec:theoreticalresults} contains a brief overview of all the convergence rates presented in this section.  For readability, we now summarize all of the convergence results of this section, briefly indicate the proof structure, and 
\iftechreport
  place the formal proofs in the Appendix.
\else
  leave the proofs to the technical report~\cite[Appendix D]{OurTechReport}.
\fi

\subsection{General rates}\label{sec:generalrates}

We establish our most general convergence rates for the following quantities: If $z^\ast$ is a fixed point of $T$,  $x^\ast = J_{\gamma B}(z^\ast)$, and $x \in \cH$, then let
\begin{align*}
\kappa_{1}^k(\lambda, x) &= \|z^k - x\|^2 - \|z^{k+1} - x\|^2 + \left(1 - \frac{2}{\lambda}\right)\|z^k- z^{k+1}\|^2 + 2\gamma\dotp{z^k - z^{k+1}, Cx_B^k}, \\
\kappa_{2}^k(\lambda, x^\ast) &= \|z^k - z^\ast\|^2 - \|z^{k+1} - z^\ast\|^2 + \left(1 - \frac{2}{\lambda}\right)\|z^k- z^{k+1}\|^2 + 2\gamma\dotp{z^k - z^{k+1}, Cx_B^k - Cx^\ast}. 
\end{align*}
In 
\iftechreport
  Theorems~\ref{thm:nonergodicmain}, \ref{thm:ergodic1}, and \ref{thm:ergodic2}
\else
  \cite[Theorems~D.1, D.2, and D.3]{OurTechReport} 
\fi
we deduce the following convergence rates: For $j \in \{1, 2\}$, $x_1 = x$, $x_2 = x^\ast$ and for all $k \geq 0$, we have
\begin{align*}
\text{Nonergodic (Algorithm~\ref{alg:basic}):}  && \kappa_j^k(1, x_j) &= o\left(\frac{1 + \|x_j\|}{\sqrt{k+1}}\right); \\
\text{Ergodic (Algorithm~\ref{alg:basic} \& \eqref{avg1}):}  && \frac{1}{\sum_{i=0}^k \lambda_k}\sum_{i = 0}^k \kappa_j^i(\lambda_i, x_j) &= O\left(\frac{ 1+ \|x_j\|^2}{k+1}\right); \\
\text{Ergodic (Algorithm~\ref{alg:basic} \& \eqref{avg2}):}  && \frac{2}{(k+1)(k+2)}\sum_{i = 0}^k (i+1) \kappa_j^i(\lambda, x_j)  &= O\left(\frac{1 + \|x_j\|^2}{k+1}\right).
\end{align*}
It may be hard to see how these terms relate to the convergence of Algorithm~\ref{alg:basic}.  The key observation 
\iftechreport
  of Proposition~\ref{prop:fundamentalequality}
\else
  of~\cite[Proposition D.2]{OurTechReport}
\fi  
shows that $\kappa_1^k$ is an upper bound for a certain variational inequality associated to Problem~\eqref{eq:mainprob} and that $\kappa_2^k$ bounds the distance of the current iterate $x^k$ (or its averaged variant $\overline{x}^k$ in Equations~\eqref{avg1} and~\eqref{avg2}) to the solution whenever one of the operators is strongly monotone.

The proofs of these convergence rates are straightforward, though technical. The nonergodic rates follow from an application of Part~\ref{thm:convergence:part:convergencerate} of Theorem~\ref{thm:convergence}, which shows that $\|z^{k+1} - z^k\|^2 = o(1/(k+1))$. The ergodic convergence rates follow from the alternating series properties of $\kappa_j^k$ together with the summability of the gradient shown in Part~\ref{thm:convergence:part:gradientsum} of Theorem~\ref{thm:convergence}.

\subsection{Objective error and variational inequalities} \label{sec:nonergodicratesbrief}

In this section, we use the convergence rates of the upper and lower bounds derived 
\iftechreport
  in Theorems~\ref{thm:nonergodicmain}, \ref{thm:ergodic1}, and \ref{thm:ergodic2}
\else
  in~\cite[Theorems D.1, D.2, and D.3]{OurTechReport}
\fi
to deduce convergence rates of function values and variational inequalities. All of the convergence rates have the following orders:
\begin{align*}
\text{Nonergodic: $o\left(\frac{1}{\sqrt{k+1}}\right)$} && \text{and} && \text{Ergodic: $O\left(\frac{1}{k+1}\right)$}.
\end{align*}

The convergence rates in this section generalize some of the known convergence rates provided in~\cite{davis2014convergence,davis2014convergenceprimaldual,davis2014convergenceFDRS} for Douglas-Rachford splitting, forward-Douglas-Rachford splitting, and the primal-dual forward-backward splitting, Douglas-Rachford splitting, and the proximal-point algorithms.

\subsubsection{Nonergodic rates (Algorithm~\ref{alg:basic})}

Suppose that $A = \partial f + \overline{A}$, $B = \partial g + \overline{B}$ and $C = \nabla h + \overline{C}$ where $f, g$ and $h$ are functions and $\overline{A}, \overline{B}$ and $\overline{C}$ are monotone operators. Whenever $f$ and $\overline{A}$ are Lipschitz continuous, the following convergence rate holds:
\begin{align*}
f(x_B^k) + g(x_B^k) + h(x_B^k) - (f+g+h)(x) + \dotp{ x_B^k - x,  \overline{A}x_B^k + u_{\overline{B}}^k + \overline{C} x_B^k} &= o\left(\frac{1+\|x\|}{\sqrt{k+1}}\right). \numberthis\label{eq:variationalinequalityrate}
\end{align*}
A more general rate holds when $f$ and $\overline{A}$ are not necessarily Lipschitz. 
\iftechreport
  See Corollaries~\ref{cor:nonergodicfunction} and \ref{cor:nonergodicvariational}
\else
  See~\cite[Corollaries~D.3 and D.4]{OurTechReport}
\fi
for the exact convergence statements.

Note that quantity on the left hand side of Equation~\eqref{eq:variationalinequalityrate} can be negative. The point $x_B^k$ is a solution to the variational inequality problem if, and only if, the Equation~\eqref{eq:variationalinequalityrate} is negative for all $x \in \cH$, which is why we include the dependence on $\|x\|$.

Notice that when the operators $\overline{A}, \overline{B}$ and $\overline{C}$ vanish and $x = x^\ast$, the convergence rate in \eqref{eq:variationalinequalityrate} reduces to the objective error of the function $f + g + h$ at the point $x_B^k$
\begin{align*}
f(x_B^k) + g(x_B^k) + h(x_B^k) - (f+g+h)(x^*)  &= o\left(\frac{1+\|x^*\|}{\sqrt{k+1}}\right). \numberthis\label{eq:variationalinequalityrate1}
\end{align*}
and we deduce the rate $o(1/\sqrt{k+1})$ for our method. By~\cite[Theorem 11]{davis2014convergence}, this rate is sharp.

Further nonergodic rates can be deduced whenever any $A, B$, or $C$ are $\mu_A$, $\mu_B$ and $\mu_C$-strongly monotone respectively.  In particular, the following two rates hold for all $k \geq 0$ 
\iftechreport
  Corollary~\ref{thm:strong}:
\else
  (\cite[Corollary~D.9]{OurTechReport}):
\fi
\begin{align*}
\mu_A\|x_A^k - x^\ast\|^2 + (\mu_B+\mu_C)\|x_B^k - x^\ast\|^2 &= o\left(\frac{1}{\sqrt{k+1}}\right); \\
\min_{i = 0, \cdots, k}\left\{ \mu_A\|x_A^i - x^\ast\|^2 + (\mu_B+\mu_C)\|x_B^i - x^\ast\|^2 \right\} &= o\left(\frac{1}{k+1}\right).
\end{align*}

\subsubsection{Ergodic Rates}

We use the same set up as Section~\ref{sec:nonergodicratesbrief}, except we assume that $\overline{A}$ and $\overline{B}$ are skew linear mappings (i.e., $A^\ast = -A$ and $B^\ast = -B$) and $\overline{C} = 0$. If $(\overline{x}_B^j)_{j \geq 0}$ is generated as in Equation~\eqref{avg1} or Equation~\eqref{avg2} and $f$ is Lipschitz continuous, the following convergence rate holds:
\begin{align*}
f(\overline{x}_B^k) + g(\overline{x}_B^k) + h(\overline{x}_B^k) - (f+g+h)(x) + \dotp{ \overline{x}_B^k - x,  \overline{A}\overline{x}_B^k + \overline{B}\overline{x}_B^k} &= o\left(\frac{1+\|x\|^2}{k+1}\right). \numberthis\label{eq:ergodicvariationalinequalityrate}
\end{align*}
A more general rate holds when $f$ is not necessarily Lipschitz. 
\iftechreport
  See Corollaries~\ref{cor:ergodic1function}--\ref{cor:ergodic2variational}
\else
  See \cite[Corollaries~D.5--D.8]{OurTechReport}
\fi
for the exact convergence statements.

Further nonergodic rates can be deduced whenever any $A, B$, or $C$ are $\mu_A$, $\mu_B$ and $\mu_C$-strongly monotone respectively.  In particular, the following two rates hold for all $k \geq 0$ 
\iftechreport
  Corollary~\ref{thm:strong}:
\else
  (\cite[Corollary~D.9]{OurTechReport}):
\fi
Let $(\overline{x}_A^j)_{j \geq 0}$ and $(\overline{x}_B^j)_{j \geq 0}$ be generated by Algorithm~\ref{alg:basic} and Equations~\eqref{avg1} or~\eqref{avg2}. Then
\begin{align*}
\mu_A\|\overline{x}_A^k - x^\ast\|^2 + (\mu_B+\mu_C)\|\overline{x}_B^k - x^\ast\|^2 &= O\left(\frac{1}{k+1}\right).
\end{align*}
%
%
\subsection{Improving the objective error with Lipschitz differentiability}

The worst case convergence rate $o(1/\sqrt{k+1})$ for objective error discussed in proved in
\iftechreport
  Corollary~\ref{cor:nonergodicfunction}
\else
  \cite[Corollary~D.3]{OurTechReport} 
\fi
is quite slow.  Although averaging can improve the rate of convergence, this technique does not necessarily translate into better practical performance as discussed in Section~\ref{averaging}. We can deduce a better rate of convergence for the nonergodic iterate, whenever one of the functions $f$ or $g$ has a Lipschitz continuous derivative.  In particular, if $\nabla f$ exists and is Lipschitz, we show 
\iftechreport
  in Proposition~\ref{prop:Lipschitz}
\else
  in \cite[Proposition~D.3]{OurTechReport}
\fi
that the objective error sequence $((f+g + h)(x_B^j) - (f+g+h)(x^\ast))_{j \geq 0}$ is summable. From this, we immediately deduce 
\iftechreport
  Theorem~\ref{thm:lipschitzderivative}
\else
  in \cite[Theorem~D.5]{OurTechReport} 
\fi
the following rate: for all $k \geq 0$, we have
\begin{align*}
\min_{i=0, \cdots, k} \left\{(f+g + h)(x_B^i) - (f+g+h)(x^\ast)\right\} &= o\left(\frac{1}{k+1}\right).
\end{align*}
A similar result holds for the objective error sequence $((f+g + h)(x_A^j) - (f+g+h)(x^\ast))_{j \geq 0}$ when the function $g$ is Lipschitz differentiable. Thus, when $f$ or $g$ is sufficiently regular, the convergence rate of the nonergodic iterate is actually faster than the convergence rate for the ergodic iterate, which motivates its use in practice.

\subsection{Linear convergence}

Whenever $A, B$ and $C$ are sufficiently regular, we can show that the operator $T$ is strictly contractive towards the fixed point set. In particular, Algorithm~\ref{alg:basic} converges linearly whenever $$(\mu_A + \mu_B + \mu_C)(1/L_A + 1/L_B) > 0$$
where $L_A$ and $L_B$ are the Lipschitz constants of $A$ and $B$ respectively and $A, B$, or $C$ are $\mu_A$, $\mu_B$ and $\mu_C$-strongly monotone respectively (where we allow the $L_A = L_B = \mu_A = \mu_B = \mu_C = 0$).

Note that this linear convergence result is the best we can expect in some sense.  Indeed, even if $\mu_C$ and $\mu_A$ are strongly monotone, Algorithm~\ref{alg:basic} will not necessarily converge linearly.
\iftechreport
  Section~\ref{sec:slowconvergence}
\else
  In \cite[Section~D.6]{OurTechReport}
\fi
we provide an example such that
\begin{align*}
\text{$\mu_A\mu_C > 0$, but $\|z^k - z^\ast\|$ converges arbitrarily slowly to $0$}.
\end{align*}

\subsection{Convergence rates for multi-block ADMM}

All of the results in this section imply convergence rates for  Algorithm~\ref{alg:admm}, which is applied to the dual objective in Problem~\eqref{eq:dualprob}. Using the techniques of~\cite[Section 8]{davis2014convergence} and~\cite[Section 6]{davis2014convergenceFaster}, we can easily derive convergence rates of the primal objective in Problem~\eqref{eq:admmprob}. We do not pursue these results in this paper due to lack of space.

%% file: section_numerical.tex
\section{Numerical results}\label{sec:numerical}

In this section, we present some numerical examples of Algorithm~\ref{alg:basic}. We emphasize that to keep our implementations  simple, we did not attempt to optimize the  codes or their parameters for best performance. We also did not attempt to seriously evaluate the prediction ability of the models we tested, which is beyond the scope of this paper. Our Matlab codes will be released online on the authors' websites. All tests were run  on a PC\ with 32GB memory and an Intel i5-3570 CPU with Ubuntu 12.04 and Matlab R2011b installed.

\subsection{Image inpainting with texture completion}\label{sec:numerical:sub:tensor}
This section presents the results of applying Problem \eqref{eq:inpainting} to the color images\footnote{We are grateful of Professor Ji Liu for sharing his data in \cite{LiuMusialskiWonkaYe2013} with us.} of a building, parts of which are manually occluded with  white colors. See
Figure \ref{fig:imageinpainting}. The images have a $517\times493$ resolution and three color channels. At each iteration of Algorithm~\ref{alg:basic}, the SVDs  of two matrices of sizes $517\times1479$ and $1551\times493$ consume most of the computing time. However, it  took less 150 iterations to return good recoveries.
\begin{figure}
\centering
\subcaptionbox{%
    Original image%
    \label{subfig:sublabel1}%
}
[%
    0.25\textwidth 
]%
{%
    \includegraphics[width=0.25\textwidth]%
    {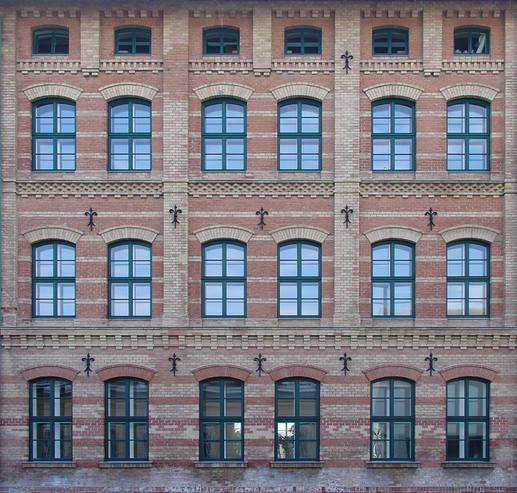}%
}%
\hspace{0.03\textwidth} 
\subcaptionbox{%
    Occluded image 1
    \label{subfig:sublabel2}%
}
[%
    0.25\textwidth 
]%
{%
    \includegraphics[width=0.25\textwidth]%
    {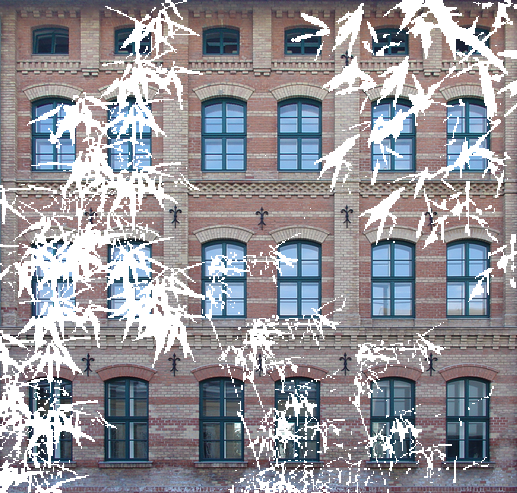}%
}%
\hspace{0.03\textwidth} 
\subcaptionbox{%
    Occluded image 2%
    \label{subfig:sublabel3}%
}
[%
    0.25\textwidth 
]%
{%
    \includegraphics[width=0.25\textwidth]%
    {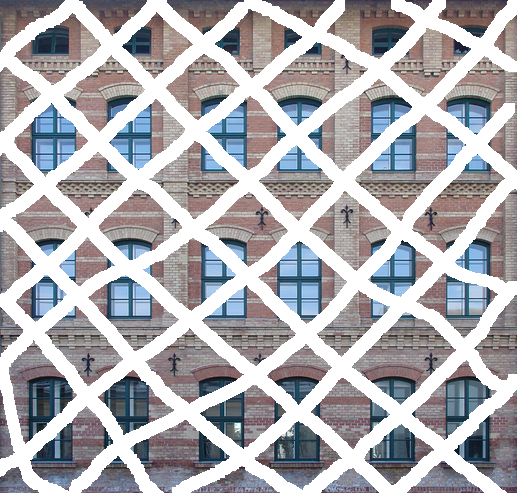}%
}%
\vspace{.03\textwidth}
\subcaptionbox{%
    Recovered image 1
    \label{subfig:sublabel21}%
}
[%
   0.25\textwidth 
]%
{%
    \includegraphics[width=0.25\textwidth]%
    {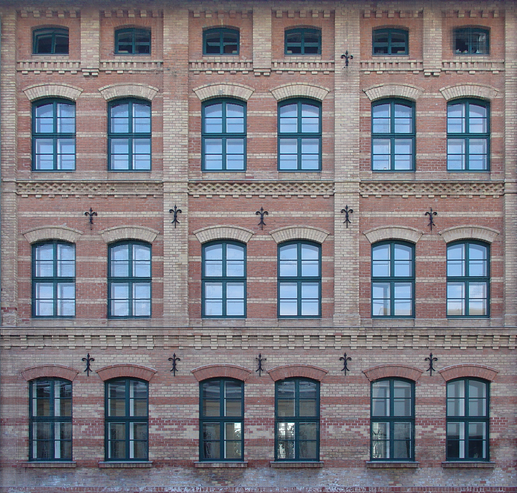}%
}%
\hspace{0.03\textwidth}
\subcaptionbox{%
    Recovered image 2
    \label{subfig:sublabel4}%
}
[%
   0.25\textwidth 
]%
{%
    \includegraphics[width=0.25\textwidth]%
    {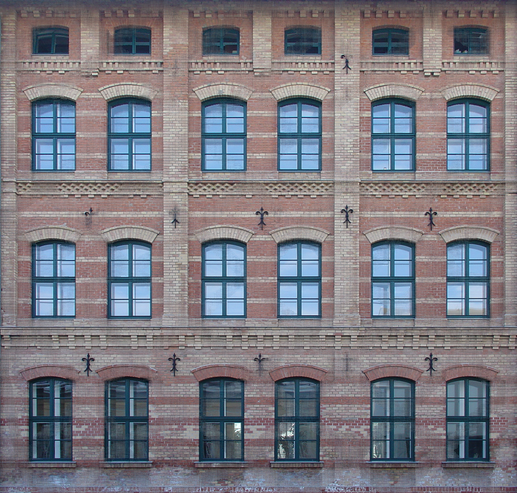}%
}%
\caption[Short Caption]{Images recovered by solving the  tensor completion Problem~\eqref{eq:inpainting} using Algorithm~\ref{alg:basic} for two different types of occlusions.}
\label{fig:imageinpainting}
\end{figure}

\subsection{Matrix completion for movie recommendations}\label{sec:numerical:sub:matcompletion}

In this section, we apply Problem~\eqref{eq:matrixcompletion} to a movie recommendation dataset. In this example, each row of $X_0\in \vR^{m \times n}$ corresponds to a user and each column corresponds to a movie, and for all $i =1, \cdots, m$ and $j = 1, \cdots, m$, the matrix entry $(X_0)_{ij}$ is the ranking that user $i$ gave to movie $j$.

We use the MovieLens-1M~\cite{movielens} dataset for evaluation. This dataset consists of $1000209$ observations of the matrix $X_0 \in \vR^{6040 \times 3952}$. We plot our numerical results in Figure~\ref{fig:matrixcompletion}. In our code we set $l = 0$, $u = 5$ and  solved the problem with different choices of $\mu$ in order to achieve solutions of desired rank. In Figure~\ref{fig:matrixcompletion:subfig:RMSE} we plot the root mean-square error
\begin{align}\label{eq:RMSE}
\frac{\|\cA(X_g^k - X_0)\|_F}{\sqrt{1000209}},
\end{align} which does not  decrease to zero, but represents how closely the current iterate fits the observed data.

The code runs fairly quick for the scale of the data. The main bottleneck in this algorithm is evaluating the proximal operator of $\|\cdot \|_\ast$, which requires computing the SVD of a $6040 \times 3952$ size matrix.

\begin{figure}
\centering
\subcaptionbox{%
    Fixed-point residual at iteration $k$.%
    \label{fig:matrixcompletionsubfig:FPR}%
}
[%
    0.45\textwidth 
]%
{%
    \includegraphics[width=0.45\textwidth]%
    {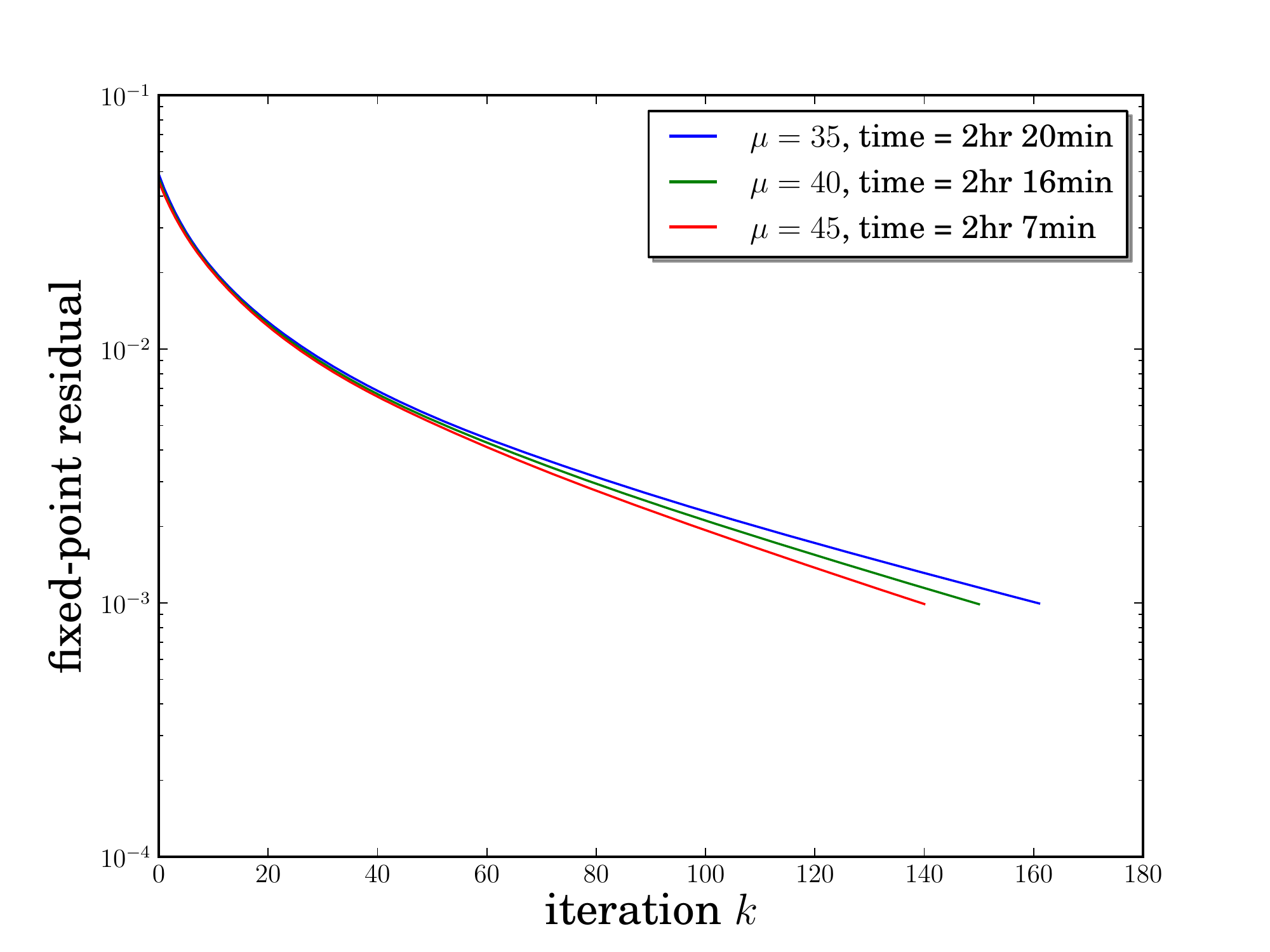}%
}%
\hspace{0.03\textwidth} 
\subcaptionbox{%
    Rank at iteration $k$.
    \label{fig:matrixcompletion:subfig:rank}%
}
[%
    0.45\textwidth 
]%
{%
    \includegraphics[width=0.45\textwidth]%
    {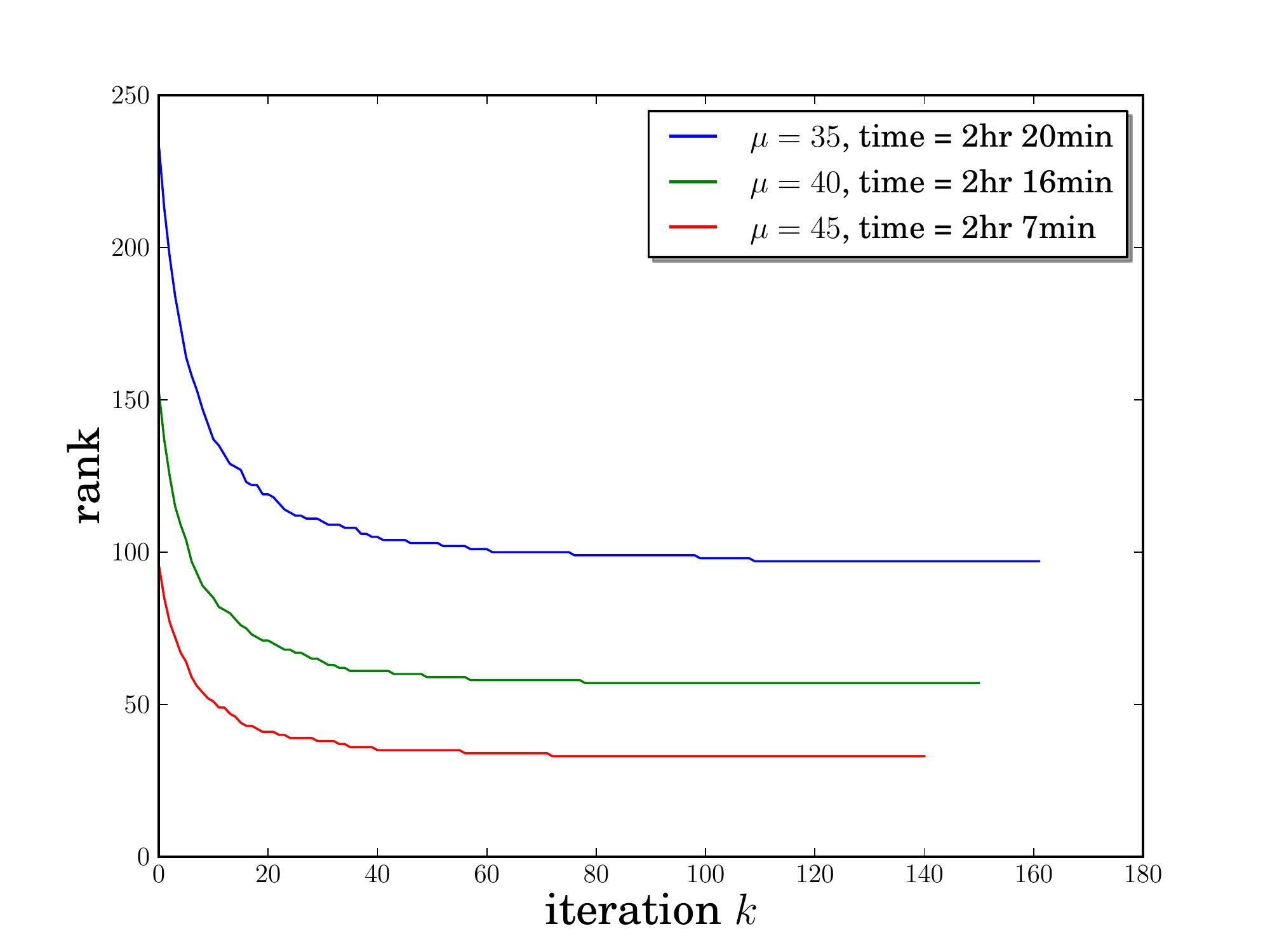}%
}%
\hspace{0.03\textwidth} 
\subcaptionbox{%
    Root mean square error (Equation~\eqref{eq:RMSE}) at iteration $k$.%
    \label{fig:matrixcompletion:subfig:RMSE}%
}
[%
    0.45\textwidth 
]%
{%
    \includegraphics[width=0.45\textwidth]%
    {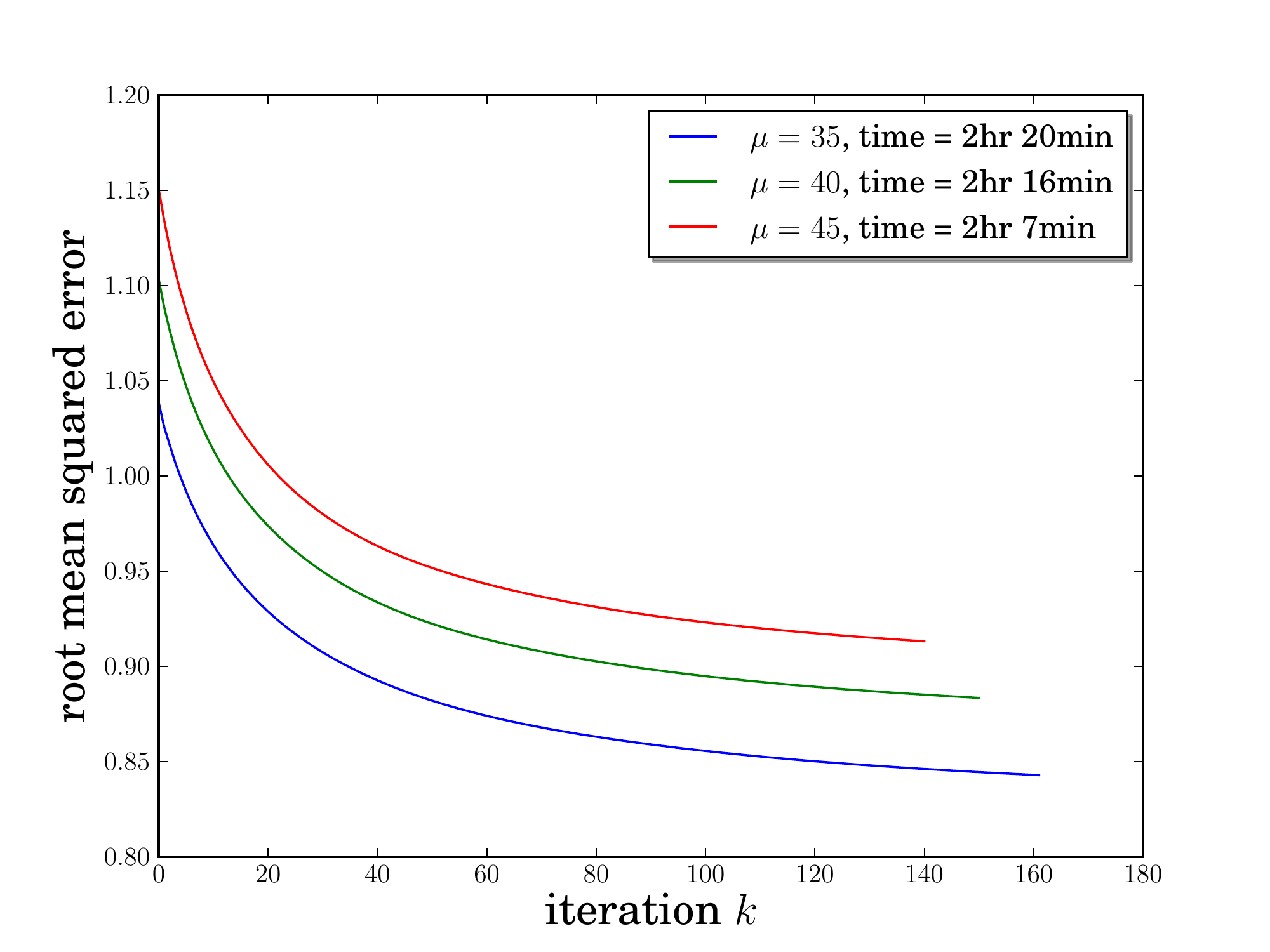}%
}%
\caption[Short Caption]{Run time and convergence rate statistics for the matrix completion Problem~\eqref{eq:matrixcompletion} on the MovieLens-1M database~\cite{movielens}.}
\label{fig:matrixcompletion}
\end{figure}

\subsection{Support vector machine classification}\label{sec:numerical:sub:SVM}
In support vector machine classification we have a kernel matrix $K \in \vR^{d \times d}$ generated from a training set $X = \{t_1, \cdots, t_d\}$ using a kernel function $\cK : X \times X \rightarrow \vR$: for all $i, j =1, \cdots, d$, we have $K_{i,j} = \cK(t_i,t_j)$. In our particular example, $X \subseteq \vR^n$ for some $n > 0$ and $\cK_\sigma : \vR^n \times \vR^n \rightarrow \vR_{++}$ is the Gaussian kernel given by $\cK_{\sigma}(t, t') = e^{-\sigma \|t - t'\|^2}$ for some $\sigma > 0$. We are also given a label vector $y \in \{-1, 1\}^d$, which indicates the label given to each point in $X$. Finally, we are given a real number $C > 0$ that controls how much we let our final classifier stray from perfect classification on the training set $X$.

We define constraint sets $\cC_1 = \{0 \leq x \leq C\}$ and $\cC_2 = \{ x \in \vR^d \mid \dotp{y, x } = 0\}$. We also define $Q_0 = \mathrm{diag}(y) K \mathrm{diag}(y)$. Then the solution to Problem~\eqref{eq:cqprogramming} with $Q = Q_0$ is precisely the dual form soft-margin SVM classifier~\cite{cortes1995support}. Unfortunately, the Lipschitz constant of $Q_0$ is often quite large (i.e., $\gamma$ must be small), which results in poor practical performance. Thus, to improve practical we solve Problem~\eqref{eq:cqprogramming} with $Q = P_{\cC_2} Q_0 P_{\cC_2}$, which is  equivalent to the original problem because the minimizer must lie in $\cC_2$. The result is a much smaller Lipschitz constant for $Q$ and better practical performance. This trick was first reported in~\cite[Section 1.6]{davis2014convergenceFDRS}.

We evaluated our algorithm on a subset $X_{\text{all}}$ of the UCI ``Adult" machine learning dataset which is entitled ``a7a" and is available from the LIBSVM website~\cite{CC01a}. Our training set $X_{\text{train}}$ consisted of a $d = 9660$ element subsample of this $16100$ element training set (i.e., a $60\%$ sample). Note that $Q$ has $d^2 = 9660^2 = 93315600$ nonzero entries. In table~\ref{table:SVM}, we trained the SVM model~\eqref{eq:cqprogramming} with different choices of parameters $C$ and $\sigma$, and then evaluated their prediction accuracy on the remaining $16100 - 9660 = 6440$ elements in $X_{\text{test}} = X_{\text{all}} \backslash X_{\text{train}}$. We found that the parameters $\sigma = 2^{-3}$ and $C = 1$ gave the best performance on the test set, so we set these to be the parameters for our numerical experiments.

Figure~\ref{fig:SVM} plots the results of our test. Figures~\ref{fig:SVM:subfig:LS} and~\ref{fig:SVM:subfig:sublabel:LSobj} compare the line search method in Algorithm~\ref{alg:linesearch} with the basic Algorithm~\ref{alg:basic}. We see that the line search method performs better than the basic algorithm in terms of number of iterations and total CPU time needed to reach a desired accuracy. Because of the linearity of the projection $P_{\cC_2}$, we can find a closed form solution for the line search weight $\rho$ in Algorithm~\ref{fig:SVM:subfig:LS} as the root of a third degree polynomial. Thus, although Algorithm~\ref{alg:linesearch} requires more work per iteration than Algorithm~\ref{alg:basic}, it still takes less time overall because Algorithm~\ref{alg:basic} must compute $\beta = 1/\|Q\|$, which is quite costly.

Finally, in Figure~\ref{fig:SVM:subfig:ergVSnonerg} we compare the performance of the nonergodic iterate generated by Algorithm~\ref{alg:basic}, the standard ergodic iterate~\eqref{avg1}, and the newly introduced ergodic iterate~\eqref{avg2}. We see that the nonergodic iterate performs better than the other two, and as expected, the the new ergodic iterate outperforms the standard ergodic iterate. We emphasize that computing these iterates is essentially costless for the user and only modifies the final output of the algorithm, not the trajectory.

We emphasize that all steps in this algorithm can be computed in closed form, so implementation is easy and each iteration is quite cheap.

\begin{figure}
\centering
\subcaptionbox{%
    Fixed-point residual with and without line search (LS).%
    \label{fig:SVM:subfig:LS}%
}
[%
    0.45\textwidth 
]%
{%
    \includegraphics[width=0.45\textwidth]%
    {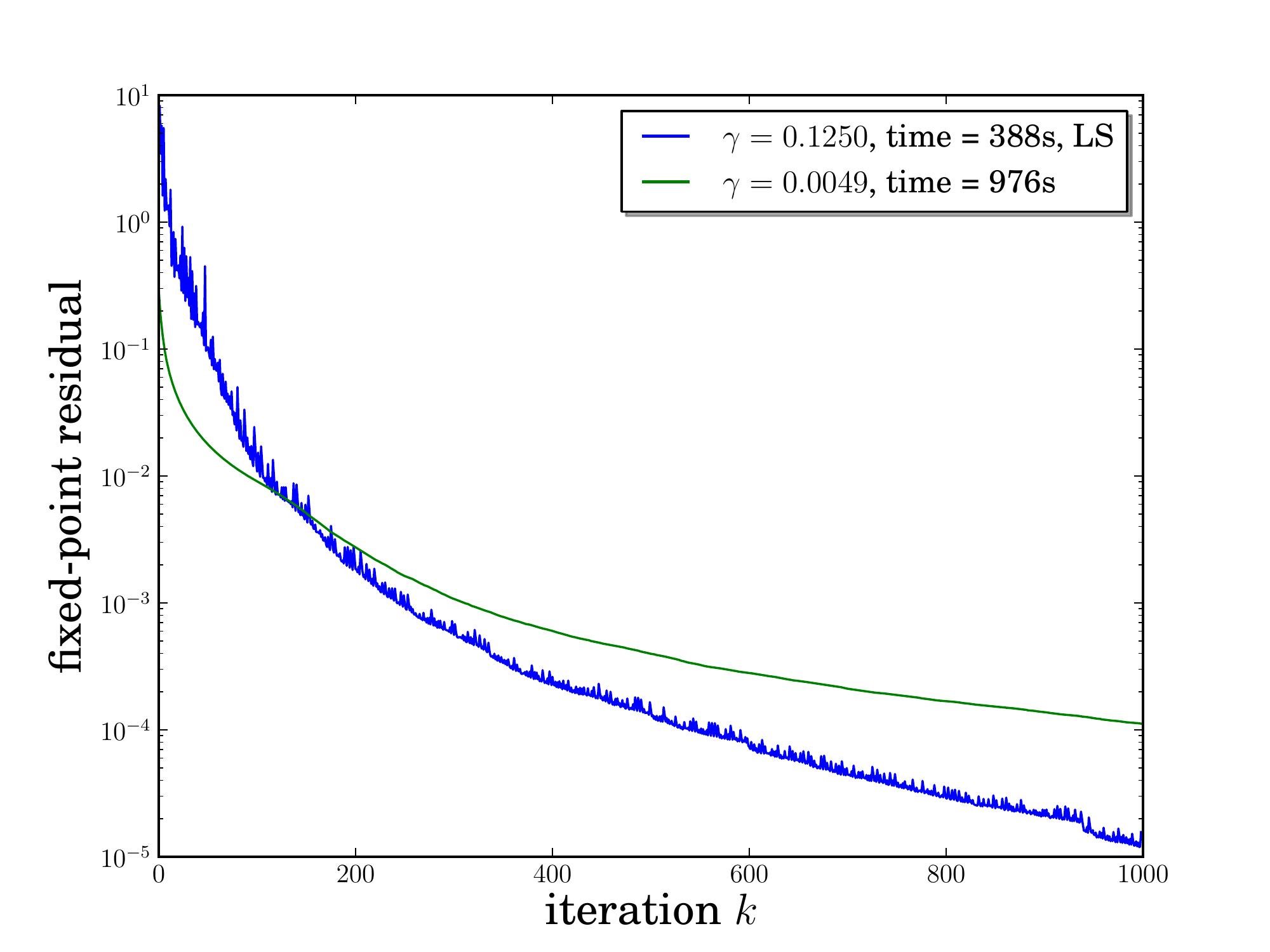}%
}%
\hspace{0.03\textwidth} 
\subcaptionbox{%
    Objective value with and without line search (LS).
    \label{fig:SVM:subfig:sublabel:LSobj}%
}
[%
    0.45\textwidth 
]%
{%
    \includegraphics[width=0.45\textwidth]%
    {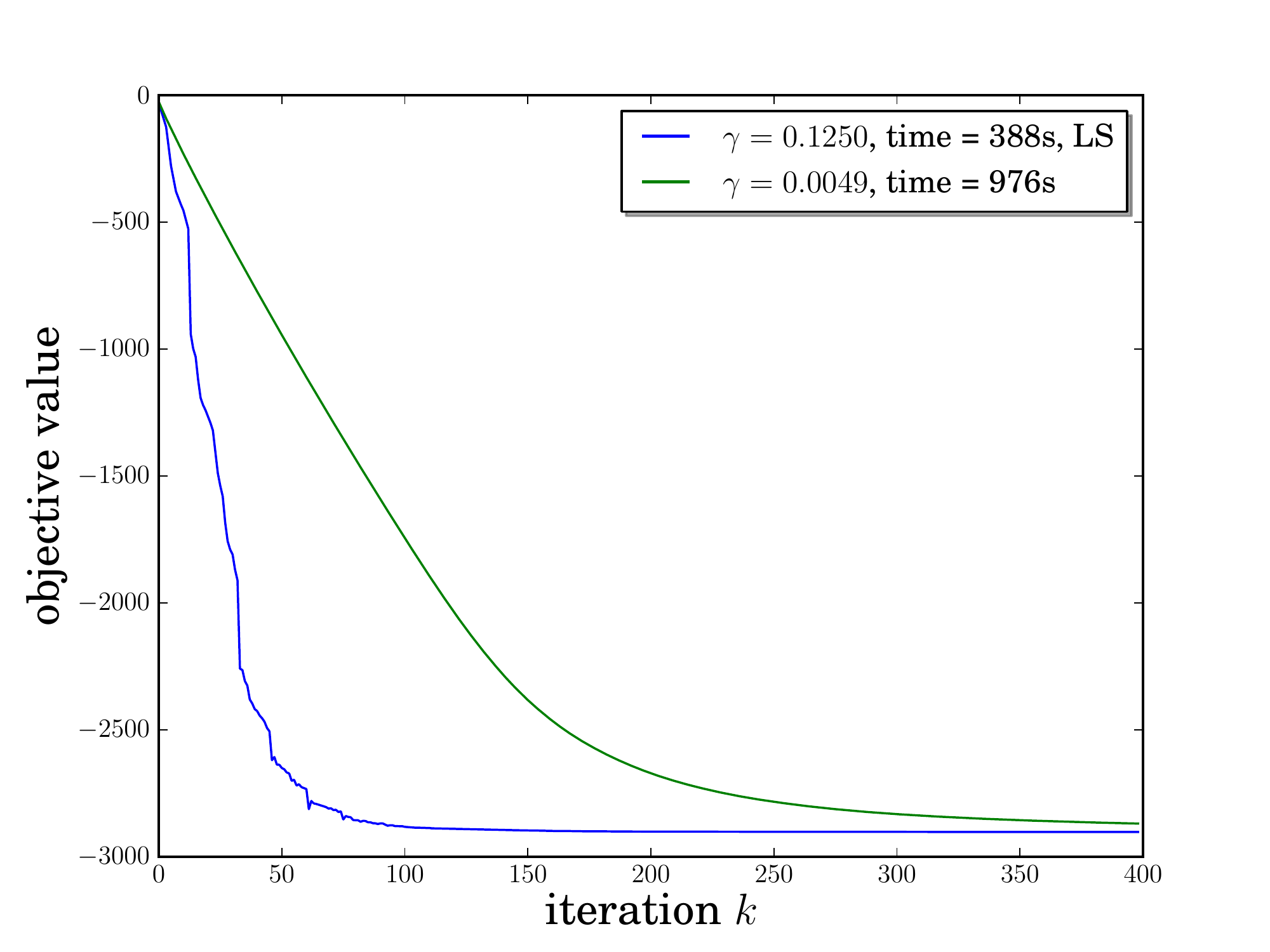}%
}%
\hspace{0.03\textwidth} 
\subcaptionbox{%
    Comparison of ergodic and nonergodic iterates.%
    \label{fig:SVM:subfig:ergVSnonerg}%
}
[%
    0.45\textwidth 
]%
{%
    \includegraphics[width=0.45\textwidth]%
    {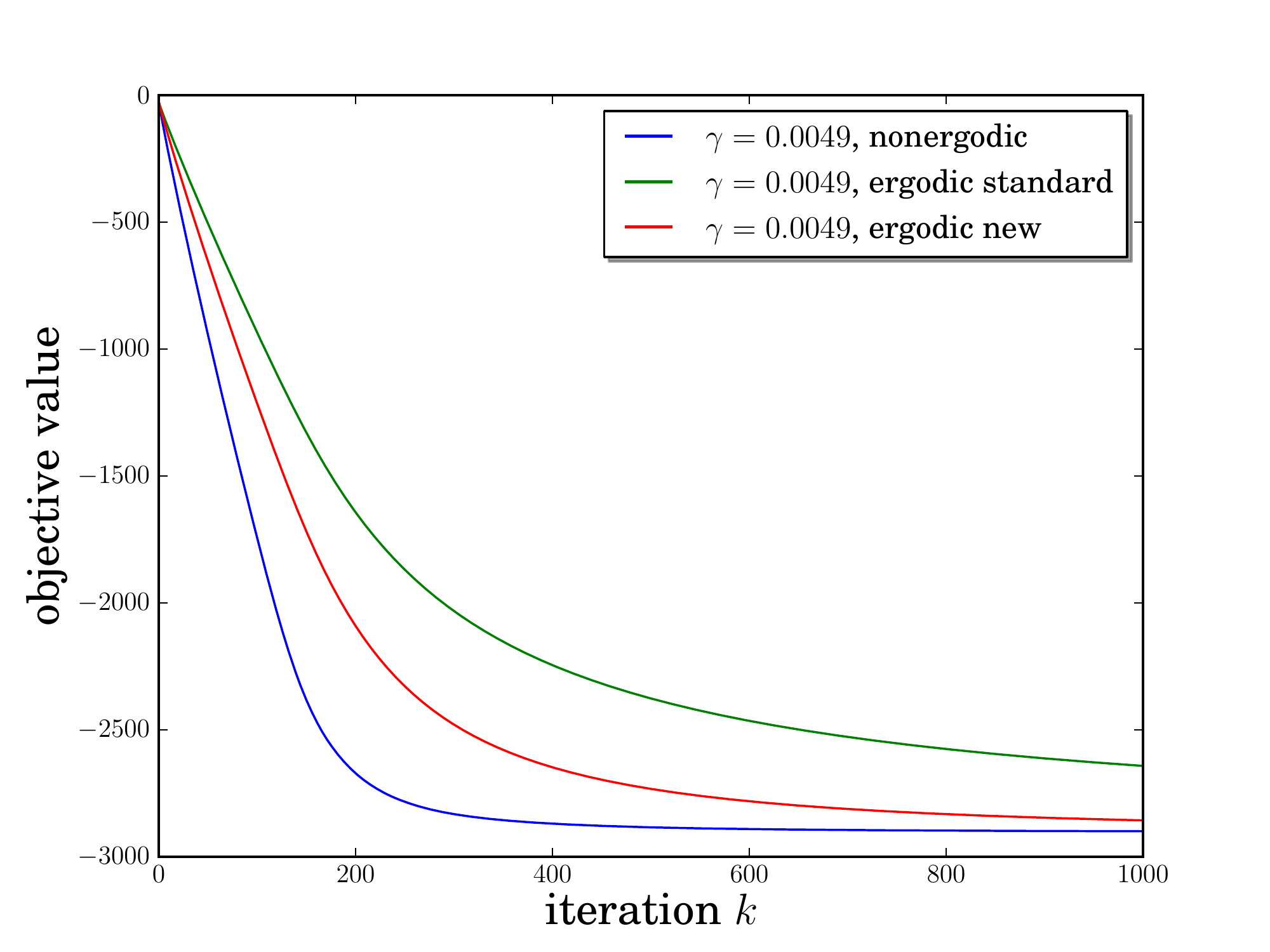}%
}%
\caption[Short Caption]{Run time and convergence rate statistics for the SVM Problem~\eqref{eq:cqprogramming} on the UCI ``Adult" Machine learning dataset~\cite{Lichman:2013}.  Results are with the parameter choice that has the best generalization to the test set $(C, \sigma) = (1,0.2^{-3})$.}
\label{fig:SVM}
\end{figure}

\begin{center}
\begin{table}
\centering
    \begin{tabular}{lllll}
    \toprule
     $C$ & \multicolumn{4}{c}{kernel parameter $\sigma$} \\
\toprule
 & $2^{-5}$ & $2^{-3}$ & $2^{-1}$ & $2$  \\ \toprule
    $1$ & $0.82689$ & \cellcolor{blue!25} $0.83636$ & $0.82782$ & $0.7755$  \\\midrule
    $2^2$ & $0.82658$ & $0.82441$ & $0.81742$ & $0.7755$ \\  \midrule
    $2^4$ & $0.83465$ & $0.81835$ & $0.8168$ &$0.7755$ \\ \midrule
    $2^6$ & $0.83465$ & $0.81835$ & $0.80795$ & $0.7755$ \\
    \bottomrule
    \end{tabular}
 \caption{Classification accuracy for different choices of $C$ and $\sigma$ in the SVM model.}\label{table:SVM}
\end{table}

\end{center}

\subsection{Portfolio optimization}\label{sec:numerical:sub:PO}

In this section, we evaluate our algorithm on the portfolio optimization problem. In this problem, we have a choice to invest in $d > 0$ assets and our goal is to choose how to distribute our resources among all the assets so that we minimize investment risk, and guarantee that our expected return on the investments is greater than $r \geq 0$. Mathematically, we model the distribution of our assets with a vector $x \in \vR^d$ where $x_i$ represents the percentage of our resources that we invest in asset $i$. For this reason, we define our constraint set $\cC_1 = \{x \in \vR^d \mid \sum_{i=1}^n x_i = 1, x_i \geq 0\}$ to be the standard simplex. We also assume that we are given a vector of mean returns $m \in \vR^d$ where $m_i$ represents the expected return from asset $i$, and we define $\cC_2 = \{x \in \vR^d \mid \dotp{m, x} \geq r\}$. Typically, we model the risk with a matrix $Q_0 \in \vR^{d\times d}$, which is usually chosen as the covariance matrix of asset returns. However, we stray from the typical model by setting $Q = Q_0 + \mu I_{\vR^d}$ for some $\mu \geq 0$, {which has the effect of encouraging diversity of investments among the assets}. In order to choose our optimal investment strategy, we solve Problem~\eqref{eq:cqprogramming} with $Q, ~\cC_1$ and $\cC_2$ introduced here.

In our numerical experiments, we solve a $d= 1000$ dimensional portfolio optimization problem with a randomly generated covariance matrix $Q_0$ (using the Matlab ``gallery" function) and mean return vector $m$.  We report our results in Figure~\ref{fig:PO}. In order to get an estimate of the solution of Problem~\eqref{eq:cqprogramming}, we first solved this problem to high-accuracy using an interior point solver.

The matrix $Q$ in this example is positive definite for any choice of $\mu \geq 0$, but the condition number of $Q_0$ is around $8000$, while the condition number of $Q$ with $\mu = .1$ is around $5$. For this reason, we see a huge improvement in Figure~\ref{fig:PO:subfig:mup1} with the acceleration in Algorithm~\ref{alg:accl}, while in the case $\mu = 0$ in Figure~\ref{fig:PO:subfig:mu0}, the accelerated and non accelerated versions are nearly identical.

We emphasize that all steps in this algorithm can be computed in nearly closed form, so implementation is easy and each iteration is quite cheap.

\begin{figure}
\centering
\subcaptionbox{%
    Distance to solution using accelerated and non accelerated methods.%
    \label{fig:PO:subfig:mup1}%
}
[%
    0.45\textwidth 
]%
{%
    \includegraphics[width=0.45\textwidth]%
    {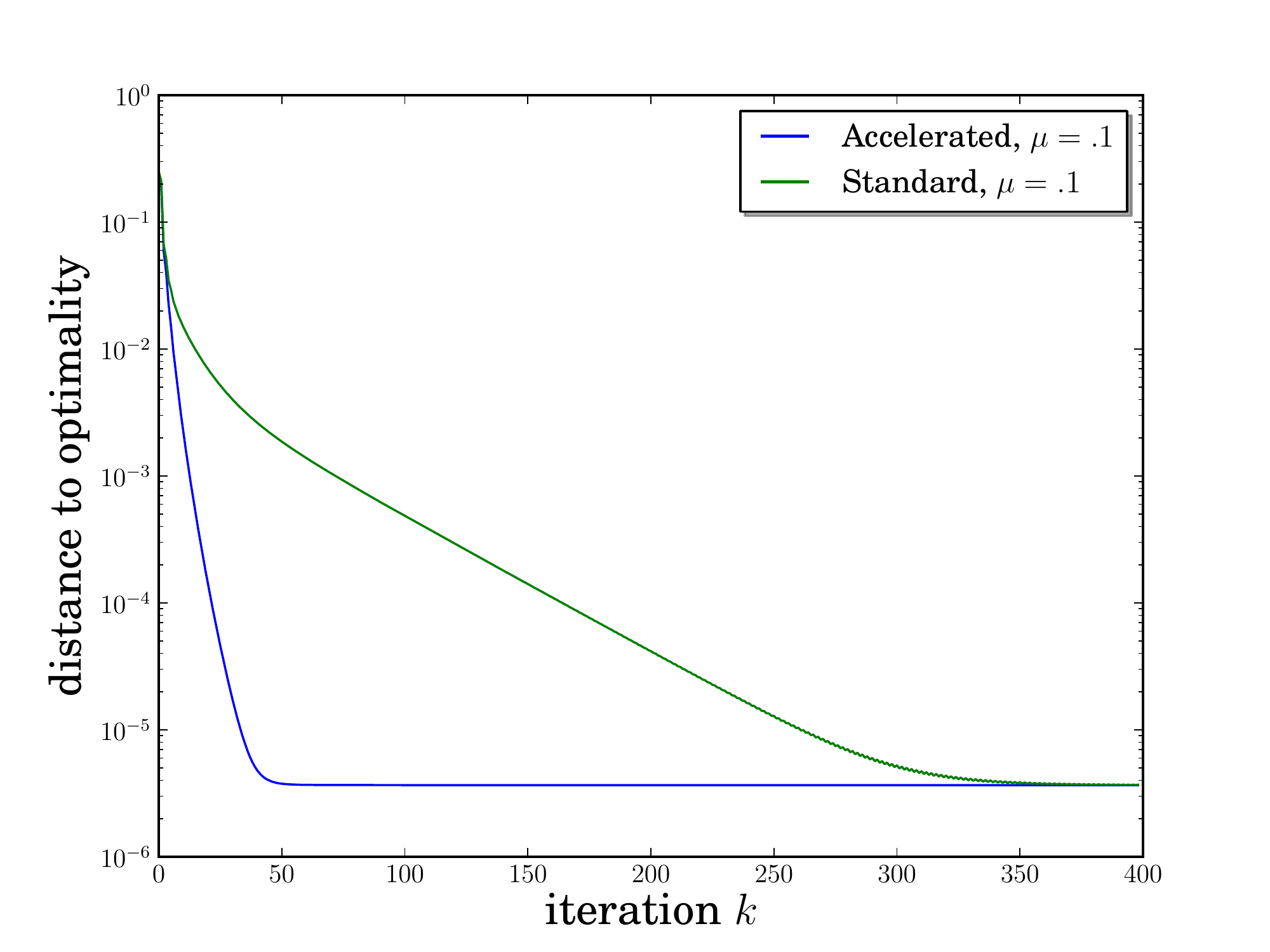}%
}%
\hspace{0.03\textwidth} 
\subcaptionbox{%
    Objective value of accelerated and non accelerated methods. The blue curve is covered by the blue curve.
    \label{fig:PO:subfig:mu0}%
}
[%
    0.45\textwidth 
]%
{%
    \includegraphics[width=0.45\textwidth]%
    {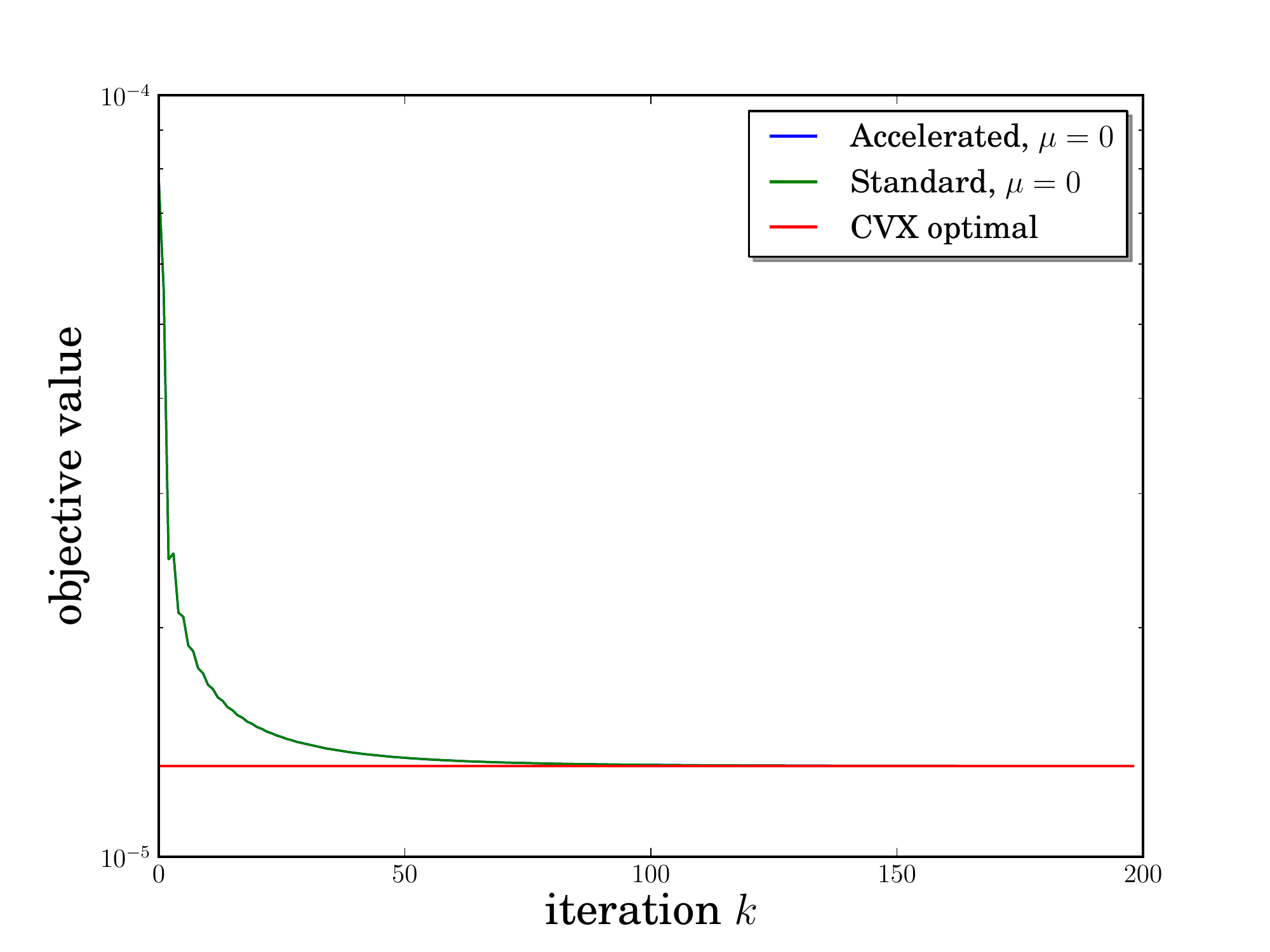}%
}%
\caption[Short Caption]{Convergence rate statistics for the portfolio optimization problem in Section~\ref{sec:numerical:sub:PO}.}
\label{fig:PO}
\end{figure}

%% file: conclusion.tex
\section{Conclusion}
In this paper, we introduced a new operator-splitting algorithm for the three-operator monotone inclusion problem, which has a large variety of applications. We  showed how to accelerate the algorithm whenever one of the involved operators is strongly monotone, and we also introduced a line search procedure and two averaging strategies that can improve the convergence rate. We characterized the convergence rate of the algorithm under various scenarios and showed that many of our rates are sharp. Finally, we introduced numerous applications of the algorithm and showed how it unifies many existing splitting schemes.

%% file: section_appendix.tex
\section{Proof of Theorem~\ref{thm:accl}}\label{appdx:accl}

We first prove a useful inequality.

\begin{proposition}\label{prop:acclinequality}
Let $B$ be $\mu_B$-strongly monotone where we allow the case $\mu_B = 0$. Suppose that $x_A^0 \in \cH$ and set $x_B^0 = J_{\gamma_0 B}(x_A^0), u_B^0 = (1/\gamma_0)(I - J_{\gamma B})(x_A^0)$. For all $k \geq 0$, let
\begin{align}\label{eq:acceleratedalgorithm}
\begin{cases}
x_B^{k+1} = J_{\gamma_k B}(x_A^k + \gamma_ku_B^k); \\
u_B^{k+1} = (1/\gamma_k)(x_A^{k} + \gamma_ku_B^k - x_B^{k+1}); \\
x_A^{k+1} = J_{\gamma_{k+1}A}(x_B^{k+1} - \gamma_{k+1}u_B^{k+1} - \gamma_{k+1} Cx_B^{k+1}).
\end{cases}
\end{align}
\begin{enumerate}
\item \label{prop:acclinequality:part:coco} Suppose that $C$ is $\beta$-cocoercive and $\mu_C$-strongly monotone. Let $\eta \in (0, 1)$ and let $(\gamma_j)_{j \geq 0} \subseteq (0, 2(1-\eta)\beta)$. Then the following inequality holds for all $k \geq 0$:
\begin{align*}
&(1+2\gamma_k\mu_B)\|x_B^{k+1} - x^\ast\|^2 + \gamma_k^2\|u_B^{k+1} - u_B^\ast\|^2  + \left(1 - \frac{\gamma_k}{2(1-\eta)\beta}\right)\|x_A^k - x_B^k\|^2\\
&\leq (1- 2\gamma_k\mu_C\eta)\|x_B^k - x^\ast\|^2  + \gamma_k^2\|u_B^k - u_B^\ast\|^2. \numberthis \label{eq:acclinequality}
\end{align*}
\item \label{prop:acclinequality:part:Lipschitz} Suppose that $C$ is $L_C$-Lipschitz, but not necessarily strongly monotone. In addition, suppose that $\mu_B > 0$.  Then the following inequality holds for all $k \geq 0$:
\begin{align*}
&(1+2\gamma_k(\mu_B - \gamma_k L_C^2/2))\|x_B^{k+1} - x^\ast\|^2 + \gamma_k^2L_C^2\|x_B^{k+1} - x^\ast\|^2 + \gamma_k^2\|u_B^{k+1} - u_B^\ast\|^2 \\
&\leq \|x_B^k - x^\ast\|^2  + \gamma_k^2L_C^2\|x_B^k - x^\ast\|^2+ \gamma_k^2\|u_B^k - u_B^\ast\|^2. \numberthis \label{eq:acclinequalityLipschitz}
\end{align*}
\end{enumerate}
\end{proposition}
\begin{proof}
Fix $k \geq 0$.

Part~\ref{prop:acclinequality:part:coco}: Following Fig. \ref{fig:DYSTR} and Lemma \ref{lem:identities}, let
\begin{align*}
u_A^k = \frac{1}{\gamma_{k}}((x_B^k - \gamma_{k}u_B^{k} - \gamma_{k} Cx_B^{k}) - J_{\gamma_{k}A}(x_B^k - \gamma_{k}u_B^{k} - \gamma_{k} Cx_B^{k})) \in Au_A^k.
\end{align*}
In addition, $u_B^k \in Bu_B^k$ for all $k \geq 0$.  The following identities from Fig. \ref{fig:DYSTR} will be useful in the proof:
\begin{align*}
x_A^k - x_B^{k+1} &= \gamma_k(u_B^{k+1} - u_B^k) \\
x_B^k - x_B^{k+1} &= \gamma_k (u_B^{k+1} + Cx_B^{k} + u_A^{k}) \\
x_B^k - x_A^{k} &= \gamma_k(u_B^{k} + Cx_B^k + u_A^k).
\end{align*}
First we bound the sum of two inner product terms.
\begin{align*}
&2\gamma_k\left( \dotp{x_A^k - x^\ast, u_A^k + Cx_B^k} +  \dotp{x_B^{k+1} - x^\ast, u_B^{k+1}}\right) \\
&=2\gamma_k \left( \dotp{x_A^k - x_B^{k+1}, u_A^k + Cx_B^{k}} + \dotp{x_B^{k+1} - x^\ast, u_B^{k+1} + u_A^k + Cx_B^k} \right) \\
&= 2\gamma_k\left(\dotp{x_A^k - x_B^{k+1}, u_A^k + Cx_B^{k} + u_B^{k}} + \dotp{ x_A^k - x_B^{k+1}, u_B^k}\right) + 2\dotp{x_B^{k+1} - x^\ast, x_B^k - x_B^{k+1}} \\
&= 2\dotp{x_A^{k} - x_B^{k+1}, x_B^k - x_A^k} + 2\dotp{x_B^{k+1} - x^\ast, x_B^k - x_B^{k+1}} + 2\gamma_k\dotp{x_A^k - x_B^{k+1}, u_B^{k} - u_B^\ast} \\
&+ 2\gamma_k\dotp{x_A^k - x_B^{k+1}, u_B^\ast}\\
&= \|x_B^k - x_B^{k+1}\|^2 - \|x_A^k - x_B^{k+1}\|^2 - \|x_A^k - x_B^k\|^2  \\
&+ \|x_B^k - x^\ast\|^2 - \|x_B^{k+1} - x^\ast\|^2 - \|x_B^k - x_B^{k+1}\|^2 \\
&+  2\gamma_k^2\dotp{u_B^{k} - u_B^{k+1}, u_B^{k} - u_B^\ast} +  2\gamma_k\dotp{x_A^k - x_B^{k+1}, u_B^\ast}\\
&= \|x_B^k - x^\ast\|^2 - \|x_B^{k+1} - x^\ast\|^2 - \|x_A^{k} - x_B^{k+1}\|^2 - \|x_A^k - x_B^{k}\|^2\\
&+  \gamma_k^2\left(\|u_B^{k} - u_B^\ast\|^2 - \|u_B^{k+1} - u_B^\ast\|^2 + \|u_B^k - u_B^{k+1}\|^2\right) +  2\gamma_k\dotp{x_B^{k+1} - x_A^{k}, u_B^\ast} \\
&= \|x_B^k - x^\ast\|^2 - \|x_B^{k+1} - x^\ast\|^2 - \|x_A^k - x_B^{k}\|^2 \\
&+ \gamma_k^2\|u_B^k - u_B^\ast\|^2- \gamma_k^2\|u_B^{k+1} - u_B^\ast\|^2 + 2\gamma_k\dotp{x_B^{k+1} - x_A^{k}, u_B^\ast}.\numberthis\label{eq:acclupperbound}
\end{align*}
Furthermore, we have the lower bound
\begin{align*}
&2\gamma_k\left( \dotp{x_A^k - x^\ast, u_A^k + Cx_B^k} +  \dotp{x_B^{k+1} - x^\ast, u_B^{k+1}}\right) \\
&\geq 2\gamma_k\left( \dotp{x_A^k - x^\ast, u_A^\ast + Cx_B^k} +  \dotp{x_B^{k+1} - x^\ast, u_B^{\ast}}\right)  +  2\gamma_k\mu_B\|x_B^{k+1} - x^\ast\|^2. \numberthis\label{eq:lowerboundtomodify}
\end{align*}
We have the further lower bound: For all $\eta \in (0, 1)$, we have
\begin{align*}
2\dotp{x_A^k - x^\ast,  Cx_B^k} & = 2\dotp{x_A^k - x_B^k, Cx_B^k - Cx^\ast} + 2\dotp{ x_A^k - x_B^k, Cx^\ast} + 2\dotp{x_B^k - x^\ast, Cx_B^k} \\
&\geq -\frac{1}{2\beta(1-\eta)} \|x_A^k - x_B^k\|^2 - 2\beta(1-\eta) \|Cx_B^k - Cx^\ast\|^2  + 2\mu_C\eta\|x_B^k - x^\ast\|^2  \\
&+ 2\beta(1-\eta)\|Cx_B^k - Cx^\ast\|^2+ 2\dotp{x_A^k - x^\ast, Cx^\ast}. \numberthis \label{eq:pivotalequationaccl}
\end{align*}
Altogether, we have
\begin{align*}
&2\gamma_k\left( \dotp{x_A^k - x^\ast, u_A^k + Cx_B^k} +  \dotp{x_B^{k+1} - x^\ast, u_B^{k+1}}\right) \\
&\geq 2\gamma_k \left(\dotp{x_A^k - x^\ast, u_A^\ast + Cx^\ast} + \dotp{x_B^{k+1} - x^\ast, u_B^\ast}\right)\\
&-\frac{\gamma_k}{2(1-\eta)\beta} \|x_A^k - x_B^k\|^2 + 2\gamma_k\mu_C(1-\eta)\|x_B^k - x^\ast\|^2 + 2\gamma_k\mu_B\|x_B^{k+1} - x^\ast\|^2\\
&= 2\gamma_k\dotp{x_B^{k+1} - x_A^k, u_B^\ast}-\frac{\gamma_k}{2(1-\eta)\beta} \|x_A^k - x_B^k\|^2 + 2\gamma_k\mu_C\eta\|x_B^k - x^\ast\|^2  + 2\gamma_k\mu_B\|x_B^{k+1} - x^\ast\|^2. \numberthis \label{eq:accllowbound}
\end{align*}
Thus, combine~\eqref{eq:acclupperbound} and~\eqref{eq:accllowbound} to get
\begin{align*}
&(1+2\gamma_k\mu_B)\|x_B^{k+1} - x^\ast\|^2  + \gamma_k^2\|u_B^{k+1} - u_B^\ast\|^2  + \left(1 - \frac{\gamma_k}{2(1-\eta)\beta}\right)\|x_A^k - x_B^k\|^2\\
&\leq (1- 2\gamma_k\mu_C\eta)\|x_B^k - x^\ast\|^2  + \gamma_k^2\|u_B^k - u_B^\ast\|^2.
\end{align*}

Part~\ref{prop:acclinequality:part:Lipschitz}: This follows the exact same reasoning, except we replace Equation~\eqref{eq:pivotalequationaccl} with the following lower bound:
\begin{align*}
2\dotp{x_A^k - x^\ast,  Cx_B^k} & = 2\dotp{x_A^k - x_B^k, Cx_B^k - Cx^\ast} + 2\dotp{ x_A^k - x_B^k, Cx^\ast} + 2\dotp{x_B^k - x^\ast, Cx_B^k} \\
&\geq -\frac{1}{\gamma_k} \|x_A^k - x_B^k\|^2 - \gamma_k\|Cx_B^k - Cx^\ast\|^2 + 2\dotp{x_A^k - x^\ast, Cx^\ast} \\
&\geq -\frac{1}{\gamma_k} \|x_A^k - x_B^k\|^2 - \gamma_kL_C^2 \|x_B^k - x^\ast\|^2 + 2\dotp{x_A^k - x^\ast, Cx^\ast}.
\end{align*}
\qed\end{proof}

We are now ready to prove Theorem~\ref{thm:accl}.
\begin{proof}[Theorem~\ref{thm:accl}]
Part~\ref{thm:accl:part:coco}: The definition of $\gamma_{k+1}$ ensures that
\begin{align*}
\frac{1+2\gamma_k\mu_B}{\gamma_k^2} = \frac{(1-2\gamma_{k+1}\mu_C\eta)}{\gamma_{k+1}^2}.
\end{align*}
Therefore, by \eqref{eq:acclinequality}, the following inequality holds for all $k \geq 0$:
\begin{align*}
\frac{(1-\gamma_{k+1}\mu_C\eta)}{\gamma_{k+1}^2}\|x_B^{k+1} - x^\ast\|^2 + \|u_B^{k+1} - u_B^\ast\|^2 &\leq \frac{(1- \gamma_k\mu_C\eta)}{\gamma_k^2}\|x_B^k - x^\ast\|^2  + \|u_B^k - u_B^\ast\|^2.
\end{align*}

Now observe that from Equation~\eqref{eq:faststepsizecoco}, we have $\gamma_k \rightarrow 0$ as $k \rightarrow \infty$.  Therefore,
\begin{align*}
\frac{\gamma_{k}}{\gamma_{k+1}} &= \sqrt{\frac{1+2\gamma_k\mu_B}{1- 2\gamma_{k+1}\mu_C\eta}} \rightarrow 1 \quad \text{as $k \rightarrow \infty$.}
\end{align*}
In addition, the sequence $(1/\gamma_j)_{j \geq 0}$ is increasing:
\begin{align*}
\gamma_{k}^2 - \gamma_{k+1}^2 &= \gamma_{k}\gamma_{k+1}\left(2\gamma_k \mu_B + 2\gamma_{k+1}\mu_C\eta\right) > 0.
\end{align*}
Thus, we apply the Stolz-Ces\`aro theorem to compute the following limit:
\begin{align*}
\lim_{k \rightarrow\infty} (k+1)\gamma_k &= \lim_{k \rightarrow \infty} \frac{k+1}{\frac{1}{\gamma_k}} = \lim_{k \rightarrow \infty} \frac{(k+2) - (k+1)}{\frac{1}{\gamma_{k+1}} - \frac{1}{\gamma_k}} = \lim_{k \rightarrow \infty} \frac{\gamma_k \gamma_{k+1}}{\gamma_k - \gamma_{k+1}}  = \lim_{k \rightarrow \infty} \frac{\gamma_{k} \gamma_{k+1}(\gamma_{k+1} + \gamma_k)}{\gamma_k^2 - \gamma_{k+1}^2} \\
&= \lim_{k \rightarrow \infty}  \frac{\gamma_{k} \gamma_{k+1}(\gamma_{k+1} + \gamma_k)}{\gamma_k \gamma_{k+1}(2\gamma_k\mu_B +2 \gamma_{k+1}\mu_C\eta)} = \lim_{k \rightarrow \infty}  \frac{1+ \frac{\gamma_k}{\gamma_{k+1}}}{2\frac{\gamma_k}{\gamma_{k+1}}\mu_B +2 \mu_C\eta} = \lim_{k \rightarrow \infty}  \frac{1+ \frac{\gamma_k}{\gamma_{k+1}}}{2\frac{\gamma_k}{\gamma_{k+1}}\mu_B +2 \mu_C\eta} \\
&= \frac{1}{\mu_C\eta + \mu_B}.
\end{align*}
Thus, we have
\begin{align*}
\|x_B^{k+1} - x^\ast\|^2  &\leq \frac{\gamma_{k+1}^2}{(1-\gamma_{k+1}\mu_C\eta)} \left(\frac{(1- \gamma_0\mu_C\eta)}{\gamma_0^2}\|x_B^0 - x^\ast\|^2  + \|u_B^0 - u_B^\ast\|^2\right) = O\left(\frac{1}{(k+1)^2}\right).
\end{align*}

Part~\ref{thm:accl:part:Lipschitz}: The proof is nearly identical to the proof of Part~\ref{thm:accl:part:coco}. The difference is that the definition of $\gamma_{k+1}$ ensures that for all $k \geq 0$, we have
\begin{align*}
\frac{1}{\gamma_{k+1}^2} &= \frac{(1+2\gamma_k(\mu_B - \gamma_k L_C^2/2))}{\gamma_k^2}.
\end{align*}
In addition, we have $\gamma_k \rightarrow 0$ as $k \rightarrow \infty$. The sequence $(1/\gamma_j)_{j \geq 0}$ is also increasing because $\gamma_k < 2L_C^2/\mu_B$ for all $k \geq 0$.  Finally we note that $\gamma_k/\gamma_{k+1} \rightarrow 1$ as $k \rightarrow \infty$. Thus, we apply the Stolz-Ces\`aro theorem to compute the following limit:
\begin{align*}
\lim_{k \rightarrow\infty} (k+1)\gamma_k &= \lim_{k \rightarrow \infty} \frac{k+1}{\frac{1}{\gamma_k}} = \lim_{k \rightarrow \infty} \frac{(k+2) - (k+1)}{\frac{1}{\gamma_{k+1}} - \frac{1}{\gamma_k}} = \lim_{k \rightarrow \infty} \frac{\gamma_k \gamma_{k+1}}{\gamma_k - \gamma_{k+1}}  = \lim_{k \rightarrow \infty} \frac{\gamma_{k} \gamma_{k+1}(\gamma_{k+1} + \gamma_k)}{\gamma_k^2 - \gamma_{k+1}^2} \\
&= \lim_{k \rightarrow 0} \frac{\gamma_k\gamma_{k+1}(\gamma_{k+1} + \gamma_k)}{2\gamma_{k+1}^2\gamma_k(\mu_B - \gamma_kL_C^2/2)} = \lim_{k \rightarrow 0} \frac{\gamma_{k+1} + \gamma_k}{\gamma_{k+1}(2\mu_B - \gamma_kL_C^2)} = \lim_{k \rightarrow 0} \frac{1+ \gamma_k/\gamma_{k+1}}{2\mu_B - \gamma_kL_C^2} = \frac{1}{\mu_B}.
\end{align*}
\qed\end{proof}

\section{Derivation of Algorithm~\ref{alg:pdours}}\label{app:pdours}
Observe the following identities from Fig. \ref{fig:DYSTR} and Lemma \ref{lem:identities}:
\begin{subequations}\label{xuidentities}
\begin{align}
x_A^k - x_B^{k+1} &= \gamma(u_B^{k+1} - u_B^k); \\
x_B^k - x_B^{k+1} &= \gamma (u_B^{k+1} + Cx_B^{k} + u_A^{k}); \\
x_B^k - x_A^{k} &= \gamma(u_B^{k} + Cx_B^k + u_A^k).
\end{align}
\end{subequations}
These give us the further subgradient identity:
\begin{align*}
u_A^{k+1} &= u_A^{k} + (u_A^{k+1} + u_B^{k+1} + Cx_B^{k+1}) + (Cx_{B}^{k} - Cx_{B}^{k+1}) - (u_A^k + u_B^{k+1} + Cx_B^k) \\
&= u_A^k + \frac{1}{\gamma}(x_B^{k+1} - x_A^{k+1}) + (Cx_B^k - Cx_B^{k+1}) + \frac{1}{\gamma}(x_B^{k+1} - x_B^k) \\
&= J_{\frac{1}{\gamma} A^{-1}}\left( u_A^k + \frac{1}{\gamma}(2x_B^{k+1} - x_B^{k}) + (Cx_B^k - Cx_B^{k+1})\right),
\end{align*}
where the first equality follows from cancellation, the second from \eqref{xuidentities}, and the third from the property:
$$\mbox{for any}~v\in\cH,\quad u_A^{k+1} =  v-\frac{1}{\gamma}x_A^{k+1},~(x_A^{k+1},u_A^{k+1})\in\gra A ~\Longleftrightarrow~u_A^{k+1} =J_{\frac{1}{\gamma}A^{-1}}(v),  $$
which follows from the definition of resolvent $J_{\frac{1}{\gamma}A^{-1}}$.
In addition,
\begin{align*}
x_B^{k+1} &= x_B^{k} - \gamma (u_B^{k+1} + Cx_B^{k} + u_A^{k}) \\
&= J_{\gamma B}\left(x_B^{k} - \gamma Cx_B^k - \gamma u_A^k\right),
\end{align*}
where the second equality follows from the property
$$\mbox{for any}~v\in \cH,\quad x_B^{k+1} =  v-\gamma u_B^{k+1},~(x_B^{k+1},u_B^{k+1})\in\gra B ~\Longleftrightarrow~x_B^{k+1} =J_{\gamma B}(v).  $$
Altogether, for all $k \geq 0$, we have
\begin{align*}
x_B^{k+1} &= J_{\gamma B}\left(x_B^{k} - \gamma Cx_B^k - \gamma u_A^k\right),\\
u_A^{k+1} &= J_{\frac{1}{\gamma} A^{-1}}\left( u_A^k + \frac{1}{\gamma}(2x_B^{k+1} - x_B^{k}) + (Cx_B^k - Cx_B^{k+1})\right).
\end{align*}
Algorithm~\ref{alg:pdours} is obtained with the change of variable: $x^k \gets x_B^k$ and $y^k \gets u_A^k$.

\section{Proofs from Section~\ref{sec:convergence}}\label{app:convergencetheory}

\begin{proof}[of Lemma~\ref{lem:fixedpoints}]
Let $x \in \zer(A + B + C)$, that is, $0\in (A+B+C)x$.  Let  $u_A \in Ax$ and $u_B \in Bx$ be such that that $ u_A +u_B + Cx = 0$.  In addition, let $z = x + \gamma u_B$.  We will show that $z$ is a fixed point of $T$. Then $J_{\gamma B}(z) = x$ and $2  J_{\gamma B} (z) - z - \gamma CJ_{\gamma B}(z) = 2  x - z - \gamma Cx = x - \gamma Cx - \gamma u_B = x + \gamma u_A$.  Thus, $x = J_{\gamma A} (x + \gamma u_A) = J_{\gamma A}(2  J_{\gamma B} (z) - z - \gamma CJ_{\gamma B}(z))$. Therefore,
\begin{align*}
Tz &= T(x + \gamma u_B) \\
& = J_{\gamma A}(2  J_{\gamma B} (z) - z - \gamma CJ_{\gamma B}(z)) + (I_{\cH} - J_{\gamma B})(z) \\
&= x + \gamma u_B \\
&= z.
\end{align*}

Next, suppose that $z \in \Fix T$. Then there exists $u_B \in B(J_{\gamma B}(z))$ and  $u_A \in A(J_{\gamma A}(2  J_{\gamma B} (z) - z - \gamma CJ_{\gamma B}(z)))$ such that
\begin{align*}
z &= Tz \\
&= z + J_{\gamma A}(2  J_{\gamma B} (z) - z - \gamma CJ_{\gamma B}(z)) - J_{\gamma B}(z)\\
&= z - \gamma (u_A + u_B + CJ_{\gamma B}(z)).
\end{align*}
Thus, $x = J_{\gamma A}(2  J_{\gamma B} (z) - z - \gamma CJ_{\gamma B}(z)) = J_{\gamma B}(z)$ and $u_A + u_B + Cx = 0$. Therefore, $x = J_{\gamma B}(z) \in \zer(A + B+ C)$.

The identity for $\Fix T$ immediately follows from the fixed-point construction process in the first paragraph.\qed
\end{proof}

\begin{proof}[of Lemma~\ref{lm:STUVW}] Let $z, w \in \cH$. Then
\begin{align*}
\|Sz-Sw\|^2 &= \|Uz-Uw\|^2+\|T_1\circ V z- T_1\circ V w\|^2+2\dotp{T_1\circ V z- T_1\circ V w,Uz-Uw}\\
&\le \dotp{Uz-Uw,z-w}+ \dotp{T_1\circ V z- T_1\circ V w,Vz-Vw}+2\dotp{T_1\circ V z- T_1\circ V w,Uz-Uw}\\
&=\dotp{Uz-Uw,z-w} +\dotp{T_1\circ V z- T_1\circ V w,(2U+V)z-(2U+V)w}\\
&=\dotp{Uz-Uw,z-w} +\dotp{T_1\circ V z- T_1\circ V w,(I-W)z-(I-W)w}\\
&=\dotp{Sz-Sw,z-w} -\dotp{T_1\circ V z- T_1\circ V w,Wz-Ww}
\end{align*}
where the inequality follows from the firm nonexpansiveness of $U$ and $T_1$.
Then, the result follows from the identity:
$$\dotp{Sz-Sw,z-w}=\frac{1}{2}\|z - w\|^2 -\frac{1}{2}\|(I_{\cH} - S)z - (I_{\cH} - S)w\|^2 + \frac{1}{2}\|Sz - Sw\|^2. $$
\qed\end{proof}

\iftechreport
  \input{app_convergencerate}
\fi

%% file: app_convergencerate.tex
\section{Proofs for convergence rate analysis}
We now recall a lower bound property for convex functions that are strongly convex and Lipschitz differentiable.  The first bound is a consequence of~\cite[Theorem 18.15]{bauschke2011convex} and the second bound is a combination of \cite[Theorem 18.15]{bauschke2011convex} and~\cite[Theorem 2.1.12]{nesterov2004introductory}.
\begin{proposition}\label{prop:regularlowerbound}
Suppose that $f : \cH \rightarrow (0, \infty]$ is $\mu$-strongly convex and $(1/\beta)$-Lipschitz differentiable. For all $x, y \in \dom(f)$, let
\begin{align}
S_f(x, y) &:= \max\left\{\frac{\beta}{2}\|\nabla f(x) - \nabla f(y)\|^2, \frac{\mu}{2}\|x- y\|^2\right\}  \label{eq:Sfunction}\\
Q_f(x, y) &:= \max\left\{2S_f(x, y),  \frac{\mu}{(\mu\beta + 1)}\|x-y\|^2 + \frac{\beta}{(\mu\beta+ 1)}\|\nabla f(x) - \nabla f(y)\|^2\right\}\label{eq:Tfunction}.
\end{align}
Then for all $x, y \in \dom(f)$, we have
\begin{align*}
f(x) - f(y) - \dotp{x-y, \nabla f(y)} &\geq S_f(x, y); \numberthis\label{eq:Slowerbound1} \\
\dotp{\nabla f(x) - \nabla f(y), x- y} &\geq Q_f(x, y). \numberthis\label{eq:Tlowerbound}
\end{align*}
Similarly, if $A : \cH \rightarrow \cH$ is $\mu$-strongly monotone and $\beta$-cocoercive, we let
\begin{align*}
Q_A(x, y) &= \max\left\{ \mu \|x - y\|^2, \beta\|Ax - Ay\|^2\right\}
\end{align*}
for all $x, y \in \dom(A)$.  Then for all $x, y \in \dom(A)$, we have
\begin{align*}
\dotp{Ax - Ay, x - y}& \geq  Q_A(x, y).
\end{align*}
\end{proposition}

We follow the convention that every function $f$ is $\mu_f\geq 0$ strongly convex and $(1/\beta_f) \geq 0$ Lipschitz where we allow the possibility that $\beta_f = \mu_f = 0$. With this notation, the results of Proposition~\ref{prop:regularlowerbound} continue hold for all $f$.  We follow the same convention for monotone operators. In particular, every monotone operator $A : \cH \rightarrow 2^\cH$ is $\mu_A$-strongly monotone  and $\beta_A$-cocoercive where $\mu_A \geq 0$ and $\beta_A \geq 0$. Finally, we follow convention that $Q_{\partial f} := Q_{f}$.

Note that we could extend our definition of $Q_{A}(\cdot, \cdot)$ (or $Q_f(\cdot, \cdot)$) to the case where $A$ is merely strongly monotone in a subset of the coordinates of $\cH$ (which is then assumed to be a product space). This extension is straightforward, though slightly messy. Thus, we omit this extension.

The following identity will be applied repeatedly:
\begin{proposition}\label{prop:fundamentalequality}
Let $z \in \cH$, let $z^\ast$ be a fixed point of $T$, let $\gamma > 0$, let $ \lambda > 0$, and let $z^+ = (1-\lambda) z + \lambda Tz$.  Then
\begin{align*}
&2\gamma \lambda \dotp{x_B - x_A, u_B^\ast + Cx^\ast} + 2\gamma \lambda Q_A(x_A, x^\ast) + 2\gamma \lambda Q_B(x_B, x^\ast) + 2\gamma \lambda Q_C(x_B, x^\ast)\\
&\leq 2\gamma\lambda\dotp{x_A - x^\ast, u_A} + 2\gamma\lambda\dotp{x_B - x^\ast, u_B + Cx_B} \numberthis \label{eq:filowerbound} \\
&=  \|z - x^\ast\|^2 - \|z^{+} - x^\ast\|^2 + \left(1 - \frac{2}{\lambda}\right)\|z- z^{+}\|^2 + 2\gamma\dotp{z - z^{+}, Cx_B} \numberthis \label{eq:fiformmain} \\
&= \|z - z^\ast\|^2 - \|z^{+} - z^\ast\|^2 + \left(1 - \frac{2}{\lambda}\right)\|z - z^+\|^2 + 2\gamma\dotp{z - z^+, Cx_B + u_B^\ast} \numberthis \label{eq:fistrongconvergenceform}
\end{align*}
where $x_A \in \dom(A), x_B\in \dom(B), u_B\in Bx_B$ and  $u_A \in A u_A$ are defined in Lemma~\ref{lem:identities} and Equation~\eqref{eq:fiformmain} holds for all $x^\ast \in \cH$, while Equations~\eqref{eq:filowerbound} and ~\eqref{eq:fistrongconvergenceform} hold when $x^\ast = J_{\gamma B}(z^\ast)$ and $u_B^\ast = \frac{1}{\gamma}(z^\ast - x^\ast)$.  In particular, when $x^\ast = J_{\gamma B}(z^\ast)$, we have
\begin{align*}
\|z^+ - z^\ast\|^2 &+ \left( \frac{2}{\lambda} - 1\right)\|z - z^+\|^2 + 2 \gamma\lambda Q_A(x_A, x^\ast) + 2\gamma\lambda Q_B(x_B, x^\ast) + 2\gamma \lambda Q_C(x_B, x^\ast)   \\
&\leq \|z - z^\ast\|^2 + 2\gamma \dotp{z - z^+, C(x_B)  -  C (x^\ast)}. \numberthis\label{eq:linearconvergenceinequality}
\end{align*}
\end{proposition}
\begin{proof}
First we show inequality~\eqref{eq:filowerbound}: Let $u_A^\ast \in Ax^\ast$ and $u_B^*\in Bx^*$ be such that $u_A^\ast + u_B^\ast + Cx^\ast = 0 $. Then
\begin{align*}
&2\gamma\lambda \dotp{x_A - x^\ast, u_A} + 2\gamma\lambda\dotp{x_B - x^\ast, u_B + Cx_B} \\
&\geq 2\gamma\lambda \dotp{ x_A - x^\ast, u_A^\ast}  + 2\gamma\lambda\dotp{x_B - x^\ast, u_B^\ast + Cx^\ast} + 2\gamma \lambda Q_A(x_A, x^\ast) + 2\gamma \lambda Q_B(x_B, x^\ast) + 2\gamma \lambda Q_C(x_B, x^\ast)\\
&=\gamma\lambda \dotp{ x_A - x^\ast, u_A^\ast + u_B^\ast + Cx^\ast}  + 2\gamma\lambda\dotp{x_B - x_A, u_B^\ast + Cx^\ast} + 2\gamma \lambda Q_A(x_A, x^\ast) + 2\gamma \lambda Q_B(x_B, x^\ast) + 2\gamma \lambda Q_C(x_B, x^\ast)\\
&= 2\gamma\lambda\dotp{x_B - x_A, u_B^\ast + Cx^\ast} + 2\gamma \lambda Q_A(x_A, x^\ast) + 2\gamma \lambda Q_B(x_B, x^\ast) + 2\gamma \lambda Q_C(x_B, x^\ast).
\end{align*}
Now we show Equation~\eqref{eq:fiformmain}:
\begin{align*}
&2\lambda\gamma\dotp{x_A - x^\ast, u_A} + 2\gamma\lambda\dotp{x_B - x^\ast, u_B + Cx_B} \\
&=2 \gamma\lambda\dotp{x_A - x_B, u_A} + 2\gamma\lambda\dotp{ x_B - x^\ast,  u_A + u_B + Cx_B} \\
&= 2 \lambda\dotp{x_A - x_B, \gamma u_A} + 2\lambda\dotp{ x_B - x^\ast,  x_B - x_A} \\
&=  2\lambda\dotp{ x_B - \gamma u_A - x^\ast,  x_B - x_A} \\
&=  2\dotp{ z+ (x_B - z - \gamma u_A) - x^\ast,  z- z^{+}} \\
&=  2\dotp{ z- \gamma(u_B + u_A + Cx_B) - x^\ast,  z - z^{+}} + 2\gamma\dotp{z - z^{+}, Cx_B} \\
&=  2\dotp{ z - \frac{1}{\lambda}(z - z^{+}) - x^\ast,  z - z^{+}} + 2\gamma\dotp{z - z^{+}, Cx_B}\\
&=  2\dotp{ z - x^\ast,  z - z^{+}} - \frac{2}{\lambda}\|z - z^{+}\|^2 + 2\gamma\dotp{z - z^{+}, Cx_B}\\
&\stackrel{\eqref{eq:cosinerule}}{=} \|z - x^\ast\|^2 - \|z^{+} - z^\ast\|^2 + \left(1 - \frac{2}{\lambda}\right)\|z - z^{+}\|^2 + 2\gamma\dotp{z - z^{+}, Cx_B}.
\end{align*}
Now assume that $x^\ast = J_{\gamma B}(z^\ast)$ and show Equation~\eqref{eq:fistrongconvergenceform}:
\begin{align*}
&2\lambda\gamma\dotp{x_A - x^\ast, u_A} + 2\gamma\lambda\dotp{x_B - x^\ast, u_B + Cx_B} \\
&=  2\dotp{ z - x^\ast,  z - z^{+}} - \frac{2}{\lambda}\|z - z^{+}\|^2 + 2\gamma\dotp{z - z^{+}, Cx_B}\\
&=  2\dotp{ z- z^\ast,  z - z^{+}} - \frac{2}{\lambda}\|z - z^{+}\|^2 + 2\gamma\dotp{z - z^{+}, Cx_B + u_B^\ast} \\
&\stackrel{\eqref{eq:cosinerule}}{=} \|z - z^\ast\|^2 - \|z^{+} - z^\ast\|^2 + \left(1 - \frac{2}{\lambda}\right)\|z - z^{+}\|^2 + 2\gamma\dotp{z - z^{+}, Cx_B + u_B^\ast}.
\end{align*}
Equation~\eqref{eq:linearconvergenceinequality} follows from rearranging the above inequalities.
\qed\end{proof}

\begin{corollary}[Function value bounds]\label{cor:fundamentalequalityfunctions}
Assume the notation of Proposition~\ref{prop:fundamentalequality}. Let $f,g$, and $h$ be closed, proper and convex functions from $\cH$ to $(-\infty, \infty]$.  Suppose that $h$ is $(1/\beta)$-Lipschitz differentiable.  Suppose that $A = \partial f$, $B = \partial g$, and $C = \nabla h$. Then if $x^\ast =\prox_{\gamma g}(z^\ast)$, $\tnabla g(x^\ast) = (1/\gamma)(z^\ast - x^\ast)$, and $\tnabla f(x^\ast) \in \partial f(x^\ast)$ and $\tnabla g(x^\ast)\in\partial g(x^\ast)$ are such that $\nabla h(x^\ast) + \tnabla g(x^\ast) + \tnabla f(x^\ast) = 0$, we have
\begin{align*}
&2\gamma \dotp{z - z^+,\tnabla g(x^\ast) +  \nabla h(x^\ast)} + 4\gamma \lambda S_f(x_f, x^\ast) + 4\gamma \lambda S_g(x_g, x^\ast) + 4 \gamma \lambda S_h(x_g, x^\ast)\\
&\leq 2\gamma\lambda\left(f(x_f) + g(x_g) + h(x_g) - (f+g+h)(x^\ast) +  S_f(x_f, x^\ast) +  S_g(x_g, x^\ast) + S_h(x_g, x^\ast)\right) \numberthis\label{eq:filowerbound_function}\\
&\leq \|z - x^\ast\|^2 - \|z^{+} - x^\ast\|^2 + \left(1 - \frac{2}{\lambda}\right)\|z- z^{+}\|^2 + 2\gamma\dotp{z - z^{+}, \nabla  h(x_B)}\numberthis\label{eq:fiformmain_function} \\
&= \|z - z^\ast\|^2 - \|z^{+} - z^\ast\|^2 + \left(1 - \frac{2}{\lambda}\right)\|z - z^+\|^2 + 2\gamma\dotp{z - z^+, \nabla h(x_B) + \tnabla g(x^\ast)}. \numberthis\label{eq:functionvaluefundamentalinequality}
\end{align*}
where $x_f \in \dom(f), x_g\in \dom(g)$ are defined in Lemma~\ref{lem:identities} and Equation~\eqref{eq:fiformmain_function} holds for all $x^\ast \in \cH$, while Equations~\eqref{eq:filowerbound_function} and ~\eqref{eq:functionvaluefundamentalinequality} hold when $x^\ast = \prox_{\gamma g}(z^\ast)$ and $\tnabla g(x^\ast) = \frac{1}{\gamma}(z^\ast - x^\ast)$.  In particular, when $x^\ast = \prox_{\gamma g}(z^\ast)$, we have
 \begin{align*}
\|z^+ - z^\ast\|^2 &+ \left( \frac{2}{\lambda} - 1\right)\|z - z^+\|^2 + {4\gamma \lambda S_f(x_f, x^\ast) + 4\gamma \lambda S_g(x_g, x^\ast) + 4 \gamma \lambda S_h(x_g, x^\ast)}  \\
&\leq \|z - z^\ast\|^2 + 2\gamma \dotp{z - z^+, \nabla h(x_B)  - \nabla h (x^\ast)}. \numberthis\label{eq:functionlinearconvergenceinequality}
 \end{align*}
\end{corollary}
\begin{proof}
Equation~\eqref{eq:functionvaluefundamentalinequality} is a direct consequence of Proposition~\ref{prop:fundamentalequality} together with the inequalities:
\begin{align*}
f(x_f) + g(x_g) &+ h(x_g) - (f+ g+ h)(x^\ast)  \\
&\leq \dotp{ x_f - x^\ast, \tnabla f(x_f)} + \dotp{ x_g - x^\ast, \tnabla g(x_g) + \nabla h(x_g)} - S_f(x_f, x^\ast) - S_g(x_g, x^\ast) - S_h(x_g, x^\ast); \\
f(x_f) + g(x_g) &+ h(x_g) - (f+ g+ h)(x^\ast)  \\
&\geq \dotp{x_f - x^\ast, \tnabla f(x^\ast)} + \dotp{x_g - x^\ast, \tnabla g(x^\ast) + \nabla h(x^\ast)}  + S_f(x_f, x^\ast) + S_g(x_g, x^\ast) +  S_h(x_g, x^\ast) \\
&= \dotp{ x_g - x_f, \tnabla g(x^\ast) + \nabla h(x^\ast)} + \dotp{ x_f - x^\ast, \tnabla f(x^\ast) + \tnabla g(x^\ast) + \nabla h(x^\ast)}. \\
& + S_f(x_f, x^\ast) + S_g(x_g, x^\ast) +  S_h(x_g, x^\ast) \\
&= \frac{1}{\lambda}\dotp{ z - z^+, \tnabla g(x^\ast) + \nabla h(x^\ast)} + S_f(x_f, x^\ast) + S_g(x_g, x^\ast) +  S_h(x_g, x^\ast).
\end{align*}
where we use that $x_g - x_f = (1/\lambda)(z - z^+)$ (see Lemma~\ref{lem:identities}.)

Equation~\eqref{eq:functionlinearconvergenceinequality} is a consequence of the Equation~\eqref{eq:linearconvergenceinequality}.
\qed\end{proof}

\begin{corollary}[Subdifferentiable + monotone model variational inequality bounds]\label{cor:fundamentalequalityvariational}
Assume the notation of Proposition~\ref{prop:fundamentalequality}. Let $f,g$, and $h$ be closed, proper and convex functions from $\cH$ to $(-\infty, \infty]$, and let $\nabla h$ be $(1/\beta_h)$-Lipschitz.  Let $\overline{A}, \overline{B}$ and $\overline{C}$ be monotone operators on $\cH$, and let $\overline{C}$ be $\beta_C$-cocoercive.  Suppose that $A = \partial f + \overline{A}$, $B = \partial g + \overline{B}$, and $C = \nabla h +  \overline{C}$. Let $\tnabla f(x_A) + u_{\overline{A}} = u_A$ where $\tnabla f(x_A) \in \partial f(x_A)$ and $u_{\overline{A}} \in \overline{A}x_A$. Likewise let $\tnabla g(x_B) + u_{\overline{B}} = u_B$ where $\tnabla g(x_B) \in \partial g(x_B)$ and $u_{\overline{B}} \in \overline{A}x_A$. Then for all $ x\in \dom(f) \cap \dom(g)$, we have
\begin{align*}
& 2\gamma\lambda\biggl(f(x_A) + g(x_B) + h(x_B) - (f+g+h)(x) +  S_f(x_f, x) +  S_g(x_g, x) + S_h(x_g, x) \\
&+ \dotp{x_A - x, u_{\overline{A}}} + \dotp{ x_B - x, u_{\overline{B}} + \overline{C} x_B} \biggr)  \\
&\leq \|z - x\|^2 - \|z^{+} - x\|^2 + \left(1 - \frac{2}{\lambda}\right)\|z- z^{+}\|^2 + 2\gamma\dotp{z - z^{+}, \nabla  h(x_B) + \overline{C}x_B} \numberthis \label{eq:monotonevariationalinequalitybound}
\end{align*}
\end{corollary}
\begin{proof}
Equation~\eqref{eq:functionvaluefundamentalinequality} is a direct consequence of Proposition~\ref{prop:fundamentalequality} together with the following inequality:
\begin{align*}
f(x_f) + g(x_g) &+ h(x_g) - (f+ g+ h)(x)  \\
&\leq \dotp{ x_f - x, \tnabla f(x_f)} + \dotp{ x_g - x, \tnabla g(x_g) + \nabla h(x_g)} - S_f(x_f, x^\ast) - S_g(x_g, x^\ast) - S_h(x_g, x^\ast).
\end{align*}
\qed\end{proof}

\subsection{General case: convergence rates of upper and lower bounds}

We will prove the most general rates by showing how fast the upper and lower bounds in Proposition~\ref{prop:fundamentalequality} converge.  Then we will deduce convergence rates. Thus, in this section we set
\begin{align*}
\kappa_{1}^k(\lambda, x) &= \|z^k - x\|^2 - \|z^{k+1} - x^\ast\|^2 + \left(1 - \frac{2}{\lambda}\right)\|z^k- z^{k+1}\|^2 + 2\gamma\dotp{z^k - z^{k+1}, Cx_B^k} \\
\kappa_{2}^k(\lambda, x^\ast) &= \|z^k - z^\ast\|^2 - \|z^{k+1} - z^\ast\|^2 + \left(1 - \frac{2}{\lambda}\right)\|z^k- z^{k+1}\|^2 + 2\gamma\dotp{z^k - z^{k+1}, Cx_B^k - Cx^\ast} 
\end{align*}
{where $\lambda > 0$, $z^\ast$ is a fixed point of $T$, $x^\ast = J_{\gamma B}(z^\ast)$, and $x \in \cH$.}

\begin{theorem}[Nonergodic convergence rates of bounds]\label{thm:nonergodicmain}
Let $(z^j)_{j \geq 0}$ be generated by Equation~\eqref{eq:zitr} with $\varepsilon \in (0, 1), \gamma \in (0, 2\beta\varepsilon), \alpha = 1/(2-\varepsilon) < 2\beta/(4\beta - \gamma)$, and $(\lambda_j)_{j \geq 0} \subseteq (0, 1/\alpha)$. Let $z^\ast$ be a fixed point of $T$, let $x^*=J_{\gamma B}(z^*)$, and let $x \in \cH$. Assume that $\underline{\tau} := \inf_{j \geq 0} \lambda_j(1 - \alpha\lambda_j)/\alpha$. Then for all $k \geq 0$,
\begin{align*}
\kappa_1^k(1, x) &\leq \frac{2(\|z^{\ast} - x\| + (1+ \gamma/\beta)\|z^0 - z^\ast\|+ \gamma \|Cx^\ast\|) \|z^0 - z^\ast\|}{\sqrt{\underline{\tau}(k+1)}} && \text{and}  && |\kappa_1^k(1, x)| = o\left(\frac{1 + \|x\|}{\sqrt{k+1}}\right); \numberthis\label{eq:nonergodickappauppermain} \\
\kappa_2^k(1, x^\ast) &\leq \frac{2(1 + \gamma/\beta)\|z^0 - z^\ast\|^2}{\sqrt{\underline{\tau}(k+1)}} && \text{and} &&  |\kappa_2^k(1, x^\ast)| = o\left(\frac{1}{\sqrt{k+1}}\right).  \numberthis\label{eq:nonergodickappauppermainSTRONG}
\end{align*}
We also have the following lower bound:
\begin{align*}
 \dotp{x_B^k - x_A^k, u_B^\ast + Cx^\ast} & \geq \frac{-\|z^0 - z^\ast\|\|u_B^\ast + Cx^\ast\|}{\sqrt{\underline{\tau}(k+1)}} && \text{and} && | \dotp{x_B^k - x_A^k, u_B^\ast + Cx^\ast}| = o\left(\frac{1}{\sqrt{k+1}}\right). \numberthis\label{eq:nonergodickappalowermain}
\end{align*}
\end{theorem}
\begin{proof}
Fix $k \geq 0$. Observe that
\begin{align*}
\|Cx_B^k\| \leq \|Cx_B^k - Cx^\ast\| + \|Cx^\ast\| \leq \frac{1}{\beta}\|x_B^k - x^\ast\| + \|Cx^\ast\| \leq \frac{1}{\beta}\|z^k - z^0\| + \|Cx^\ast\| \leq \frac{1}{\beta} \|z^0 - z^\ast\| + \|Cx^\ast\|
\end{align*}
by the $(1/\beta)$-Lipschitz continuity of $C$, the nonexpansiveness of $J_{\gamma B}$, and the monotonicity of the sequence $(\|z^j - z^\ast\|)_{j \geq 0}$ (see Part~\ref{thm:convergence:part:biginequality} of theorem~\ref{thm:convergence}).  Thus,
\begin{align*}
|\kappa_1^k(1,x)| &\stackrel{\eqref{eq:cosinerule}}{=} \left| 2\dotp{z^{k+1} - x, z^k - z^{k+1}} + 2\gamma\dotp{z^k - z^{k+1}, Cx_B^k}\right| \\
&\leq \frac{2\|z^{k+1} - x\|\|z^0 - z^\ast\| + (2\gamma/\beta)\|z^0 - z^\ast\|^2 + 2\gamma \|Cx^\ast\|\|z^0 - z^\ast\|}{\sqrt{\underline{\tau}(k+1)}} \\
&\leq \frac{(2\|z^{\ast} - x\| + (2+ 2\gamma/\beta)\|z^0 - z^\ast\|+ 2\gamma \|Cx^\ast\|) \|z^0 - z^\ast\|}{\sqrt{\underline{\tau}(k+1)}}
\end{align*}
where the bound in the second inequality follows from Cauchy-Schwarz and the upper bound in Part~\ref{thm:convergence:part:convergencerate} of Theorem~\ref{thm:convergence}, and the last inequality follows because  $\|z^{k+1} - x\| \leq \|z^{k+1} - z^\ast\| + \|z^\ast - x\| \leq \|z^{0} - z^\ast\| + \|z^\ast - x\|$ (see Part~\ref{thm:convergence:part:biginequality} of Theorem~\ref{thm:convergence}). The little-$o$ rate follows because $\|z^k - z^{k+1}\| = o\left(1/\sqrt{k+1}\right)$ by  Part~\ref{thm:convergence:part:convergencerate} of Theorem~\ref{thm:convergence}.

The proof of Equation~\eqref{eq:nonergodickappauppermainSTRONG} follows nearly the same reasoning as the proof of Equation~\eqref{eq:nonergodickappauppermain}. Thus, we omit the proof.

Next, because $x_B^k - x_A^k = z^k - Tz^k$ (see Lemma~\ref{lem:identities}), we have
\begin{align*}
|\dotp{z^k - Tz^k, u_B^\ast + Cx_B^\ast}| &\leq \frac{\|z^0 - z^\ast\|\|u_B^\ast + Cx^\ast\|}{\sqrt{\underline{\tau}(k+1)}}.
\end{align*}
by Part~\ref{thm:convergence:part:convergencerate} of Theorem~\ref{thm:convergence}.  Similarly The little-$o$ rate follows because $\|z^k - z^{k+1}\| = o\left(1/\sqrt{k+1}\right)$ by  Part~\ref{thm:convergence:part:convergencerate} of Theorem~\ref{thm:convergence}.
\qed\end{proof}

We now prove two ergodic results.

\begin{theorem}[Ergodic convergence rates of bounds for Equation~\eqref{avg1}]\label{thm:ergodic1}
Let $(z^j)_{j \geq 0}$ be generated by Equation~\eqref{eq:zitr} with $\varepsilon \in (0, 1), \gamma \in (0, 2\beta\varepsilon), \alpha = 1/(2-\varepsilon) < 2\beta/(4\beta - \gamma)$, and $(\lambda_j)_{j \geq 0} \subseteq (0, 1/\alpha]$. {Let $z^\ast$ be a fixed point of $T$, let $x^*=J_{\gamma B}(z^*)$, and let $x \in \cH$.} Then for all $k \geq 0$,
\begin{align*}
\frac{1}{\sum_{i=0}^k \lambda_i}\sum_{i = 0}^k \kappa_1^i(\lambda_i, x) &\leq \frac{\|z^0 - x\|^2 + \frac{\gamma}{(2\beta\varepsilon - \gamma)}\|z^0 - z^\ast\|^2 + 4\gamma \|z^0 -z^\ast\|\|Cx^\ast\|}{\sum_{i=0}^k \lambda_i}  \numberthis\label{eq:ergodickappauppermain} \\
\frac{1}{\sum_{i=0}^k \lambda_i}\sum_{i = 0}^k \kappa_2^i(\lambda_i, x^\ast) &\leq \frac{\left(1 + \frac{\gamma}{(2\beta\varepsilon - \gamma)}\right)\|z^0 - z^\ast\|^2}{\sum_{i=0}^k \lambda_i}.  \numberthis\label{eq:ergodickappauppermainSTRONG}
\end{align*}
We also have the following lower bound:
\begin{align*}
\frac{1}{\sum_{i=0}^k \lambda_i} \sum_{i=0}^k\lambda_i \dotp{x_B^i - x_A^i, u_B^\ast + Cx^\ast} & \geq \frac{-2\|z^0 - z^{\ast}\|\|u_B^\ast + Cx^\ast\|}{\sum_{i=0}^k \lambda_i} . \numberthis\label{eq:ergodickappalowermain}
\end{align*}
In addition, the following feasibility bound holds:
\begin{align*}
\left\|\frac{1}{\sum_{i=0}^k \lambda_i}\sum_{i=0}^k \lambda_i(x_B^i - x_A^{i})\right\| \leq \frac{2\|z^0 - z^\ast\|}{\sum_{i=0}^k \lambda_i}. \numberthis \label{eq:feasibilitySCHEME1}
\end{align*}
\end{theorem}
\begin{proof}
Fix $k \geq 0$. We first prove the feasibility bound:
\begin{align*}
\left\|\frac{1}{\sum_{i=0}^k \lambda_i}\sum_{i=0}^k \lambda_i(x_B^i - x_A^i) \right\| \leq \frac{1}{\sum_{i=0}^k \lambda_i}\left\|\sum_{i=0}^k (z^i - z^{i+1})\right\| &= \frac{\|z^0 - z^{k+1}\|}{\sum_{i=0}^k \lambda_i} \leq \frac{2\|z^0 - z^\ast\|}{\sum_{i=0}^k \lambda_i}.
\end{align*}
where the last inequality follows from $\|z^0 - z^{k+1}\| \leq \|z^0 - z^\ast\| + \|z^{k+1} - z^\ast\| \leq 2 \|z^0 - z^\ast\|$.

Let $\eta_k = 2/\lambda_k - 1$. Note that $\eta_k > 0$, by assumption. In addition, $1/\eta_k = \lambda_k/(2-\lambda_k) \leq \lambda_k/\varepsilon$.  Thus,
we have
\begin{align*}
2\gamma \dotp{z^k - z^{k+1}, Cx_B^k} &= 2\gamma \dotp{z^k - z^{k+1}, Cx_B^k - Cx^\ast} + 2\gamma \dotp{z^k - z^{k+1}, Cx^\ast} \\
&\leq \eta_k \|z^k - z^{k+1}\|^2 + \frac{\gamma^2}{\eta_k} \|Cx_B^k - Cx^\ast\|^2 + 2\gamma \dotp{z^k - z^{k+1}, Cx^\ast} \numberthis \label{eq:gradientinnerproductbound}
\end{align*}
Thus, for all $k \geq 0$, we have
\begin{align*}
\sum_{i=0}^k \kappa_1^i(\lambda_i,x) &\leq \|z^0 - x\|^2 + \sum_{i=0}^k \left(-\eta_i \|z^{i+1} - z^i\|^2 + 2\gamma \dotp{z^i - z^{i+1}, Cx_B^i}\right) \\
&\stackrel{\eqref{eq:gradientinnerproductbound}}{\leq}  \|z^0 - x\|^2 + \sum_{i=0}^k \left(\frac{\gamma^2\lambda_i}{\varepsilon} \|Cx_B^i - Cx^\ast\|^2 + 2\gamma \dotp{z^i - z^{i+1}, Cx^\ast} \right) \\
&\leq \|z^0 - x\|^2 + \frac{\gamma^2}{\varepsilon\gamma(2\beta - \gamma/\varepsilon)}\|z^0 - z^\ast\|^2 + 2\gamma \dotp{ z^0  - z^{k+1}, Cx^\ast} \\
&\leq \|z^0 - x\|^2 + \frac{\gamma}{(2\beta\varepsilon - \gamma)}\|z^0 - z^\ast\|^2 + 4\gamma \|z^0 -z^\ast\|\|Cx^\ast\|.
\end{align*}
where the third inequality follows from Part~\ref{thm:convergence:part:gradientsum} of Theorem~\ref{thm:convergence} and the fourth inequality follows because $\|z^0 - z^{k+1}\| \leq \|z^0 - z^\ast\| + \|z^{k+1} - z^\ast\| \leq 2 \|z^0 - z^\ast\|$.

The proof of Equation~\eqref{eq:ergodickappauppermainSTRONG} follows nearly the same reasoning as the proof of Equation~\eqref{eq:ergodickappauppermain}. Thus, we omit the proof.

Finally, Equation~\eqref{eq:ergodickappalowermain} follows directly from Cauchy Schwarz and Equation~\eqref{eq:feasibilitySCHEME1}.
\qed\end{proof}

\begin{theorem}[Ergodic convergence rates of bounds for Equation~\eqref{avg2}]\label{thm:ergodic2}
Let $(z^j)_{j \geq 0}$ be generated by Equation~\eqref{eq:zitr} with $\varepsilon \in (0, 1), \gamma \in (0, 2\beta\varepsilon), \alpha = 1/(2-\varepsilon) < 2\beta/(4\beta - \gamma)$, and $\lambda_j \equiv \lambda \subseteq (0, 1/\alpha]$. {Let $z^\ast$ be a fixed point of $T$, let $x^*=J_{\gamma B}(z^*)$, and let $x \in \cH$.}  Then for all $k \geq 0$,
\begin{align*}
\frac{2}{(k+1)(k+2)}\sum_{i = 0}^k (i+1) \kappa_1^i(\lambda, x) &\leq \frac{2\left(2\|z^\ast - x\|^2 + \left(2 + \frac{\gamma}{(2\beta\varepsilon - \gamma)}\right) \|z^0 - z^\ast\|^2 + 10\|z^0 - z^\ast\|\|Cx^\ast\|\right) }{k+1}; \numberthis\label{eq:ergodickappauppermainSCHEME2} \\
\frac{2}{(k+1)(k+2)}\sum_{i = 0}^k (i+1) \kappa_2^i(\lambda, x^\ast) &\leq \frac{2\left(1 + \frac{\gamma}{(2\beta\varepsilon - \gamma)}\right)\|z^0 - z^\ast\|^2}{k+1}.  \numberthis\label{eq:ergodickappauppermainSCHEME2STRONG}
\end{align*}
We also have the following lower bound:
\begin{align*}
\frac{2}{(k+1)(k+2)}\sum_{i=0}^k (i+1) \dotp{x_B^i - x_A^i, u_B^\ast + Cx^\ast} & \geq  \frac{-5\|z^0 - z^\ast\|}{\lambda(k+1)}   \numberthis\label{eq:ergodickappalowermainSCHEME2}
\end{align*}
In addition, the following feasibility bound holds:
\begin{align*}
\left\|\frac{2}{(k+1)(k+2)}\sum_{i=0}^k (i+1) (x_B^i - x_A^{i})\right\| \leq \frac{5\|z^0 - z^\ast\|}{\lambda(k+1)}. \numberthis \label{eq:feasibilitySCHEME2}
\end{align*}
\end{theorem}
\begin{proof}
Fix $k \geq 0$. We first prove the feasibility bound:
\begin{align*}
&\left\|\frac{2}{(k+1)(k+2)}\sum_{i=0}^k (i+1) (x_B^i - x_A^{i})\right\| = \frac{1}{\lambda}\left\|\frac{2}{(k+1)(k+2)}\sum_{i=0}^k (i+1) (z^i - z^{i+1})\right\| \\
&\leq \left\|\frac{2}{(k+1)(k+2)}\sum_{i=0}^k \left((z^{i+1} - z^\ast) +  (i+1) (z^i - z^\ast) - (i+2)(z^{i+1} - z^\ast)\right)\right\| \\
&\leq \left\|\frac{2}{(k+1)(k+2)}\left(\sum_{i=0}^k (z^{i+1} - z^\ast) +   (z^0 - z^\ast) - (k+2)(z^{k+1} - z^\ast)\right)\right\| \\
&\leq \frac{2}{(k+1)(k+2)}\sum_{i=0}^k \|z^{i+1} - z^\ast\| + \frac{2}{(k+1)(k+2)} \|z^0 - z^\ast\| +  \frac{2}{(k+1)}\|z^{k+1} - z^\ast\| \\
&\leq \frac{2\|z^0 - z^\ast\|}{(k+2)} + \frac{2 \|z^0 - z^\ast\|}{(k+1)(k+2)} +  \frac{2\|z^{0} - z^\ast\|}{(k+1)} \\
&\leq \frac{2\|z^0 - z^\ast\|\left(2 + \frac{1}{k+2}\right)}{(k+1)} \leq \frac{5\|z^0 - z^\ast\|}{k+1} \numberthis\label{eq:feasibilitybound}
\end{align*}
where we use the bound $\|z^k - z^\ast\| \leq \|z^0 - z^\ast\|$ for all $k \geq 0$ (see Part~\ref{thm:convergence:part:biginequality} of theorem~\ref{thm:convergence}). The bound then follows because $\lambda(x_B^k - x_A^k) = z^k - z^{k+1}$ for all $k \geq 0$ (Lemma~\ref{lem:identities}).

We proceed as in the proof of Theorem~\ref{thm:ergodic1} (which is where $\eta_i := 2/\lambda_i - 1$ is defined):
\begin{align*}
\frac{2}{(k+2)(k+1)}\sum_{i=0}^k (i+1)\kappa_1^i(\lambda, x) &= \frac{2}{(k+2)(k+1)}\sum_{i=0}^k\biggr(  \left((i+1)\|z^i - x\|^2 - (i+1) \|z^{i+1} - x\|^2 \right) \\
&+ (i+1)\left(-\eta_i \|z^{i+1} - z^i\|^2 + 2\gamma \dotp{z^i - z^{i+1}, Cx_B^i}\right)\biggr) \\
&\leq \frac{2}{(k+2)(k+1)} \sum_{i=0}^k\biggl( \left(\|z^{i+1} - x\|^2 + (i+1)\|z^i - x\|^2 - (i+2) \|z^{i+1} - x\|^2 \right) \\
&\stackrel{\eqref{eq:gradientinnerproductbound}}{+} (i+1)\left(\frac{\gamma^2\lambda}{\varepsilon} \|Cx_B^i - Cx^\ast\|^2 + 2\gamma \dotp{z^i - z^{i+1}, Cx^\ast} \right)
\biggr) \\
&\leq \frac{2}{(k+2)(k+1)} \sum_{i=0}^k\|z^{i} - x\|^2 + \frac{2}{k+2}\sum_{i=0}^k\frac{\gamma^2\lambda}{\varepsilon} \|Cx_B^i - Cx^\ast\|^2\\
&+  \frac{2}{(k+2)(k+1)}\sum_{i=0}^k2\gamma(i+1) \dotp{z^i - z^{i+1}, Cx^\ast} \\
&\leq \frac{2}{(k+2)(k+1)} \sum_{i=0}^k\left(2\|z^{i} - z^\ast\|^2  + 2\|z^\ast - x\|^2\right) +  \frac{\frac{2\gamma}{(2\beta\varepsilon - \gamma)} \|z^0 - z^\ast\|^2}{k+2}\\
&\stackrel{\eqref{eq:feasibilitybound}}{+}  \frac{20\gamma\|z^0 - z^\ast\|\|Cx^\ast\|}{k+1} \\
&\leq \frac{2\left(2\|z^\ast - x\|^2 + \left(2 + \frac{\gamma}{(2\beta\varepsilon - \gamma)}\right) \|z^0 - z^\ast\|^2 + 10\|z^0 - z^\ast\|\|Cx^\ast\|\right) }{k+1}.
\end{align*}

The proof of Equation~\eqref{eq:ergodickappauppermainSCHEME2STRONG} follows nearly the same reasoning as the proof of Equation~\eqref{eq:ergodickappauppermainSCHEME2}. Thus, we omit the proof.

Finally, Equation~\eqref{eq:ergodickappalowermainSCHEME2} follows directly from Cauchy Schwarz and Equation~\eqref{eq:feasibilitySCHEME2}.
\qed\end{proof}

\subsection{General case: Rates of function values and variational inequalities}\label{sec:generalcaserates}

In this section, we use the convergence rates of the upper and lower bounds derived in Theorems~\ref{thm:nonergodicmain},~\ref{thm:ergodic1}, and~\ref{thm:ergodic2} to deduce convergence rates function values and variational inequalities. All of the convergence rates have the following orders:
\begin{align*}
\text{Nonergodic: $o\left(\frac{1}{\sqrt{k+1}}\right)$} && \text{and} && \text{Ergodic: $O\left(\frac{1}{k+1}\right)$}.
\end{align*}
We work with three model problems.
\begin{itemize}
\item {\bf Most general:} $A = \partial f + \overline{A}$, $B = \partial g + \overline{B}$ and $C = \nabla h + \overline{C}$ where $f, g$ and $h$ are functions and $\overline{A}, \overline{B}$ and $\overline{C}$ are monotone operators.  See Corollary~\ref{cor:fundamentalequalityvariational} for our assumptions about this case, and see Corollary~\ref{cor:nonergodicvariational} for the nonergodic convergence rate of the variational inequality associated to this problem.  Note that for variational inequalities, only upper bounds are important, because we only wish to make certain quantities negative.
\item {\bf Subdifferential + Skew:} We use the same set up as above, except we assume that $\overline{A}$ and $\overline{B}$ are skew linear mappings (i.e., $A^\ast = -A$ and $B^\ast = -B$) and $\overline{C} = 0$. See Corollaries~\ref{cor:ergodic1variational} and~\ref{cor:ergodic2variational} for the ergodic convergence rate of the variational inequality associated to this problem. This inclusion problem arises in primal-dual operator-splitting algorithms.
\item {\bf Functions}: We assume that $\overline{A} = \overline{B} = \overline{C} \equiv 0$. See Corollary~\ref{cor:nonergodicfunction} for the nonergodic convergence rate and see Corollaries~\ref{cor:ergodic1function} and~\ref{cor:ergodic2function} for the ergodic convergence rates of the function values associated to our method.
\end{itemize}

Note that by~\cite[Theorem 11]{davis2014convergence}, all of the convergence rates below are sharp (in terms of order, but not necessarily in terms of constants). In addition, they generalize some of the known convergence rates provided in~\cite{davis2014convergence,davis2014convergenceprimaldual,davis2014convergenceFDRS} for Douglas-Rachford splitting, forward-Douglas-Rachford splitting, and the primal-dual forward-backward splitting, Douglas-Rachford splitting, and the proximal-point algorithms.

The following fact will be used several times:
\begin{lemma}\label{lem:Lipschitz}
Suppose that $(z^j)_{j \geq 0}$ is generated by Equation~\eqref{eq:zitr} and $\gamma \in (0, 2\beta)$. Let $z^\ast$ be a fixed point of $T$ and let $x^\ast = J_{\gamma B}(z^\ast)$.  Then $(x_A^j)_{j \geq 0}$ and $(x_B^j)_{j \geq 0}$ are contained within the closed ball $\overline{B (x^\ast, (1+\gamma/\beta) \|z^0 - z^\ast\|)}$.
\end{lemma}
\begin{proof}
Fix $k \geq 0$. Observe that
\begin{align*}
\|x_B^k - x^\ast\| =  \|J_{\gamma B}(z^k) - J_{\gamma B}(z^\ast)\|  \leq \| z^k - z^\ast\| \leq \|z^0 - z^\ast\|
\end{align*}
by Part~\ref{thm:convergence:part:biginequality} of Theorem~\ref{thm:convergence}.  Similarly,
\begin{align*}
\|x_A^k - x^\ast\| &\leq \|\refl_{\gamma B}(z^k) - \refl_{\gamma B}(z^\ast) + \gamma Cx^\ast - \gamma Cx_B\| \leq \|z^k - z^\ast\| + \frac{\gamma}{\beta}\|z^k - z^\ast\| \leq \left(1+\frac{\gamma}{\beta}\right)\|z^0 - z^\ast\|.
\end{align*}
\qed\end{proof}

\begin{corollary}[Nonergodic convergence of function values]\label{cor:nonergodicfunction}
Suppose that $(z^j)_{j \geq 0}$ is generated by Equation~\eqref{eq:zitr}, with $A = \partial f, B = \partial g$ and $C = \nabla h$. Let the assumptions be as in Theorem~\ref{thm:nonergodicmain}. Then the following convergence rates hold:
\begin{enumerate}
\item \label{cor:nonergodicfunction:part:general} For all $k \geq 0$, we have
\begin{align*}
\frac{-\|z^0 - z^\ast\|\|u_B^\ast + Cx^\ast\|}{\sqrt{\underline{\tau}(k+1)}} &\leq f(x_f^k) + g(x_g^k) + h(x_g^k) - (f + g + h)(x^\ast)  \\
&\leq \frac{(\|z^{\ast} - x^\ast\| + (1+ \gamma/\beta)\|z^0 - z^\ast\|+ \gamma \|\nabla h( x^\ast)\|) \|z^0 - z^\ast\|}{\gamma\sqrt{\underline{\tau}(k+1)}}
\end{align*}
and
\begin{align*}
|f(x_f^k) + g(x_g^k) + h(x_g^k) - (f+g+h)(x^\ast)| &= o\left(\frac{1}{\sqrt{k+1}}\right).
\end{align*}
\item \label{cor:nonergodicfunction:part:Lipschitz} Suppose that $f$ is $L$-Lipschitz continuous on the closed ball $\overline{B(0, (1+\gamma/\beta)\|z^0 - z^\ast\|)}$. Then the following convergence rate holds:
\begin{align*}
0 \leq &f(x_g^k) + g(x_g^k) + h(x_g^k) - (f+g+h)(x^\ast) \\
&\leq \frac{(\|z^{\ast} - x^\ast\| + (1+ \gamma/\beta)\|z^0 - z^\ast\|+ \gamma \|\nabla h( x^\ast)\|) \|z^0 - z^\ast\| + \gamma L\|z^0 - z^\ast\|}{\gamma\sqrt{\underline{\tau}(k+1)}}
\end{align*}
and
\begin{align*}
0 \leq f(x_g^k) + g(x_g^k) + h(x_g^k) - (f+g+h)(x^\ast) &= o\left(\frac{1}{\sqrt{k+1}}\right).\end{align*}
\end{enumerate}
\end{corollary}
\begin{proof}
Fix $k \geq 0$.

Part~\ref{cor:nonergodicfunction:part:general}: By Corollary~\ref{cor:fundamentalequalityfunctions}, we have
\begin{align*}
\dotp{x_B^k - x_A^k, u_B^\ast + Cx^\ast} \leq f(x_f^k) + g(x_g^k) + h(x_g^k) - (f+g+h)(x^\ast) \leq \frac{1}{2\gamma}\kappa_1^k(1, x^\ast)
\end{align*}
Thus, the convergence rates follow directly from Theorem~\ref{thm:nonergodicmain}.

Part~\ref{cor:nonergodicfunction:part:Lipschitz}: Note that $f(x_g^k) - f(x_f^k) \leq L \|x_f^k - x_g^k\|$ by Lemma~\ref{lem:Lipschitz}. Because $x_f - x_g = z^k - Tz^k$, we have
\begin{align*}
f(x_g^k) + g(x_f^k) + h(x_g^k) - (f+g+h)(x^\ast)  &\leq f(x_f^k) + g(x_g^k) + h(x_g^k) - (f+g+h)(x^\ast) + L\|x_f^k - x_g^k\| \\
&\stackrel{\eqref{thm:convergence:part:convergencerate}}{\leq} f(x_f^k) + g(x_g^k) + h(x_g^k) - (f+g+h)(x^\ast) +\frac{\|z^0 - z^\ast\|}{\sqrt{\underline{\tau}(k+1)}}.
\end{align*}
Thus, the rate follows by Part~\ref{cor:nonergodicfunction:part:general}.
\qed\end{proof}

\begin{corollary}[Nonergodic convergence of variational inequalities]\label{cor:nonergodicvariational}
Suppose that $(z^j)_{j \geq 0}$ is generated by Equation~\eqref{eq:zitr}, with $A = \partial f + \overline{A}, B = \partial g + \overline{B}$ and $C = \nabla h + \overline{C}$ as in Corollary~\ref{cor:fundamentalequalityvariational}. Let the assumptions be as in Theorem~\ref{thm:nonergodicmain}. Then the following convergence rates hold:
\begin{enumerate}
\item \label{cor:nonergodicvariational:part:general} For all $k \geq 0$ and $x \in \dom(f) \cap \dom(g)$, we have
\begin{align*}
f(x_A^k) + g(x_B^k) + h(x_B^k) - (f + g + h)(x) &+ \dotp{x_A^k - x, u_{\overline{A}}^k} + \dotp{ x_B^k - x, u_{\overline{B}}^k + \overline{C} x_B^k}  \\
&\leq \frac{(\|z^{\ast} - x\| + (1+ \gamma/\beta)\|z^0 - z^\ast\|+ \gamma \|Cx^\ast\|) \|z^0 - z^\ast\|}{\gamma\sqrt{\underline{\tau}(k+1)}}
\end{align*}
\item \label{cor:nonergodicvariational:part:Lipschitz} Suppose that $f$ and $\overline{A}$ are $L_f$ and $L_{\overline{A}}$-Lipschitz continuous respectively on the closed ball $\overline{B(0, (1+\gamma/\beta)\|z^0 - z^\ast\|)}$. Then the following convergence rate holds: For all $k \geq 0$ and $x \in \dom(f) \cap \dom(g)$, we have
\begin{align*}
&f(x_B^k) + g(x_B^k) + h(x_B^k) - (f+g+h)(x) + \dotp{ x_B^k - x,  \overline{A}x_B^k + u_{\overline{B}}^k + \overline{C} x_B^k}  \\
&\leq \frac{(\|z^{\ast} - x^\ast\| + (1+ \gamma/\beta)\|z^0 - z^\ast\|+ \gamma \|Cx^\ast\|) \|z^0 - z^\ast\| + \gamma L_f\|z^0 - z^\ast\|}{\gamma\sqrt{\underline{\tau}(k+1)}} \\
&+ \frac{(1 + L_A)\|z^0 - z^\ast\|\left((1+\gamma/\beta) (1+L_A )\| z^0 - z^\ast\| + \|\overline{A}x^\ast\| +  \|x^\ast - x\||\right)}{\sqrt{\underline{\tau}(k+1)}}
\end{align*}
and
\begin{align*}
f(x_B^k) + g(x_B^k) + h(x_B^k) - (f+g+h)(x) + \dotp{ x_B^k - x,  \overline{A}x_B^k + u_{\overline{B}}^k + \overline{C} x_B^k} &= o\left(\frac{1+\|x\|}{\sqrt{k+1}}\right).\end{align*}
\end{enumerate}
\end{corollary}
\begin{proof}
Fix $k \geq 0$.

Part~\ref{cor:nonergodicfunction:part:general}: By Corollary~\ref{cor:fundamentalequalityvariational}, we have
\begin{align*}
f(x_A^k) + g(x_B^k) + h(x_B^k) - (f + g + h)(x) &+ \dotp{x_A^k - x, u_{\overline{A}}^k} + \dotp{ x_B^k - x, u_{\overline{B}}^k + \overline{C} x_B^k} \leq \frac{1}{2\gamma}\kappa_1^k(1, x)
\end{align*}
Thus, the convergence rates follow directly from Theorem~\ref{thm:nonergodicmain}.

Part~\ref{cor:nonergodicfunction:part:Lipschitz}: Note that $f(x_B^k) - f(x_A^k) \leq L_f \|x_A^k - x_B^k\|$ by Lemma~\ref{lem:Lipschitz}. Because $x_B^k - x_A^k = z^k - Tz^k$, we have
\begin{align*}
f(x_B^k) + g(x_B^k) + h(x_B^k) - (f+g+h)(x)  &\leq f(x_A^k) + g(x_B^k) + h(x_B^k) - (f+g+h)(x) + L_f\|x_A^k - x_B^k\| \\
&\stackrel{\eqref{thm:convergence:part:convergencerate}}{\leq} f(x_A^k) + g(x_B^k) + h(x_B^k) - (f+g+h)(x) +\frac{\|z^0 - z^\ast\|}{\sqrt{\underline{\tau}(k+1)}}.
\end{align*}
Also,
\begin{align*}
&\dotp{x_A^k - x, \overline{A}x_A^k}\\
&= \dotp{x_A^k - x_B^k, \overline{A}x_A^k} + \dotp{x_B^k - x, \overline{A}x_A^k} \\
&= \dotp{x_A^k - x_B^k, \overline{A}x_A^k} + \dotp{x_B^k - x, \overline{A}x_A^k - \overline{A}x_B^k} + \dotp{ x_B^k - x, \overline{A}x_B^k} \\
&\leq \|x_A^k - x_B^k\|\|\overline{A}x_A^k\| + \|x_B^k - x\|\|\overline{A}x_A^k - \overline{A}x_B^k\| + \dotp{ x_B^k - x, \overline{A}x_B^k} \\
&\stackrel{\eqref{thm:convergence:part:convergencerate}}{\leq} \frac{(1 + L_{\overline{A}})\|z^0 - z^\ast\|\left(\|\overline{A}x_A^k\| + \|x_B^k - x\|\right)}{\sqrt{\underline{\tau}(k+1)}}  + \dotp{ x_B^k - x, \overline{A}x_B^k}
\end{align*}
and for $x^\ast = J_{\gamma B}(z^\ast)$,
\begin{align*}
\|\overline{A}x_A^k\| + \|x_B^k - x\| &\leq \|\overline{A}x_A^k - \overline{A}x^\ast\| + \|\overline{A}x^\ast\| + \|x_B^k - x^\ast\| + \|x^\ast - x\|  \\
&\leq (1+\gamma/\beta) (1+L_{\overline{A}})\| z^0 - z^\ast\| + \|\overline{A}x^\ast\| +  \|x^\ast - x\|
\end{align*}
Thus, the rate follows by Part~\ref{cor:ergodicfunction:part:general}.
\qed\end{proof}

\begin{corollary}[Ergodic convergence rates of function values for Equation~\eqref{avg1}]\label{cor:ergodic1function}
Suppose that $(z^j)_{j \geq 0}$ is generated by Equation~\eqref{eq:zitr}, with $A = \partial f, B = \partial g$ and $C = \nabla h$. Let the assumptions be as in Theorem~\ref{thm:ergodic1}. For all $k \geq 0$, let $\overline{x}_f^k = (1/\sum_{i=0}^k \lambda_i) \sum_{i=0}^k \lambda_i x_f^i$, and let $\overline{x}_g^k = (1/\sum_{i=0}^k \lambda_i) \sum_{i=0}^k \lambda_i x_g^i$. Let $x^\ast = J_{\gamma B}(z^\ast)$. Then the following convergence rates hold:
\begin{enumerate}
\item \label{cor:ergodicfunction:part:general} For all $k \geq 0$, we have
\begin{align*}
\frac{-2\|z^0 - z^{\ast}\|\|u_B^\ast + Cx^\ast\|}{\sum_{i=0}^k \lambda_i} &\leq f(\overline{x}_f^k) + g(\overline{x}_g^k) + h(\overline{x}_g^k) - (f + g + h)(x^\ast)  \\
&\leq\frac{\|z^0 - x^\ast\|^2 + \frac{\gamma}{(2\beta\varepsilon - \gamma)}\|z^0 - z^\ast\|^2 + 4\gamma \|z^0 -z^\ast\|\|Cx^\ast\|}{2\gamma \sum_{i=0}^k \lambda_i}.
\end{align*}
\item \label{cor:ergodicfunction:part:Lipschitz} Suppose that $f$ is $L$-Lipschitz continuous on the closed ball $\overline{B(0, (1+\gamma/\beta)\|z^0 - z^\ast\|)}$. Then the following convergence rate holds:
\begin{align*}
0 \leq &f(\overline{x}_g^k) + g(\overline{x}_g^k) + h(\overline{x}_g^k) - (f+g+h)(x^\ast) \\
&\leq \frac{\|z^0 - x^\ast\|^2 + \frac{\gamma}{(2\beta\varepsilon - \gamma)}\|z^0 - z^\ast\|^2 + 4\gamma \|z^0 -z^\ast\|\|Cx^\ast\| + 4\gamma L\|z^0 - z^\ast\|}{2\gamma \sum_{i=0}^k \lambda_i}
\end{align*}
\end{enumerate}
\end{corollary}
\begin{proof}
Fix $k \geq 0$.

Part~\ref{cor:ergodicfunction:part:general}: We have the lower bound:
\begin{align*}
&f(\overline{x}_f^k) + g(\overline{x}_g^k) + h(\overline{x}_g^k) - (f + g + h)(x^\ast) \\
&\geq \dotp{\overline{x}_f^k - x^\ast, \tnabla f(x^\ast)} + \dotp{ \overline{x}_g^k - x^\ast,  \tnabla g(x^\ast) + \nabla h(x^\ast)} \\
&= \dotp{ \overline{x}_g^k - \overline{x}_f^k, \tnabla g(x^\ast) + \nabla h(x^\ast)}.
\end{align*}
where $\tnabla g(x^\ast) + \tnabla f(x^\ast) + \nabla h(x^\ast) = 0$. In addition,
\begin{align*}
f(\overline{x}_f^k) + g(\overline{x}_g^k) + h(\overline{x}_g^k) - (f + g + h)(x^\ast)&\leq \frac{1}{2\gamma \sum_{i=0}^k \lambda_i}\sum_{i=0}^k \lambda_i \kappa_1^i(\lambda_i, x^\ast)
\end{align*}
by Jensen's inequality and Corollary~\ref{cor:fundamentalequalityfunctions}. Thus, the convergence rate follows by Theorem~\ref{thm:ergodic1}.

Part~\ref{cor:ergodicfunction:part:Lipschitz}: Note that $f(\overline{x}_g^k) - f(\overline{x}_f^k) \leq L \|\overline{x}_f^k - \overline{x}_g^k\|$ by Lemma~\ref{lem:Lipschitz} because $\overline{B(x^\ast, (1+\gamma/\beta)\|z^0 - z^\ast\|)}$ is convex so the averaged sequences $(\overline{x}_f)_{j \geq 0}$ and $(\overline{x}_g)_{j \geq 0}$ must continue to lie in the ball. Therefore,
\begin{align*}
f(\overline{x}_g^k) + g(\overline{x}_g^k) + h(\overline{x}_g^k) - (f+g+h)(x^\ast)  &\leq f(\overline{x}_f^k) + g(\overline{x}_g^k) + h(\overline{x}_g^k) - (f+g+h)(x^\ast) + L\|\overline{x}_f^k- \overline{x}_g^k\| \\
&\stackrel{\eqref{eq:feasibilitySCHEME1}}{\leq} f(\overline{x}_f^k) + g(\overline{x}_g^k) + h(\overline{x}_g^k) - (f+g+h)(x^\ast) +\frac{2\|z^0 - z^\ast\|}{\sum_{i=0}^k \lambda_i}.
\end{align*}
Thus, the rate follows by Part~\ref{cor:ergodicfunction:part:general}.\qed
\end{proof}

\begin{corollary}[Ergodic convergence of variational inequalities for Equation~\eqref{avg1}]\label{cor:ergodic1variational}
Suppose that $(z^j)_{j \geq 0}$ is generated by Equation~\eqref{eq:zitr}, with $A = \partial f + \overline{A}, B = \partial g + \overline{B}$ and $C = \nabla h + \overline{C}$ as in Corollary~\ref{cor:fundamentalequalityvariational}. In addition, suppose that $\overline{A}$ and $\overline{B}$ are skew linear maps (i.e., $A^\ast = -A$, and $B^\ast = -B$), and suppose that $\overline{C} \equiv 0$. Let the assumptions be as in Theorem~\ref{thm:ergodic1}. For all $k \geq 0$, let $\overline{x}_A^k = (1/\sum_{i=0}^k \lambda_i) \sum_{i=0}^k \lambda_i x_A^i$, and let $\overline{x}_B^k = (1/\sum_{i=0}^k \lambda_i) \sum_{i=0}^k \lambda_i x_B^i$. Then the following convergence rates hold:
\begin{enumerate}
\item \label{cor:ergodicvariational:part:general} For all $k \geq 0$ and $x \in \dom(f) \cap \dom(g)$, we have
\begin{align*}
f(\overline{x}_A^k) &+ g(\overline{x}_B^k) + h(\overline{x}_B^k) - (f + g + h)(x) + \dotp{- x, \overline{A}\overline{x}_B^k + \overline{B} \overline{x}_B^k}  \\
&\leq \frac{\|z^0 - x\|^2 + \frac{\gamma}{(2\beta\varepsilon - \gamma)}\|z^0 - z^\ast\|^2 + 4\gamma \|z^0 -z^\ast\|\|Cx^\ast\| + 4\gamma \|\overline{A}\|\|x\|\|\|z^0 - z^\ast\|}{2\gamma \sum_{i=0}^k \lambda_i}
\end{align*}
\item \label{cor:ergodicvariational:part:Lipschitz} Suppose that $f$ is $L_f$-Lipschitz continuous on the closed ball $\overline{B(0, (1+\gamma/\beta)\|z^0 - z^\ast\|)}$. Then the following convergence rate holds: For all $k \geq 0$ and $x \in \dom(f) \cap \dom(g)$, we have
\begin{align*}
&f(\overline{x}_B^k) + g(\overline{x}_B^k) + h(\overline{x}_B^k) - (f+g+h)(x) + \dotp{- x, \overline{A}\overline{x}_B^k + \overline{B} \overline{x}_B^k}  \\
&\leq \frac{\|z^0 - x\|^2 + \frac{\gamma}{(2\beta\varepsilon - \gamma)}\|z^0 - z^\ast\|^2 + 4\gamma \|z^0 -z^\ast\|\|Cx^\ast\| + 4\gamma\|\overline{A}\|\|x\|\|\|z^0 - z^\ast\| + 4\gamma L_f\|z^0 - z^\ast\|}{2\gamma \sum_{i=0}^k \lambda_i}.
\end{align*}
\end{enumerate}
\end{corollary}
\begin{proof}
Fix $k \geq 0$.

Part~\ref{cor:ergodicvariational:part:general}: By Corollary~\ref{cor:fundamentalequalityvariational}, we have
\begin{align*}
&f(\overline{x}_A^k) + g(\overline{x}_B^k) + h(\overline{x}_B^k) - (f + g + h)(x) + \dotp{-x, \overline{A}_B\overline{x}_B^k +  \overline{B}\overline{x}_B^k } \\
&\leq \frac{1}{2\gamma\sum_{i=0}^k \lambda_i}\sum_{i=0}^k \lambda_i \kappa_1^k(\lambda_i, x) + \dotp{x, \overline{A}\left(\overline{x}_B^k - \overline{x}_A^k\right)} \\
&\stackrel{\eqref{eq:feasibilitySCHEME1}}{\leq} \frac{1}{2\gamma\sum_{i=0}^k \lambda_i}\sum_{i=0}^k \lambda_i \kappa_1^k(\lambda_i, x) + \frac{2\|A\|\|x\|\|z^0 - z^\ast\|}{\sum_{i=0}^k \lambda_i}
\end{align*}
where we use the self orthogonality of skew symmetric maps ($\dotp{\overline{A}y, y} = \dotp{\overline{B}y, y} = 0$ for all $y \in \cH$) and Jensen's inequality. Thus, the convergence rates follow directly from Theorem~\ref{thm:ergodic1}.

Part~\ref{cor:ergodicvariational:part:Lipschitz}: Note that $f(\overline{x}_B^k) - f(\overline{x}_A^k) \leq L_f \|\overline{x}_A^k - \overline{x}_B^k\|$ by Lemma~\ref{lem:Lipschitz}. Therefore,
\begin{align*}
f(\overline{x}_B^k) + g(\overline{x}_B^k) + h(\overline{x}_B^k) - (f+g+h)(x)  &\leq f(\overline{x}_A^k) + g(\overline{x}_B^k) + h(\overline{x}_B^k) - (f+g+h)(x) + L_f\|x_A^k - x_B^k\| \\
&\stackrel{\eqref{eq:feasibilitySCHEME1}}{\leq} f(\overline{x}_A^k) + g(\overline{x}_B^k) + h(\overline{x}_B^k) - (f+g+h)(x) +\frac{2L_f\|z^0 - z^\ast\|}{\sum_{i=0}^k \lambda_i}.
\end{align*}
Thus, the rate follows by Part~\ref{cor:ergodicvariational:part:general}.
\qed\end{proof}

\begin{corollary}[Ergodic convergence rates of function values for Equation~\eqref{avg2}]\label{cor:ergodic2function}
Suppose that $(z^j)_{j \geq 0}$ is generated by Equation~\eqref{eq:zitr}, with $A = \partial f, B = \partial g$ and $C = \nabla h$. Let the assumptions be as in Theorem~\ref{thm:ergodic2}. For all $k \geq 0$, let $\overline{x}_f^k = (2/((k+1)(k+2))) \sum_{i=0}^k (i+1)x_f^i$, and let $\overline{x}_g^k = (2/((k+1)(k+2)) \sum_{i=0}^k (i+1) x_g^i$. Let $x^\ast = J_{\gamma B}(z^\ast)$. Then the following convergence rates hold:
\begin{enumerate}
\item \label{cor:ergodic2function:part:general} For all $k \geq 0$, we have
\begin{align*}
&\frac{-5\|z^0 - z^\ast\|}{\lambda(k+1)}\leq f(\overline{x}_f^k) + g(\overline{x}_g^k) + h(\overline{x}_g^k) - (f + g + h)(x^\ast)  \\
&\leq \frac{2\|z^\ast - x^\ast\|^2 + \left(2 + \frac{\gamma}{(2\beta\varepsilon - \gamma)}\right) \|z^0 - z^\ast\|^2 + 10\|z^0 - z^\ast\|\|Cx^\ast\|}{\gamma\lambda(k+1)}.
\end{align*}
\item \label{cor:ergodic2function:part:Lipschitz} Suppose that $f$ is $L$-Lipschitz continuous on the closed ball $\overline{B(0, (1+\gamma/\beta)\|z^0 - z^\ast\|)}$. Then the following convergence rate holds:
\begin{align*}
0 \leq &f(\overline{x}_g^k) + g(\overline{x}_g^k) + h(\overline{x}_g^k) - (f+g+h)(x^\ast) \\
&\leq \frac{2\|z^\ast - x^\ast\|^2 + \left(2 + \frac{\gamma}{(2\beta\varepsilon - \gamma)}\right) \|z^0 - z^\ast\|^2 + 10\|z^0 - z^\ast\|\|Cx^\ast\| + 5\gamma L_f\|z^0 - z^\ast\|}{\gamma\lambda(k+1)}.
\end{align*}
\end{enumerate}
\end{corollary}
\begin{proof}
Fix $k \geq 0$.

Part~\ref{cor:ergodic2function:part:general}: We have the lower bound:
\begin{align*}
&f(\overline{x}_f^k) + g(\overline{x}_g^k) + h(\overline{x}_g^k) - (f + g + h)(x^\ast) \\
&\geq \dotp{\overline{x}_f^k - x^\ast, \tnabla f(x^\ast)} + \dotp{ \overline{x}_g^k - x^\ast,  \tnabla g(x^\ast) + \nabla h(x^\ast)} \\
&= \dotp{ \overline{x}_g^k - \overline{x}_f^k, \tnabla g(x^\ast) + \nabla h(x^\ast)}.
\end{align*}
where $\tnabla g(x^\ast) + \tnabla f(x^\ast) + \nabla h(x^\ast) = 0$. In addition,
\begin{align*}
f(\overline{x}_f^k) + g(\overline{x}_g^k) + h(\overline{x}_g^k) - (f + g + h)(x^\ast)&\leq \frac{2}{2\gamma\lambda(k+1)(k+2)}\sum_{i=0}^k (i+1) \kappa_1^i(\lambda, x^\ast)
\end{align*}
by Jensen's inequality and Corollary~\ref{cor:fundamentalequalityfunctions}. Thus, the convergence rate follows by Theorem~\ref{thm:ergodic2}.

Part~\ref{cor:ergodic2function:part:Lipschitz}: Note that $f(\overline{x}_g^k) - f(\overline{x}_f^k) \leq L \|\overline{x}_f^k - \overline{x}_g^k\|$ by Lemma~\ref{lem:Lipschitz} because $\overline{B(x^\ast, (1+\gamma/\beta)\|z^0 - z^\ast\|)}$ is convex, so the averaged sequences $(\overline{x}_f)_{j \geq 0}$ and $(\overline{x}_g)_{j \geq 0}$ must continue to lie in the ball. Therefore,
\begin{align*}
f(\overline{x}_g^k) + g(\overline{x}_g^k) + h(\overline{x}_g^k) - (f+g+h)(x^\ast)  &\leq f(\overline{x}_f^k) + g(\overline{x}_g^k) + h(\overline{x}_g^k) - (f+g+h)(x^\ast) + L\|\overline{x}_f^k- \overline{x}_g^k\| \\
&\stackrel{\eqref{eq:feasibilitySCHEME1}}{\leq} f(\overline{x}_f^k) + g(\overline{x}_g^k) + h(\overline{x}_g^k) - (f+g+h)(x^\ast) +\frac{5L\|z^0 - z^\ast\|}{\lambda(k+1)}.
\end{align*}
Thus, the rate follows by Part~\ref{cor:ergodic2function:part:general}.\qed
\end{proof}

\begin{corollary}[Ergodic convergence of variational inequalities for Equation~\eqref{avg2}]\label{cor:ergodic2variational}
Suppose that $(z^j)_{j \geq 0}$ is generated by Equation~\eqref{eq:zitr}, with $A = \partial f + \overline{A}, B = \partial g + \overline{B}$ and $C = \nabla h + \overline{C}$ as in Corollary~\ref{cor:fundamentalequalityvariational}. In addition, suppose that $\overline{A}$ and $\overline{B}$ are skew linear maps (i.e., $A^\ast = -A$, and $B^\ast = -B$), and suppose that $\overline{C} \equiv 0$. Let the assumptions be as in Theorem~\ref{thm:ergodic2}. For all $k \geq 0$, let $\overline{x}_A^k = (2/((k+1)(k+2))) \sum_{i=0}^k (i+1) x_A^i$, and let $\overline{x}_B^k = (2/((k+1)(k+2))) \sum_{i=0}^k (i+1) x_B^i$. Then the following convergence rates hold:
\begin{enumerate}
\item \label{cor:ergodic2variational:part:general} For all $k \geq 0$ and $x \in \dom(f) \cap \dom(g)$, we have
\begin{align*}
f(\overline{x}_A^k) &+ g(\overline{x}_B^k) + h(\overline{x}_B^k) - (f + g + h)(x) + \dotp{- x, \overline{A}\overline{x}_B^k + \overline{B} \overline{x}_B^k}  \\
&\leq \frac{2\|z^\ast - x\|^2 + \left(2 + \frac{\gamma}{(2\beta\varepsilon - \gamma)}\right) \|z^0 - z^\ast\|^2 + 10\|z^0 - z^\ast\|\|Cx^\ast\| + 5\gamma \|\overline{A}\|\|x\|\|\|z^0 - z^\ast\|}{\gamma\lambda(k+1)}
\end{align*}
\item \label{cor:ergodic2variational:part:Lipschitz} Suppose that $f$ is $L_f$-Lipschitz continuous on the closed ball $\overline{B(0, (1+\gamma/\beta)\|z^0 - z^\ast\|)}$. Then the following convergence rate holds: For all $k \geq 0$ and $x \in \dom(f) \cap \dom(g)$, we have
\begin{align*}
&f(\overline{x}_B^k) + g(\overline{x}_B^k) + h(\overline{x}_B^k) - (f+g+h)(x) + \dotp{- x, \overline{A}\overline{x}_B^k + \overline{B} \overline{x}_B^k}  \\
&\leq \frac{2\|z^\ast - x\|^2 + \left(2 + \frac{\gamma}{(2\beta\varepsilon - \gamma)}\right) \|z^0 - z^\ast\|^2 + 10\|z^0 - z^\ast\|\|Cx^\ast\| + 5\gamma \|\overline{A}\|\|x\|\|\|z^0 - z^\ast\| + 5\gamma L_f\|z^0 - z^\ast\|}{\gamma\lambda(k+1)}
\end{align*}
\end{enumerate}
\end{corollary}
\begin{proof}
Fix $k \geq 0$.

Part~\ref{cor:ergodic2variational:part:general}: By Corollary~\ref{cor:fundamentalequalityvariational}, we have
\begin{align*}
&f(\overline{x}_A^k) + g(\overline{x}_B^k) + h(\overline{x}_B^k) - (f + g + h)(x) + \dotp{-x, \overline{A}_B\overline{x}_B^k +  \overline{B}\overline{x}_B^k } \\
&\leq\frac{2}{2\gamma\lambda(k+1)(k+2)}\sum_{i=0}^k (i+1) \kappa_1^k(\lambda, x) + \dotp{x, \overline{A}\left(\overline{x}_B^k - \overline{x}_A^k\right)} \\
&\stackrel{\eqref{eq:feasibilitySCHEME2}}{\leq} \frac{2}{2\gamma\lambda(k+1)(k+2)}\sum_{i=0}^k (i+1) \kappa_1^k(\lambda, x) + \frac{5\|A\|\|x\|\|z^0 - z^\ast\|}{\lambda(k+1)}
\end{align*}
where we use the self orthogonality of skew symmetric maps ($\dotp{\overline{A}y, y} = \dotp{\overline{B}y, y} = 0$ for all $y \in \cH$) and Jensen's inequality. Thus, the convergence rates follow directly from Theorem~\ref{thm:ergodic2}.

Part~\ref{cor:ergodic2variational:part:Lipschitz}: Note that $f(\overline{x}_B^k) - f(\overline{x}_A^k) \leq L_f \|\overline{x}_A^k - \overline{x}_B^k\|$ by Lemma~\ref{lem:Lipschitz}. Therefore,
\begin{align*}
f(\overline{x}_B^k) + g(\overline{x}_B^k) + h(\overline{x}_B^k) - (f+g+h)(x)  &\leq f(\overline{x}_A^k) + g(\overline{x}_B^k) + h(\overline{x}_B^k) - (f+g+h)(x) + L_f\|\overline{x}_A^k - \overline{x}_B^k\| \\
&\stackrel{\eqref{eq:feasibilitySCHEME2}}{\leq} f(\overline{x}_A^k) + g(\overline{x}_B^k) + h(\overline{x}_B^k) - (f+g+h)(x) +\frac{5L_f\|z^0 - z^\ast\|}{\lambda(k+1)}.
\end{align*}
Thus, the rate follows by Part~\ref{cor:ergodic2variational:part:general}.
\qed\end{proof}

\subsection{Strong monotonicity}\label{sec:strongconvergence}

In this section, we deduce the convergence rates of the terms $Q_\cdot(\cdot, \cdot)$ under general assumptions.

\begin{corollary}[Strong convergence]\label{thm:strong}
Suppose that $(z^j)_{j \geq 0}$ is generated by Equation~\eqref{eq:zitr}. Let $z^\ast$ be a fixed point of $T$ and let $x^\ast = J_{\gamma B}(z^\ast)$.  Then for all $k \geq 0$, the following convergence rates hold:
\begin{enumerate}
\item {\bf Nonergodic convergence:} Let the assumptions of Theorem~\ref{thm:nonergodicmain} hold. Then
\begin{align*}
Q_A(x_A^k, x^\ast) + Q_B(x_B^k, x^\ast) + Q_C(x_B^k, x^\ast) &\leq \frac{(1 + \gamma/\beta)\|z^0 - z^\ast\|^2}{\gamma\sqrt{\underline{\tau}(k+1)}}
\end{align*}
and $Q_A(x_A^k, x^\ast) + Q_B(x_B^k, x^\ast) + Q_C(x_C^k, x^\ast) = o\left(1/\sqrt{k+1}\right)$.
\item \label{thm:strong:part:best} {\bf ``Best" iterate convergence:} Let the assumptions of Theorem~\ref{thm:nonergodicmain} hold. Suppose that $\underline{\lambda} := \inf_{j \geq 0} \lambda_j$. Then
\begin{align*}
\min_{i = 0, \cdots, k} \left\{Q_A(x_A^i, x^\ast) + Q_B(x_B^i, x^\ast) + Q_C(x_B^i, x^\ast)\right\} &\leq \frac{\left(1 + \frac{\gamma}{(2\beta\varepsilon - \gamma)}\right)\|z^0 - z^\ast\|^2}{2\gamma\underline{\lambda} (k+1) }
\end{align*}
and $\min_{i = 0, \cdots, k} \left\{Q_A(x_A^i, x^\ast) + Q_B(x_B^i, x^\ast) + Q_C(x_C^i, x^\ast)\right\} = o\left(1/(k+1)\right)$.
\item {\bf Ergodic convergence for Equation~\eqref{avg1}}: Let the assumptions for Theorem~\ref{thm:ergodic1} hold. Then
\begin{align*}
\frac{1}{\sum_{i=0}^k \lambda_i}\sum_{i=0}^k \lambda_i\left(Q_A(x_A^k, x^\ast) + Q_B(x_B^k, x^\ast) + Q_C(x_B^k, x^\ast)\right) &\leq \frac{\left(1 + \frac{\gamma}{(2\beta\varepsilon - \gamma)}\right)\|z^0 - z^\ast\|^2}{2\gamma \sum_{i=0}^k \lambda_i}
\end{align*}
\item  {\bf Ergodic convergence for Equation~\eqref{avg2}}: Let the assumptions for Theorem~\ref{thm:ergodic2} hold. Then
\begin{align*}
\frac{2}{(k+1)(k+2)}\sum_{i=0}^k (i+1)\left(Q_A(x_A^k, x^\ast) + Q_B(x_B^k, x^\ast) + Q_C(x_C^k, x^\ast)\right) &\leq \frac{\left(1 + \frac{\gamma}{(2\beta\varepsilon - \gamma)}\right)\|z^0 - z^\ast\|^2}{\gamma \lambda (k+1)}
\end{align*}
\end{enumerate}
\end{corollary}
\begin{proof}
The ``best" iterate convergence result follows~\cite[Lemma 3]{davis2014convergence} because $\sum_{i=0}^\infty 2\gamma \underline{\lambda}(Q_A(x_A^i, x^\ast) + Q_B(x_B^i, x^\ast) + Q_C(x_B^i, x^\ast)) \leq \sum_{i=0}^\infty 2\gamma \lambda_i(Q_A(x_A^i, x^\ast) + Q_B(x_B^i, x^\ast) + Q_C(x_B^i, x^\ast)) \leq \sum_{i=0}^\infty \lambda_i \kappa_2^k(\lambda_i, x^\ast) \leq \left(1 + {\gamma}/{(2\beta\varepsilon - \gamma)}\right)$ by the upper bounds in Equations~\eqref{eq:linearconvergenceinequality} and~\eqref{eq:ergodickappauppermainSTRONG}.

The rest of the results follow by combining the upper bound in Equation~\eqref{eq:linearconvergenceinequality} with the convergence rates in Theorems~\ref{thm:nonergodicmain},~\ref{thm:ergodic1}, and~\ref{thm:ergodic2}.
\qed\end{proof}

At first glance it may be seem that the ergodic bounds in Theorem~\ref{thm:strong} are not meaningful.  However, whenever $\mu_A > 0$, we can apply Jensen's inequality to show that
\begin{align*}
\sum_{i=0}^k \nu_j Q_A(x_A^k, x^\ast) \geq \mu_A\left\|\sum_{i=0}^k \nu_i x_A^i - x^\ast\right\|^2
\end{align*}
for any positive sequence of stepsizes $(\nu_j)_{j =0}^k$, such that $\sum_{i=0}^k \nu_i = 1$. Thus, the ergodic bounds really prove strong convergence rates for the ergodic iterates generated by Equations~\eqref{avg1} and~\eqref{avg2}.

\subsection{Lipschitz differentiability}\label{sec:Lipschitz}

In this section, we focus on function minimization. In particular, we let $A = \partial f$, $B = \partial g$, and $C = \nabla h$, where $f, g$ and $h$ are closed, proper, and convex, and $\nabla h$ is $(1/\beta)$-Lipschitz.  We make the following assumption regarding the regularity of $f$:
\begin{align*}
\text{The gradient of at least one of $f$ is Lipschitz.}
\end{align*}
Under this assumption we will show that $$\text{the ``best" objective error after $k$ iterations of Equation~\eqref{eq:zitr} has order $o\left(1/(k+1)\right))$}.$$  The techniques of this section can also be applied to show a similar result for $g$. The proof is somewhat more technical, so we omit it.

The following theorem will be used several times throughout our analysis. See \cite[Theorem 18.15(iii)]{bauschke2011convex} for a proof.
\begin{theorem}[Descent theorem]\label{thm:descent}
Suppose that $f : \cH \rightarrow (-\infty, \infty]$ is closed, convex, and differentiable.  If $\nabla f$ is $({1}/{\beta_f})$-Lipschitz, then for all $x, y \in \dom(f)$, we have the upper bound
\begin{align}\label{eq:lipschitzderivative}
f(x) &\leq f(y) + \dotp{ x- y, \nabla f(y)} + \frac{1}{2\beta} \|x - y\|^2.
\end{align}
\end{theorem}

\begin{proposition}[Lipschitz differentiable upper bound]\label{prop:Lipschitz}
Suppose that $(z^j)_{j \geq 0}$ is generated by Equation~\eqref{eq:zitr}. Then the following bounds hold:
\label{thm:lipschitz:part:f} Suppose that $f$ is differentiable and $\nabla f$ is $(1/\beta_f)$-Lipschitz. Then
\begin{align*}
&2\gamma\lambda((f+g+ h)(x_g) - (f+g + h)(x^\ast)) \\
&\leq \begin{cases}
\|z - z^\ast\|^2 - \|z^+ - z^\ast\|^2 +   \left(1 + \frac{\gamma - \beta_f}{\beta_f\lambda}\right)\|z-z^+\|^2 \\
+ 2\gamma \dotp{ \nabla h(x_g) - \nabla h(x^\ast), z - z^{+}} & \text{if } \gamma \leq \beta_f \\
\left(1 + \frac{\gamma - \beta_f}{2\beta_f}\right) (\|z - z^\ast\|^2 - \|z^+ - z^\ast\|^2 +   \|z-z^+\|^2) &  \\
+ 2\gamma\left(1 + \frac{\gamma - \beta_f}{2\beta_f}\right)\dotp{ \nabla h(x_g) - \nabla h(x^\ast), z - z^{+}} & \text{if } \gamma > \beta_f.
\end{cases} \numberthis \label{eq:lipshitzfundamentalinequalityf}
\end{align*}
\end{proposition}
\begin{proof}
Because $\nabla f$ is ($1/\beta_f$)-Lipschitz, we have
\begin{align}
f(x_g) &\stackrel{\eqref{eq:lipschitzderivative}}{\leq} f(x_f) + \dotp{x_g - x_f, \nabla f(x_f)} + \frac{1}{2\beta_f}\|x_g - x_f\|^2; \label{eq:descenttheorem2}\\
S_f(x_f, x^\ast) &\stackrel{\eqref{eq:Sfunction}}{\geq} \frac{\beta_f}{2} \|\nabla f(x_f) - \nabla f(x^\ast)\|^2\label{eq:Slowerbound}.
\end{align}
By applying the identity $z^\ast - x^\ast = \gamma \tnabla g(x^\ast) =- \gamma \nabla f(x^\ast) - \gamma\nabla h(x^\ast)$, the cosine rule~\eqref{eq:cosinerule}, and the identity $z - z^+ = \lambda(x_g - x_f)$ (see Lemma~\ref{lem:identities}) multiple times, we have
\begin{align*}
&2\dotp{z - z^+, z^\ast - x^\ast} + 2\gamma \lambda\dotp{ x_g - x_f, \nabla f(x_f)}  \\
&= 2\lambda\dotp{x_g - x_f, \gamma \tnabla g(x^\ast) + \gamma\nabla f(x_f)} \\
&= 2\lambda \dotp{\gamma \tnabla g(x_g) + \gamma \nabla h(x_g) + \gamma \nabla f(x_f) , \gamma \nabla f(x_f) - \gamma \nabla f(x^\ast)} - 2\dotp{z-z^+, \gamma \nabla h(x^\ast)} \\
&= \lambda \biggl(\|\gamma\nabla f(x_f) - \gamma \nabla f(x^\ast)\|^2 + \|x_g - x_f\|^2 \\
&- \|\gamma\tnabla g(x_g) + \gamma\nabla h(x_g) - \gamma \tnabla g(x^\ast) - \gamma \nabla h(x^\ast)\|^2\biggr) - 2\dotp{z-z^+, \gamma \nabla h(x^\ast)}. \numberthis \label{eq:preboundLipschitz}
\end{align*}
By Lemma~\ref{lem:identities} (i.e., $z - z^+ = \lambda(x_g - x_f)$), we have
\begin{align*}
\left(1 - \frac{2}{\lambda}\right) \|z - z^+\|^2 + \lambda \left(\frac{\gamma}{\beta_f} + 1\right) \|x_g - x_f\|^2 &= \left(1 + \frac{(\gamma - \beta_f)}{\beta_f\lambda}\right)\|z-z^+\|^2.
\end{align*}
Therefore,
\begin{align*}
&2\gamma \lambda ((f + g + h)(x_g) - (f + g + h)(x^\ast)) \\
&\stackrel{\eqref{eq:descenttheorem2}}{\leq} 2\gamma \lambda(f(x_f) + g(x_g) +  h(x_g) - (f+ g+ h)(x^\ast)) +  2\gamma \lambda\dotp{ x_g - x_f, \nabla f(x_f)} + \frac{\gamma\lambda}{\beta_f}\|x_g - x_f\|^2 \\
&\stackrel{\eqref{eq:functionvaluefundamentalinequality}}{\leq} \|z - z^\ast\|^2 - \|z^+ - z^\ast\|^2 + 2\dotp{z - z^+, z^\ast - x^\ast} + 2\gamma \lambda\dotp{ x_g - x_f, \nabla f(x_f)}\\
&+ \left(1-\frac{2}{\lambda}\right) \| z^+ - z\|^2 + 2\gamma \dotp{ \nabla h(x_g), z - z^+} + \frac{\gamma\lambda}{\beta_f}\|x_g - x_f\|^2 - 2\gamma \lambda S_f(x_f, x^\ast)  \\
&\stackrel{\eqref{eq:preboundLipschitz}}{\leq} \|z - z^\ast\|^2 - \|z^+ - z^\ast\|^2 +\left(1 - \frac{2}{\lambda}\right) \|z - z^+\|^2 + \lambda \left(\frac{\gamma}{\beta_f} + 1\right) \|x_g - x_f\|^2 \\
&+ \lambda \|\gamma\nabla f(x_f) - \gamma\nabla f(x^\ast)\|^2 + 2\gamma \dotp{ \nabla h(x_g) - \nabla h(x^\ast), z - z^+}- 2\gamma \lambda S_f(x_f, x^\ast) \\
&\leq\|z - z^\ast\|^2 - \|z^+ - z^\ast\|^2 +   \left(1 + \frac{(\gamma - \beta_f)}{\beta_f\lambda}\right)\|z-z^+\|^2 \\
&\stackrel{\eqref{eq:Slowerbound}}{+} 2\gamma \dotp{ \nabla h(x_g) - \nabla h(x^\ast), z - z^+} +  \gamma \lambda (\gamma - \beta_f)\|\nabla f(x_f) - \nabla f(x^\ast)\|^2. \numberthis\label{eq:lipschitzdiffbound}
\end{align*}
If $\gamma \leq \beta_f$, then we can drop the last term.  If $\gamma > \beta_f$, then we apply the upper bound in Equation~\eqref{eq:functionlinearconvergenceinequality} to get:
\begin{align*}
&\gamma \lambda (\gamma - \beta_f)\|\nabla f(x_f) - \nabla f(x^\ast)\|^2 \\
&\leq \frac{ (\gamma - \beta_f)}{2\beta_f} \biggl(\|z - z^\ast\|^2 - \|z^+ - z^\ast\|^2 + \left(1-\frac{2}{\lambda}\right) \| z^+ - z\|^2 \\
&+ 2\gamma \dotp{ \nabla h(x_g) - \nabla h(x^\ast), z - z^{+}}\biggr).
\end{align*}
The result follows by using the above inequality in Equation~\eqref{eq:lipschitzdiffbound} together with the following identity: $$\left(1 + \frac{(\gamma - \beta_f)}{\beta_f\lambda}\right)\|z-z^+\|^2 + \frac{ (\gamma - \beta_f)}{2\beta_f} \left(1-\frac{2}{\lambda}\right) \| z - z^+\|^2 = \left(1 + \frac{\gamma - \beta_f}{2\beta_f}\right)\|z - z^+\|^2.$$
\qed\end{proof}

\begin{theorem}[``Best" objective error rate]\label{thm:lipschitzderivative}
Let $(z^j)_{j \geq 0}$ be generated by Equation~\eqref{eq:zitr} with $\gamma \in (0, 2\beta)$ and $\underline{\tau} = \inf_{j \geq 0} \lambda_j(1-\alpha\lambda_j)/\alpha > 0$. Then the following bound holds:
 If $f$ is differentiable and $\nabla f$ is $(1/\beta_f)$-Lipschitz, then
\begin{align*}
0 \leq \min_{i=0, \cdots, k}\left\{(f + g + h)(x_g^i) - (f+ g+ h)(x^\ast)\right\} &= o\left(\frac{1}{k+1}\right).
\end{align*}
\end{theorem}
\begin{proof}
By~\cite[Part 4 of Lemma 3]{davis2014convergence} It suffices to show that all of the upper bounds in Proposition~\ref{prop:Lipschitz} are summable. In both of the cases, the alternating sequence (and any constant multiple) $(\|z^j - z^\ast\|^2 - \|z^{j+1} - z^\ast\|^2)_{j \geq 0}$ is clearly summable. In addition, we know that $(\|z^j - z^{j+1}\|^2)_{j \geq 0}$ is summable by Part~\ref{thm:convergence:part:biginequality} of Theorem~\ref{thm:convergence}, and every coefficient of this sequence in the two upper bounds is bounded (because $(\lambda_j)_{j \geq 0}$ is a bounded sequence). Thus, the part pertaining to $(\|z^j - z^{j+1}\|^2)_{j \geq 0}$ is summable.

Finally, we just need to show that $(\dotp{\nabla h(x_g^j) - \nabla h(x^\ast), z^j - z^{j+1}})_{j \geq 0}$ is summable.  The Cauchy-Schwarz inequality and Young's inequality for real numbers show that for all $k \geq 0$, we have
\begin{align*}
2\dotp{\nabla h(x_g^k) - \nabla h(x^\ast), z^k - z^{k+1}}&\leq \|\nabla h(x_g^k) - \nabla h(x^\ast)\|^2 + \|z^k - z^{k+1}\|^2.
\end{align*}
The second term is summable by the argument above, and the first term is summable by Part~\ref{thm:convergence:part:gradientsum} of Theorem~\ref{thm:convergence}.
\qed\end{proof}

\begin{remark}
The order of convergence in Theorem~\ref{thm:lipschitzderivative} is sharp~\cite[Theorem 12]{davis2014convergence}, and generalizes similar results known for Douglas-Rachford splitting, forward-backward splitting and forward-Douglas-Rachford splitting~\cite{davis2014convergence,davis2014convergenceFDRS,davis2014convergenceFaster}.
\end{remark}
\subsection{Linear convergence}\label{sec:linear}

In this section we show that
$$\text{Equation~\eqref{eq:zitr} converges linearly whenever $(\mu_A + \mu_B + \mu_C)(1/L_A + 1/L_B) > 0$}$$
where $L_A$ and $L_B$ are the Lipschitz constants of $A$ and $B$ and we follow the convention that $1/L_A = 0$ or $1/L_B = 0$ whenever $A$ or $B$ fail to be Lipschitz, respectively.

The first result of this section is an inequality that will help us deduce contraction factors for $T$ in Theorem~\ref{thm:linear}.

\begin{proposition}
Assume the setting of Theorem~\ref{thm:convergence}. In particular, let $\varepsilon \in (0, 1)$, let $\gamma \in (0, 2\beta \varepsilon)$, let $\alpha = 1/(2-\varepsilon)$, and let $\lambda \in (0, 1/\alpha)$. Let $z \in \cH$ and let $z^+ = (1-\lambda)z + \lambda Tz$. Let $z^\ast$ be a fixed point of $T$ and let $x^\ast = J_{\gamma B}(z^\ast)$. Let $x_A$ and $x_B$ be defined as in Lemma~\ref{lem:identities}. Let $Q_A, Q_B$ and $Q_C$ be defined as in Proposition~\ref{prop:regularlowerbound}. Then the following inequality holds:
\begin{align*}
&\|z^+ - z^\ast\|^2 + \left( \frac{1}{\lambda \alpha} - 1\right)\|z - z^+\|^2 + 2 \gamma\lambda Q_A(x_A, x^\ast) + 2\gamma\lambda Q_B(x_B, x^\ast) + 2\gamma \lambda Q_C(x_B, x^\ast) \\
&- \frac{\gamma^2 \lambda}{\varepsilon}\|Cx_B - Cx^\ast\|^2  \\
&\leq\|z - z^\ast\|^2 \\
&\leq \min \biggl\{\left(1+ \gamma L_B\right)^2\|x_B - x^\ast\|^2, 3\left(\left(1+\gamma L_A\right)^2\|x_A - x^\ast\|^2 + \gamma^2\|Cx_B - Cx^\ast\|^2 + 4\|x_B - x_A\|^2\right), \\
&  3(1+2\gamma^2L_B^2)\left(\|x_A - x^\ast\|^2 + \|x_A - x_B\|^2\right), \\
&4\left(\left(1+ 2\gamma^2L_A^2\right)\|x_B - x^\ast\|^2 + \gamma^2\|Cx_B - Cx^\ast\|^2 + \left(1+2\gamma^2L_A^2\right)\|x_B - x_A\|^2\right) \biggr\}. \numberthis \label{eq:linearconvergenceinequality2}
\end{align*}
\end{proposition}
\begin{proof}
Equation~\eqref{eq:linearconvergenceinequality2} shows that:
\begin{align*}
\|z^+ - z^\ast\|^2 &+ \left( \frac{2}{\lambda} - 1\right)\|z - z^+\|^2 + 2 \gamma\lambda Q_A(x_A, x^\ast) + 2\gamma\lambda Q_B(x_B, x^\ast) + 2\gamma \lambda Q_C(x_B, x^\ast)   \\
&\leq \|z - z^\ast\|^2 + 2\gamma \dotp{z - z^+, Cx_B  -  C x^\ast}.
\end{align*}
From Cauchy-Schwarz and Young's inequality, we have
\begin{align*}
2\gamma \dotp{z - z^+, Cx_B  -  C x^\ast} &\leq \frac{\varepsilon}{\lambda}\|z - z^+\|^2 +   \frac{\gamma^2\lambda}{\varepsilon}\|Cx_B - Cx^\ast\|^2.
\end{align*}
The lower bound now follows by rearranging.

The upper bound follows from the following bounds (where we take $L_B = \infty$ or $L_A= \infty$ respectively whenever $A$ or $B$ fail to be Lipschitz):
\begin{align*}
\|z - z^\ast\|^2 &= \|x_B + \gamma u_B - (x^\ast + \gamma u_B^\ast)\|^2 \leq \left(1+ \gamma L_B\right)^2\|x_B - x^\ast\|^2; \\
\|z - z^\ast\|^2 &= \|x_B +  \gamma u_B - (x^\ast - \gamma u_A^\ast - \gamma Cx^\ast)\|^2 \\
&= \|x_B - \gamma( u_A + Cx_B) + \gamma (u_B + u_A + Cx_B) - (x^\ast - \gamma u_A^\ast - \gamma Cx^\ast)\|^2 \\
&\leq  \left\|x_A - \gamma( u_A + Cx_B) + 2(x_B - x_A) - (x^\ast - \gamma u_A^\ast - \gamma Cx^\ast)\right\|^2 \\
&\leq 3\left(\|x_A - \gamma u_A - (x^\ast - \gamma u_A^\ast)\|^2 + \gamma^2\|Cx_B - Cx^\ast\|^2 + 4\|x_B - x_A\|^2\right) \\
&\leq 3\left(\left(1+\gamma L_A\right)^2\|x_A - x^\ast\|^2 + \gamma^2\|Cx_B - Cx^\ast\|^2 + 4\|x_B - x_A\|^2\right); \\
\|z - z^\ast\|^2 &= \|x_B + \gamma u_B - (x^\ast + \gamma u_B^\ast)\|^2 \\
&= \|x_A + \gamma u_B -  (x^\ast + \gamma u_B) + (x_B - x_A)\|^2 \\
&\leq 3\left(\|x_A - x^\ast\|^2 + \gamma^2 \|u_B - u_B^\ast\|^2 + \|x_A - x_B\|^2\right) \\
&\leq 3\left(\|x_A - x^\ast\|^2 + \gamma^2L_B^2 \|x_B - x^\ast\|^2 + \|x_A - x_B\|^2\right) \\
&\leq 3(1+2\gamma^2L_B^2)\left(\|x_A - x^\ast\|^2 + \|x_A - x_B\|^2\right); \\
\|z - z^\ast\|^2 &= \|x_B + \gamma u_B - (x^\ast + \gamma u_B^\ast)\|^2 \\
&=\|x_B - \gamma (u_A+ Cx_B) - (x^\ast - \gamma (u_A^\ast + Cx^\ast)) + (x_B - x_A)\|^2 \\
&\leq 4\left(\|x_B - x^\ast\|^2 + \gamma^2\|u_A - u_A^\ast\|^2 + \gamma^2\|Cx_B - Cx^\ast\|^2 + \|x_B - x_A\|^2\right) \\
&\leq 4\left(\|x_B - x^\ast\|^2 + \gamma^2 L_A^2\|x_A - x^\ast\|^2 + \gamma^2\|Cx_B - Cx^\ast\|^2 + \|x_B - x_A\|^2\right) \\
&\leq4\left(\left(1+ 2\gamma^2L_A^2\right)\|x_B - x^\ast\|^2 + \gamma^2\|Cx_B - Cx^\ast\|^2 + \left(1+2\gamma^2L_A^2\right)\|x_B - x_A\|^2\right).
\end{align*}
\qed\end{proof}

The following theorem proves linear convergence of Equation~\eqref{eq:zitr} whenever $(\mu_A + \mu_B + \mu_C)(1/L_A + 1/L_B) > 0$.

\begin{theorem}\label{thm:linear}
Assume the setting of Theorem~\ref{thm:convergence}. In particular, let $\varepsilon \in (0, 1)$, let $\gamma \in (0, 2\beta \varepsilon)$, let $\alpha = 1/(2-\varepsilon)$, and let $\lambda \in (0, 1/\alpha)$. Let $z \in \cH$ and let $z^+ = (1-\lambda)z + \lambda Tz$. Let $z^\ast$ be a fixed point of $T$ and let $x^\ast = J_{\gamma B}(z^\ast)$.  Then the following inequality holds under each of the conditions below:
\begin{align*}
\|z^+ - z^\ast \| & \leq \left(1 - C(\lambda)\right)^{1/2}\|z- z^\ast\|
\end{align*}
where $C(\lambda) \in [0, 1]$ is defined below under different scenarios.
\begin{enumerate}
\item \label{thm:linear:part:BLBM} Suppose that $B$ is $L_B$-Lipschitz, and $\mu_B$ strongly monotone. Then
\begin{align*}
C(\lambda) &=  \frac{2L_B\gamma \lambda}{\left(1 + \gamma L_B\right)^2}.
\end{align*}
\item \label{thm:linear:part:ALAM} Suppose that $A$ is $L_A$-Lipschitz and $\mu_A$-strongly monotone. Then
\begin{align*}
C(\lambda) &= \frac{\lambda}{3} \min\left\{\frac{2\mu_A\gamma}{(1+\gamma L_A)^2}, \frac{\lambda}{4}\left(\frac{1}{\alpha\lambda} - 1\right), \frac{2\beta - \gamma/\varepsilon}{\gamma}\right\}.
\end{align*}
\item  \label{thm:linear:part:BLAM} Suppose that $A$ is $\mu_A$ strongly monotone and $B$ is $L_B$-Lipschitz. Then
\begin{align*}
C(\lambda) &= \frac{\lambda}{3(1 + 2\gamma^2L_B^2)}\min\left\{2\gamma \mu_A, \lambda\left(\frac{1}{\alpha\lambda} - 1\right)\right\}
\end{align*}
\item \label{thm:linear:part:ALBM} Suppose that $A$ is $L_A$-Lipschitz and $B$ is $\mu_B$-strongly monotone.  Then
\begin{align*}
C(\lambda) &= \frac{\lambda}{4}\min\left\{\frac{2\gamma\mu_B}{(1+2\gamma L_A^2)},\frac{2\beta - \gamma/\varepsilon}{\gamma}, \frac{\lambda}{(1+ 2\gamma^2L_A^2)}\left(\frac{1}{\alpha\lambda} - 1\right) \right\}
\end{align*}
\item \label{thm:linear:part:ALCM} Suppose that $A$ is $L_A$-Lipschitz and $C$ is $\mu_C$-strongly monotone. Let $\eta \in (0, 1)$ be large enough that $2\eta\beta > \gamma/\varepsilon$. Then
\begin{align*}
C(\lambda) &= \frac{\lambda}{4}\min\left\{\frac{2\gamma\mu_C(1-\eta)}{(1+2\gamma^2 L_A^2)},\frac{2\beta - \gamma/\varepsilon}{\gamma}, \frac{\lambda}{(1+ 2\gamma^2L_A^2)}\left(\frac{1}{\alpha\lambda} - 1\right) \right\}
\end{align*}
\item \label{thm:linear:part:BLCM} Suppose that $B$ is $L_B$-Lipschitz and $C$ is $\mu_C$-strongly monotone. Let $\eta \in (0, 1)$ be large enough that $2\eta\beta > \gamma/\varepsilon$. Then
\begin{align*}
C(\lambda) &= \frac{2\gamma\mu_C(1-\eta)}{(1+ \gamma L_B)^2}
\end{align*}
\end{enumerate}
\end{theorem}
\begin{proof}
Each part of the proof is based on the following idea: If $a_0, \cdots, a_n, b_0, \cdots, b_n, c_0, \cdots, c_n \in \vR_{++}$ for some $n \geq 0$, and
\begin{align*}
\|z^+ - z^\ast\|^2 + \sum_{i=0}^n a_ic_i &\leq \|z - z^\ast\|^2 \leq \sum_{i=0}^n a_ib_i,
\end{align*}
then $\sum_{i=0}^n a_ib_i \leq \max\{b_i/c_i \mid i = 0, \cdots, n\} \sum_{i=0}^n a_ic_i$, so
\begin{align*}
\|z^+ - z^\ast\|^2 + \min\{c_i/b_i \mid i = 0, \cdots, n\} \|z - z^\ast\|^2 \leq \|z^+ - z^\ast\|^2 + \sum_{i=0}^n a_ic_i &\leq \|z - z^\ast\|^2 \leq \|z - z^\ast\|^2.
\end{align*}
Thus,
\begin{align*}
\|z^+ - z^\ast\| \leq \left(1 - \min\{c_i/b_i \mid i = 0, \cdots, n\}\right)^{1/2}\|z - z^\ast\|.
\end{align*}
In each case the terms $a_ic_i$ will be taken from the left hand side of Equation~\eqref{eq:linearconvergenceinequality2}, and the terms $a_ib_i$ will be taken from the right of the same equation.

Part~\ref{thm:linear:part:BLBM}: We use the first upper bound in Equation~\eqref{eq:linearconvergenceinequality2} and set $a_0 = \|x_B - x^\ast\|^2, c_0 = 2\gamma \lambda \mu_A$, and $b_0 = (1+\gamma L_B)^2$.

Part~\ref{thm:linear:part:ALAM}: We use the second upper bound in Equation~\eqref{eq:linearconvergenceinequality2} and set $a_0 = \|x_A - x^\ast\|^2, c_0 = 2\gamma \lambda \mu_A, b_0 = 3(1+\gamma L_A)^2$, $a_1 = \|Cx_B - Cx^\ast\|^2, c_1 = \gamma \lambda( 2\beta - \gamma/\varepsilon), b_1 = 3\gamma^2$,  $a_2 = \|x_B - x_A\|, c_2 = \lambda^2(1/(\lambda \alpha) - 1),$ and $b_2 = 12$.

Part~\ref{thm:linear:part:BLAM}: We use the third upper bound in Equation~\eqref{eq:linearconvergenceinequality2} and set $a_0 = \|x_A - x^\ast\|^2, c_0 = 2\gamma\lambda \mu_A, b_0 = 3(1+\gamma L_A)^2$,  $a_1 = \|x_A - x_B\|^2, c_1 = \lambda^2(1/(\lambda \alpha) - 1),$ and $b_1 = 3(1+2\gamma^2L_B^2)$.

Part~\ref{thm:linear:part:ALBM}: We use the fourth upper bound in Equation~\eqref{eq:linearconvergenceinequality2} and set $a_0 = \|x_B - x^\ast\|^2, c_0 = 2\gamma\lambda\mu_B , b_0 = 4(1+2\gamma^2L_A^2), a_1 = \|Cx_B - Cx^\ast\|^2, c_1 = \gamma \lambda( 2\beta - \gamma/\varepsilon), b_1 = 4\gamma^2, a_2 = \|x_B - x_A\|^2, c_2 = \lambda^2(1/(\lambda \alpha) - 1), b_2 = 4(1+2\gamma^2L_A^2)$.

Part~\ref{thm:linear:part:ALCM}: We use the fourth upper bound in Equation~\eqref{eq:linearconvergenceinequality2} and set $a_0 = \|x_B - x^\ast\|^2, c_0 = 2\gamma\lambda\mu_C(1-\eta) , b_0 = 4(1+2\gamma^2L_A^2), a_1 = \|Cx_B - Cx^\ast\|^2, c_1 = \gamma \lambda( 2\eta\beta - \gamma/\varepsilon), b_1 = 4\gamma^2, a_2 = \|x_B - x_A\|^2, c_2 = \lambda^2(1/(\lambda \alpha) - 1), b_2 = 4(1+2\gamma^2L_A^2)$.

Part~\ref{thm:linear:part:BLCM}: We use the first upper bound in Equation~\eqref{eq:linearconvergenceinequality2} and set $a_0 = \|x_B - x^\ast\|^2, c_0 = 2\gamma \lambda \mu_C(1-\eta)$, and $b_0 = (1+\gamma L_B)^2$.
\qed\end{proof}

\begin{remark}
Note that the contraction factors can be improved whenever $A$ or $B$ are known to be subdifferential operators of convex functions because the function $Q_\cdot(\cdot, \cdot)$ can be made larger with Proposition~\ref{prop:regularlowerbound}. We do not pursue this here due to lack of space.
\end{remark}

\begin{remark}
Note that we can relax the conditions of Theorem~\ref{thm:linear}. Indeed, we only need to assume that $C$ is Lipschitz to derive linear convergence, not necessarily cocoercive.  We do not pursue this extension here due to lack of space.
\end{remark}

\subsection{Arbitrarily slow convergence when $\mu_C\mu_A > 0$.}\label{sec:slowconvergence}

This section shows that the result of Theorem~\ref{thm:linear} cannot be improved in the sense that we cannot expect linear convergence even if $C$ and $A$ are strongly monotone. The results of this section parallel similar results shown in~\cite[Section 6.1]{davis2014convergenceFDRS}.

\subsubsection*{The main example}
Let $\cH = \ell_2^2(\vN) = \vR^2 \oplus \vR^2\oplus \cdots$. Let $R_{\theta}$ denote counterclockwise rotation in $\vR^2$ by $\theta$ degrees.  Let $e_0 := (1, 0)$ denote the standard unit vector, and let $e_{\theta} := R_\theta e_0$.  Suppose that $(\theta_j)_{j\geq0}$ is a sequence of angles in $(0, {\pi}/{2}]$ such that $\theta_i \rightarrow 0$ as $i \rightarrow \infty$.  For all $i \geq 0$, let $c_i := \cos(\theta_i)$. We let
\begin{align}
V := \vR^2e_0 \oplus \vR^2 e_0 \oplus  && \mathrm{and} && U := \vR^2e_{\theta_0} \oplus \vR^2e_{\theta_1} \oplus \cdots.
\end{align}
Note that \cite[Section 7]{bauschke2013rate} proves the projection identities
\begin{align*}
(P_U)_i &= \begin{bmatrix} \cos^2(\theta_i) & \sin(\theta_i)\cos(\theta_i) \\ \sin(\theta_i)\cos(\theta_i) & \sin^2(\theta_i) \end{bmatrix} && \mathrm{and} && (P_V)_i = \begin{bmatrix} 1 & 0 \\ 0 & 0\end{bmatrix},
\end{align*}
\cut{In what follows we will substitute $\sin^2(\theta_i) = 1 - c_i^2$, etc., whenever convenient.}

We now begin our extension of this example. Choose $a \geq 0$ and set $f = \iota_U + ({a}/{2})\|\cdot\|^2$, $g = \iota_V$, and $h = ({1}/{2})\|\cdot\|^2.$ Set $A = \partial f, B = \partial g$ and $C = \nabla h$. Note that $\mu_h = 1$ and $\mu_f = a$.  Thus, $\nabla h$ is $1$-Lipschitz, and, hence, $\beta = 1$ and we can choose $\gamma = 1 < 2\beta$. Therefore, $\alpha = 2\beta/(4\beta - \gamma) = 2/3$, so we can choose $\lambda_k \equiv 1 < 1/\alpha$. We also note that $\prox_{\gamma f} = (1/(1+a)) P_U$.

For all $i \geq 0$, we have
\begin{align*}
T_i &:=  \frac{1}{a+1}(P_U)_i(2(P_V)_i - I_{\vR^2} - I_{\vR^2}) + I_{\vR^2} - (P_V)_i \\
&= \frac{1}{a+1}(P_U)_i\begin{bmatrix} 0 & 0 \\ 0 & -2\end{bmatrix} + \begin{bmatrix} 0 & 0 \\ 0 & 1\end{bmatrix} \\
&= \frac{1}{a+1}\begin{bmatrix}0 & -2\sin(\theta_i)\cos(\theta_i) \\  0 &  -2\sin^2(\theta_i) + a + 1\end{bmatrix}
\end{align*}
where $T = \bigoplus_{i=0}^\infty T_i$ is the operator defined in Equation~\eqref{eq:newoperator}. Note that for all $i \geq 0$, the operator $(T)_i$ has eigenvector
\begin{align*}
z_i = \left(-\frac{2\cos(\theta_i)\sin(\theta_i)}{1 + a - 2\sin^2(\theta_i)}, 1\right)
\end{align*}
with eigenvalue $b_i := (a - 2(1-c_i)^2 + 1)/(a+1)$. Each component also has the eigenvector $(1, 0)$ with eigenvalue $0$. Thus, the only fixed point of $T$ is $0 \in \cH$.  Finally, we note that
\begin{align}\label{eq:eigenvectornorm}
\|z_i\|^2 = \frac{4c_i^2(1-c_i^2)}{(1 + a -2 (1-c_i)^2)^2} + 1. 
\end{align}

\subsubsection*{Slow convergence proofs}
Part~\ref{thm:convergence:part:FPRmonotone} of Theorem~\ref{thm:convergence} shows that $z^{k+1} - z^k \rightarrow 0$. The following result is a consequence of \cite[Proposition 5.27]{bauschke2011convex}.
\begin{lemma}[Strong convergence]
Any sequence $(z^j)_{j \geq 0}\subseteq \cH$ generated by Algortihm~\ref{alg:basic} converges strongly to $0$.
\end{lemma}

The next Lemma appeared in~\cite[Lemma 6]{davis2014convergence}.
\begin{lemma}[Arbitrarily slow sequence convergence]\label{lem:slowconvergencesequence}
Suppose that $F : \vR_+ \rightarrow (0, 1)$ is a function that is monotonically decreasing to zero.  Then there exists a monotonic sequence $(b_j)_{j \geq 0} \subseteq (0, 1)$ such that $b_k \rightarrow 1^-$ as $k \rightarrow \infty$  and an increasing sequence of integers $(n_j)_{j \geq 0} \subseteq \vN \cup \{0\}$ such that for all $k \geq 0$,
\begin{align}\label{eq:slowconvergence}
\frac{b_{n_k}^{k+1}}{n_k+1} > F(k+1)e^{-1}.
\end{align}
\end{lemma}

The following is a simple corollary of Lemma~\ref{lem:slowconvergencesequence}; The lemma first appeared in~\cite[Section 6.1]{davis2014convergenceFDRS}.
\begin{corollary}\label{cor:slowconvergencesequence}
Let the notation be as in Lemma~\ref{lem:slowconvergencesequence}. Then for all $\eta \in (0, 1)$, we can find a sequence $(b_j)_{j \geq 0} \subseteq (\eta, 1)$ that satisfies the conditions of the lemma.
\end{corollary}

We are now ready to show that FDRS can converge arbitrarily slowly.
\begin{theorem}[Arbitrarily slow convergence of~\eqref{eq:zitr}]\label{thm:arbitrarilyslow}
For every function $F : \vR_+ \rightarrow (0, 1)$ that strictly decreases to zero, there is a point $z^0 \in \ell_2^2(\vN)$ and two closed subspaces $U$ and $V$ with zero intersection, $U\cap V = \{0\}$, such that sequence $(z^j)_{j \geq 0}$ generated by Equation~\eqref{eq:zitr} applied to the functions $f = \chi_{U} + (a/2)\|\cdot\|^2$ and $g = (1/2)\|\cdot\|^2$, relaxation parameters $\lambda_k \equiv 1$, and stepsize $\gamma = 1$ satisfies the following bound:
\begin{align*}
  \|z^k - z^\ast\| \geq e^{-1} F(k),
\end{align*}
but $(\|z^j - z^\ast\|)_{j \geq 0}$ converges to $0$.
\end{theorem}
\begin{proof}
For all $i  \geq 0$, define $z_i^0 = (1/\|z_i\|(i+1))z_i$, then $\|z_i^0\| = 1/(i+1)$ and $z_i^0$ is an eigenvector of $(T)_i$ with eigenvalue $b_i = (a - 2(1-c_i)^2 + 1)/(a+1)$. Define the concatenated vector $z^0 = (z_i^0)_{i \geq 0}$. Note that $z^0 \in \cH$ because $\|z^0\|^2 = \sum_{i=0}^\infty 1/(i+1)^2 < \infty$. Thus, for all $k \geq 0$, we let $z^{k+1} = T z^k$.

Now, recall that $z^\ast= 0$. Thus, for all $n \geq 0$ and $k \geq 0$, we have
\begin{align*}
\|z^k - z^\ast\|^2 = \|T^k z^0 \|^2 = \sum_{i=0}^\infty b_i^{2(k+1)}\|z_i^0\|^2 = \sum_{i=0}^\infty \frac{b_i^{2(k+1)}}{(i+1)^2} \geq \frac{b_n^{2(k+1)}}{(n+1)^2}.
\end{align*}
Thus, $\|z^k - z^\ast\| \geq  b_n^{(k+1)}/(n+1)$.  To get the lower bound, we choose $b_n$ and the sequence $(n_j)_{j \geq 0}$ using Corollary~\ref{cor:slowconvergencesequence} with any $\eta \in (\max\{0, (a-1)/(a+1)\}, 1)$. Then we solve for the coefficients: $c_n = 1 - \sqrt{(a+1)(1-b_n)/2} > 0.$
%
\qed\end{proof}

\begin{remark}
Theorems~\ref{thm:arbitrarilyslow} and~\ref{thm:strong} show that the sequence $(z^j)_{j \geq 0}$ can converge \emph{arbitrarily slowly} even if $(x_f^j)_{j \geq 0}$ and $(x_h^j)_{j \geq 0}$ converge with rate $o(1/\sqrt{k+1})$.
\end{remark}